\documentclass{amsart}

\usepackage{amsmath,amssymb,amsthm}

\newtheorem{theorem}{Theorem}[section]
\newtheorem*{theorem0}{Theorem}
\newtheorem{lemma}[theorem]{Lemma}
\newtheorem{proposition}[theorem]{Proposition}
\newtheorem{corollary}[theorem]{Corollary}

\theoremstyle{definition}
\newtheorem{definition}[theorem]{Definition}
\newtheorem{remark}[theorem]{Remark}
\newtheorem{example}[theorem]{Example}

\newtheorem{notation}[theorem]{Notation}
\numberwithin{equation}{section}

\usepackage[all]{xy}
\usepackage{graphics}
\usepackage{mathrsfs}

\newcommand{\K}{\mathbb{K}}
\newcommand{\Gal}{\mathsf{Gal}}

\newcommand{\id}{\mathsf{id}}
\newcommand{\Ker}{\mathsf{Ker}}
\newcommand{\Aut}{\mathsf{Aut}}

\newcommand{\End}{\mathsf{End}}
\newcommand{\Hom}{\mathsf{Hom}}
\newcommand{\ad}{\mathop{\mathsf{ad}}}
\renewcommand{\exp}{\mathop{\mathsf{exp}}}

\newcommand{\g}{\mathfrak{g}}
\newcommand{\h}{\mathfrak{h}}
\newcommand{\gsigma}{\g^\sigma}
\newcommand{\hsigma}{\h^\sigma}

\newcommand{\gcc}{\g}
\newcommand{\hcc}{\h}
\newcommand{\gsigmacc}{\gcc^\sigma}
\newcommand{\hsigmacc}{\hcc^\sigma}
\newcommand{\hsigmascc}{(\hcc^\sigma)^*}
\newcommand{\elln}{\ell}

\newcommand{\XX}{\mathsf{X}}
\newcommand{\XA}{\mathsf{A}}
\newcommand{\XB}{\mathsf{B}}
\newcommand{\XC}{\mathsf{C}}
\newcommand{\XD}{\mathsf{D}}
\newcommand{\XE}{\mathsf{E}}
\newcommand{\XF}{\mathsf{F}}
\newcommand{\XG}{\mathsf{G}}
\newcommand{\XNsigma}{\XX_N^\sigma}
\newcommand{\Pis}{\Pi^\sigma}
\newcommand{\Psis}{\Psi^\Gamma}
\newcommand{\J}{I}
\newcommand{\Jsigma}{\J^\sigma}
\newcommand{\Jsigmah}{\hat{\J}^\sigma}

\newcommand{\Deltap}{\Delta_+}
\newcommand{\Deltah}{\hat{\Delta}}
\newcommand{\Deltar}{{}^{\mathsf{re}}\Delta}
\newcommand{\Deltahr}{{}^{\mathsf{re}}\Deltah}
\newcommand{\Deltas}{\Delta^\sigma}
\newcommand{\Deltasp}{\Deltas_+}
\newcommand{\Deltasm}{\Deltas_-}
\newcommand{\Deltaslong}{\Deltas_{\mathsf{long}}}
\newcommand{\Deltasshort}{\Deltas_{\mathsf{short}}}
\newcommand{\Deltash}{\Deltah^\sigma}
\newcommand{\Deltashr}{{}^{\mathsf{re}}\Deltash}
\newcommand{\Deltashrid}{{}^{\mathsf{re}}\Deltah^{\id}}

\newcommand{\As}{A^\sigma}
\newcommand{\Ash}{\hat{A}^\sigma}

\newcommand{\sigmah}{\hat\sigma}
\newcommand{\omegah}{\hat\omega}
\newcommand{\dualroot}[1]{t_{#1}}
\newcommand{\coroot}[1]{{{#1}^\vee}}

\newcommand{\Kqq}{\mathbb{Q}}
\newcommand{\gqq}{\g_\Kqq}
\newcommand{\hqq}{\h_\Kqq}
\newcommand{\gsigmaqq}{\gsigma_\Kqq}
\newcommand{\hsigmaqq}{\hsigma_\Kqq}
\newcommand{\hsigmasqq}{(\hsigmaqq)^*}
\newcommand{\Rqq}{R_\Kqq}
\newcommand{\Sqq}{S_\Kqq}
\newcommand{\otimesqq}{\otimes_{\Kqq}}
\newcommand{\kappasqq}{\kappa^\sigma_\Kqq}
\newcommand{\ghatqq}{\hat{\g}_\Kqq}
\newcommand{\hhatqq}{\hat{\h}_\Kqq}
\newcommand{\kappahatqq}{\hat{\kappa}_\Kqq}

\newcommand{\Rkk}{R_\K}
\newcommand{\Skk}{S_\K}
\newcommand{\otimeskk}{\otimes_{\K}}
\newcommand{\frakA}{\mathfrak{A}_\K}

\newcommand{\LL}{\mathcal{L}}
\newcommand{\gtw}{\LL(\gqq;\Gamma)}

\newcommand{\gidd}{\LL(\gqq)}

\newcommand{\gtwh}{\hat\LL(\gqq;\Gamma)}
\newcommand{\htwh}{\hat\LL(\hqq;\Gamma)}

\newcommand{\boldmathfont}[1]{\boldsymbol{#1}}
\newcommand{\cc}{\boldmathfont{c}}
\newcommand{\dd}{\boldmathfont{d}}

\newcommand{\GG}{\mathfrak{G}}
\newcommand{\TT}{\mathfrak{T}}
\newcommand{\BB}{\mathfrak{B}}
\newcommand{\UU}{\mathfrak{U}}
\newcommand{\NN}{\mathfrak{N}}
\newcommand{\ZZ}{\mathfrak{Z}}
\newcommand{\SSS}{\mathfrak{S}}
\newcommand{\GGh}{\hat{\GG}}
\newcommand{\TTh}{\hat{\TT}}
\newcommand{\BBh}{\hat{\BB}}
\newcommand{\UUh}{\hat{\UU}}
\newcommand{\ZZh}{\hat{\ZZ}}
\newcommand{\GGk}{\GG_\K}
\newcommand{\TTk}{\TT_\K}
\newcommand{\BBk}{\BB_\K}
\newcommand{\UUk}{\UU_\K}
\newcommand{\NNk}{\NN_\K}

\newcommand{\GGhk}{\GGh_\K}
\newcommand{\TThk}{\TTh_\K}
\newcommand{\BBhk}{\BBh_\K}
\newcommand{\UUhk}{\UUh_\K}

\newcommand{\ZZhk}{\ZZh_\K}
\newcommand{\GGhkzeta}{\GGh_{\K(\xi)}}

\newcommand{\FDeltaa}{\mathfrak{F}_{\K}(\Deltar)}
\newcommand{\FDeltas}{\mathfrak{F}_{\K}(\Deltashr)}
\newcommand{\FDeltazeta}{\mathfrak{F}_{\K(\xi)}(\Deltahr)}

\newcommand{\Gkm}{\GGhk(\XX_N^{(1)})}
\newcommand{\Gkms}{\GGhk(\XX_N^{(r)})}
\newcommand{\Gkmtw}{\Gkm^{\Gamma}}
\newcommand{\Gkmzeta}{\GGhkzeta(\XX_N^{(1)})}

\newcommand{\Gkmtwzeta}{\Gkmzeta^{\Gamma}}
\newcommand{\Gzz}{\boldmathfont{G}_\mathbb{Z}}
\newcommand{\G}{\boldmathfont{G}_\K}
\newcommand{\Gcd}{\boldmathfont{G}}
\newcommand{\Gfix}{\G(\Skk)^\Gamma}
\newcommand{\Gtw}{\LL(\G;\Gamma)}

\newcommand{\Gidd}{\LL(\G)}

\newcommand{\E}[1]{\mathcal{E}(#1;{\Deltas})}
\newcommand{\EE}[1]{\mathcal{E}(#1)}

\newcommand{\Tss}{\boldmathfont{T}}
\newcommand{\Uss}{\boldmathfont{U}}
\newcommand{\Bss}{\boldmathfont{B}}
\newcommand{\Vss}{\boldmathfont{V}}
\newcommand{\SL}{\boldmathfont{SL}}
\newcommand{\SU}{\boldmathfont{SU}}
\newcommand{\WW}{\mathcal{W}}
\newcommand{\WWh}{\hat{\WW}}

\newcommand{\lenm}{\mathfrak{m}}
\newcommand{\lenM}{\mathfrak{M}}
\newcommand{\plen}{\mathfrak{k}}
\newcommand{\sgn}{\varepsilon}
\newcommand{\ccc}{\mathfrak{c}}
\newcommand{\bcc}{\mathfrak{b}}

\newcommand{\sigmax}{\sigma'}
\newcommand{\omegax}{{\omega'}}
\newcommand{\omegaxx}{{\omega'}}
\newcommand{\deltaa}{\boldmathfont{\delta}}

\newcommand{\xsx}{\tilde{x}}
\newcommand{\wsw}{\tilde{w}}
\newcommand{\hsh}{\tilde{h}}
\newcommand{\xs}{\tilde{x}}
\newcommand{\ws}{\tilde{w}}
\newcommand{\hs}{\tilde{h}}

\newcommand{\hits}{\rightharpoonup}
\newcommand{\ifff}{\leftrightarrow}

\newcommand{\dotminus}{\mathbin{\rlap{\raisebox{-2pt}{$\dot{\phantom{-}}$}}{-}}}

\begin{document}

\title[Affine Kac-Moody Groups as Twisted Loop Groups]{Affine Kac-Moody Groups as Twisted Loop Groups obtained by Galois Descent Considerations}

\author[J.~Morita]{Jun Morita}
\author[A.~Pianzola]{Arturo Pianzola}
\author[T.~Shibata]{Taiki Shibata}
\dedicatory{Dedicated to Professor~Robert~V.~Moody on the~occasion of his~80th~birthday.}

\subjclass[2010]{Primary 20G44; Secondary 22E67}
\keywords{Affine Kac-Moody groups, Loop groups, Twisted Chevalley groups}

\address[J.~Morita]{%
Institute of Mathematics\endgraf
University of Tsukuba\endgraf
1-1-1 Tennodai, Tsukuba, Ibaraki 305-8571\endgraf
Japan}
\email{morita@math.tsukuba.ac.jp}

\address[A.~Pianzola]{%
Department of Mathematical and Statistical Sciences\endgraf
University of Alberta\endgraf
Edmonton, Alberta T6G 2G1\endgraf
Canada\endgraf
and\endgraf
Centro de Altos Estudios en Ciencias Exactas\endgraf
Avenida de Mayo 866, (1084), Buenos Aires\endgraf
Argentina}
\email{a.pianzola@ualberta.ca}

\address[T.~Shibata]{%
Department of Applied Mathematics\endgraf
Okayama University of Science\endgraf
1-1 Ridai-cho Kita-ku, Okayama, Okayama 700-0005\endgraf
Japan}
\email{shibata@xmath.ous.ac.jp}

\maketitle

\begin{abstract}
We provide explicit generators and relations for the affine Kac-Moody groups,
as well as a realization of them as (twisted) loop groups by means of Galois descent considerations.
As a consequence, we show that the affine Kac-Moody group of type $\XX_N^{(r)}$ is isomorphic to
the fixed-point subgroup of the affine Kac-Moody group of type $\XX_N^{(1)}$ under an action of the Galois group.
\end{abstract}

\setcounter{tocdepth}{1}
\tableofcontents

\section{Introduction}\label{sec:intro}
Kac-Moody Lie algebras, a particular class of infinite dimensional Lie algebras,
were independently discovered by V.~Kac \cite{Kac1} and R.~V.~Moody \cite{Mdy1}.
Given a field $\K$ of characteristic $0$,
they are defined by generators and relations (\`a la Chevalley-Harish-Chandra-Serre) encoded in a generalized Cartan matrix (GCM) $\XA$.
If $\XA$ is a Cartan matrix of type $\XX$,
then the corresponding Kac-Moody Lie algebra is nothing but the split simple finite dimensional $\K$-Lie algebra of type $\XX$.
Closely related to these are the affine Lie algebras (see \cite[Chapter~7]{Kac} and \cite{Mdy2}).

The first step towards the construction/definition of Kac-Moody groups is given in \cite{Moody-Teo}.
This was done using representation theory and admissible lattices following along the lines of
Chevalley's early work on analogues of the simple Lie groups over arbitrary fields.
A summary of this approach and a vision of the steps ahead can be found in \cite{Tit1}.
Rather than working with one representation,
Peterson and Kac considered all (integrable) representations at once in their definition of ``simply connected'' Kac-Moody groups
over fields of characteristic $0$ given in \cite{Peterson-Kac}.
This paper establishes the conjugacy theorem of ``Cartan subalgebras'' of symmetrizable Kac-Moody Lie algebras and,
as a consequence, that the GCMs and corresponding root systems are an invariant of the algebras.
Detailed expositions of this material are given in \cite{Kum} and \cite{MP}.

If $\XA$ is of finite type $\XX$,
the corresponding ``groups'' (with the simply connected being the largest) exist (Chevalley) and are unique (Demazure).
They are smooth group schemes over $\mathbb{Z}$ constructed using a ``root datum'' that includes $\XA$.
The base change to $\K$ produces a linear algebraic group over $\K$ whose Lie algebra is split simple of type $\XX$ if $\K$ is of characteristic $0$.

J.~Tits pioneered the idea of defining root datum based on GCMs,
and attaching to them group functors that ``behave right'' when evaluated at arbitrary fields \cite{Tit2}.
See also \cite{Rem}.
Some of the affine cases are discussed in examples.
Further clarity about the nature of the abstract groups obtained in this fashion is given by
the construction of Kac-Moody groups due to O.~Mathieu in \cite{Mth}.
See also \cite{HLR,Lou,PR}.

While the works described above deal with arbitrary (symmetrizable) GCM,
the focus of our paper are the (abstract) groups attached to the affine GCMs.
This can be done over arbitrary fields (avoiding characteristic $2$ and $3$ in the twisted cases) by considering
fixed points of Chevalley-Demazure group schemes evaluated at suitable Laurent polynomial rings.
The link with the representation related approach, and with the work of Tits,
is given by a detailed analysis motivated by the work of Steinberg, that yields generators and relations for these groups.
This method of studying of affine Kac-Moody groups was pioneered by E.~Abe, N.~Iwahori and H.~Matsumoto, and J.~Morita (see references).

\subsection*{Acknowledgments}
We are grateful to the referees for the careful reading of our manuscript and useful comments.
J.M. is supported by JSPS~KAKENHI (Grants-in-Aid for Scientific Research) Grant Number~JP17K05158.
A.P. acknowledges the continued support of NSERC and Conicet.
T.S. was a Pacific Institute for the Mathematical Sciences (PIMS) Postdoctoral Fellow position at University of Alberta,
and is supported by JSPS~KAKENHI Grant Number~JP19K14517.

\section{Structure of the Paper and Main Results}\label{sec:main-results}
The article is organized as follows.
In Section~\ref{sec:tw-chev-grp}, we recall some basic definitions and notation for twisted root systems
and Chevalley basis of finite-dimensional simple Lie algebra $\gcc$ over $\mathbb{C}$. The type and rank of $\gcc$ will be denoted by $\XX_N$ ($\XX=\XA,\XB,\XC,\XD,\XE,\XF$, or $\XG$). The description of Dynkin diagram automorphisms as automorphisms of $\gcc$ is given in this section.
Section~\ref{sec:aff} starts with the construction of twisted loop algebras associated to $\gcc$ over $\Kqq$.
We then give an explicit description of a Lie algebra isomorphism $\varphi$ from the twisted loop algebra 
to the affine Kac-Moody algebra $\ghatqq(\XX_N^{(r)})$ of type $\XX_N^{(r)}$ defined over $\Kqq$ (Theorem~\ref{prp:isom}),
where $r$ is the tier number (i.e., the order of the Dynkin diagram automorphism) and $\XX_N^{(r)}$ is as in Kac's list \cite[TABLE~Aff $r$]{Kac}.
Using the isomorphism $\varphi$, we translate the notion of Chevalley pairs (Definition~\ref{def:chev-pair}) of $\ghatqq(\XX_N^{(r)})$ 
into the twisted loop algebra (Proposition~\ref{prp:1}).

In the  rest of the paper (Sections~\ref{sec:tw_loop}, \ref{sec:aff-KM-grp}, and \ref{sec:tw-aff-KM-grp}),
we work over a field $\K$ of characteristic not equal to $2$ (resp.~$3$) when we consider the case $r=2$ (resp.~$r=3$).
Let $\Gcd$ be the simply-connected Chevalley-Demazure group scheme over $\mathbb{Z}$ of type $\XX_N$ (cf.~\cite{DG,SGA3}).
In Section~\ref{sec:tw_loop}, using Abe's construction \cite{Abe}, we introduce the notion of the twisted loop group $\Gtw$,
as a $\Gamma$-twisted Chevalley group over the Laurent polynomial ring $\Skk:=\K(\xi)[z^{\pm\frac{1}{r}}]$,
where $\xi$ is a fixed primitive $r$ th root of unity in an algebraic closure of $\K$.
Here, $\Gamma$ is the Galois group of $\Skk$ over $\Rkk:=\K[z^{\pm1}]$.
We determine generators of the group $\Gtw$ explicitly (Theorem~\ref{prp:elem}).
Since we have described the Lie algebra isomorphism $\varphi$ concretely (in Section~\ref{sec:chev-pair}),
we are able to write down the ``induced'' group homomorphism $\Phi$
from the simply-connected affine Kac-Moody group $\Gkms$ of type $\XX_N^{(r)}$ defined over $\K$
(see Definition~\ref{def:KM_grp} for details)
to the twisted loop group $\Gtw$ explicitly.
In Section~\ref{sec:aff-KM-grp}, we show that the homomorphism $\Phi$ is surjective (Proposition~\ref{prp:phi-surj})
and determine the kernel of $\Phi$.
As a result, we have the following.
\begin{theorem0}[\mbox{Theorem~\ref{prp:main}}]
The simply-connected affine Kac-Moody group $\Gkms$ of type $\XX_N^{(r)}$ defined over $\K$
is a one-dimensional central extension of $\Gtw$.
In particular, $\Gkms/\K^\times\cong\Gtw$.
\end{theorem0}

In the final Section~\ref{sec:tw-aff-KM-grp},
we define a $\Gamma$-action on the simply-connected affine Kac-Moody group $\Gkmzeta$ of type $\XX_N^{(1)}$ defined over $\K(\xi)$
and study structure of the fixed-point subgroup of $\Gkmzeta$ under $\Gamma$.
Using the results in the previous section, we have the following result.
\begin{theorem0}[\mbox{Theorem~\ref{prp:main2}}]
The fixed-point subgroup of $\Gkmzeta$ under $\Gamma$ is isomorphic to $\Gkms$.
\end{theorem0}

If $\xi\in\K$, then our results can be express as the following commutative diagram.
\[
\begin{xy}
(-40,24)*++{1}="223",(-23,24)*++{\K^\times}="122",(0,24)*++{\Gkm}="111",(25,24)*++{\Gidd}="222",(45,24)*++{\G(\Skk)}="999",(65,24)*++{1}="9999",
(-40,8)*++{1}="22",(-23,8)*++{\K^\times}="11",(0,8)*++{\Gkmtw}="1",(45,8)*++{\Gfix}="2",(65,8)*++{1}="99",
(-40,-8)*++{1}="44",(-23,-8)*++{\K^\times}="33",(0,-8)*++{\Gkms}="3",(45,-8)*++{\Gtw}="4",(65,-8)*++{1.}="9",
{"223"\SelectTips{cm}{}\ar@{->}"122"},{"122"\SelectTips{cm}{}\ar@{^(->}"111"},{"111"\SelectTips{cm}{}\ar@{->>}"222"},
{"222"\SelectTips{cm}{}\ar@{->}^{\cong}"999"},{"999"\SelectTips{cm}{}\ar@{->}"9999"},
{"122"\SelectTips{cm}{}\ar@{=}"11"},
{"111"\SelectTips{cm}{}\ar@{}|{\rotatebox{90}{$\subset$}}"1"},
{"999"\SelectTips{cm}{}\ar@{}|{\rotatebox{90}{$\subset$}}"2"},
{"22"\SelectTips{cm}{}\ar@{->}"11"},{"11"\SelectTips{cm}{}\ar@{^(->}"1"},{"1"\SelectTips{cm}{}\ar@{-->>}"2"},{"2"\SelectTips{cm}{}\ar@{->}"99"},
{"11"\SelectTips{cm}{}\ar@{=}"33"},{"1"\SelectTips{cm}{}\ar@{<--}^\cong"3"},{"2"\SelectTips{cm}{}\ar@{=}"4"},
{"44"\SelectTips{cm}{}\ar@{->}"33"},{"33"\SelectTips{cm}{}\ar@{^(->}"3"},{"3"\SelectTips{cm}{}\ar@{->>}"4"},{"4"\SelectTips{cm}{}\ar@{->}"9"}
\end{xy}
\]
Here, $\Gidd:=\G(\Rkk)$ is the so-called {\it loop group}, and $\Gidd\overset{\cong}{\to}\G(\Skk)$ is a group isomorphism induced from 
the canonical algebra isomorphism $\Rkk\to\Skk;z^n\mapsto z^{\frac{n}{r}}$ ($n\in\mathbb{Z}$).

\section{Twisted Root Systems} \label{sec:tw-chev-grp}
In this section, we work over the field of complex numbers $\mathbb{C}$.
Let $\gcc$ be a complex simple Lie algebra of type $\XX_N.$  Recall that  $N$ is the rank of $\gcc$ and $\XX=\XA,\XB,\XC,\XD,\XE,\XF$, or $\XG$. We fix  a Cartan subalgebra $\hcc$ and a Borel subalgebra $\bcc$ of $\gcc$ with $\hcc \subset \bcc.$ 
We let $\hcc^*$ denote the linear dual space of $\hcc$.

Let $\Delta$ be the root system of $\gcc$ with respect to $\hcc$ and
let $\Pi=\{\alpha_i \mid i\in \J\}$ be the set of all simple roots with respect to our chosen Borel subalgebra of $\gcc$,
where $\J:=\{1,\dots,N\}$.
For each $\alpha\in\Delta$, we let $\gcc_\alpha$ denote the root space in $\gcc$ corresponding to $\alpha$.
Let $<$ be the lexicographical order on $\Delta$ defined by $\Pi$,
and let $\Deltap$ (resp.~$\Delta_-$) be the set of all positive (resp.~negative) roots in $\Delta$ with respect to $<$.

\subsection{Dynkin Diagram Automorphisms} \label{sec:twnotation} 
Let $\sigma\in\Aut_{\mathbb{C}}(\hcc^*)$ be a Dynkin diagram automorphism of $\gcc$, and let $r$ be the order of $\sigma$.
For simplicity, the induced automorphism on $\hcc$ and the extended automorphism on $\gcc$ are also denoted by
the same symbol $\sigma$.
Let $\hsigmascc$ denote the fixed-point subspace of $\hcc^*$ under $\sigma$.
Then we have a canonical projection $\pi:\hcc^*\to \hsigmascc$ defined by
$\pi(\lambda)=\sum_{j=1}^r\frac{1}{\,r\,}\sigma^j(\lambda)$ for $\lambda\in\hcc^*$.

We let $\pi(\Pi)$ (resp.~$\pi(\Delta)$) denote the image of $\Pi$ (resp.~$\Delta$) under $\pi$.
For the case of $\XX_N=\XA_{2\ell}$ with $r=2$, there is $a'\in\pi(\Delta)$ such that $\frac{a'}{2}$ belongs to $\pi(\Delta)$.
Thus, we shall define a subset of $\pi(\Delta)$ as follows.
\begin{equation}
\Deltas :=
\begin{cases}
\pi(\Delta)\setminus\{a'\in\pi(\Delta) \mid \frac{a'}{2}\in\pi(\Delta)\} & \text{if }(\XX_N,r)=(\XA_{2\ell},2),\\
\pi(\Delta) & \text{otherwise}.
\end{cases}
\end{equation}
One sees that this $\Deltas$ forms a root system of the fixed-point subalgebra $\gsigmacc$ of $\gcc$ under $\sigma$
with respect to $\hsigmacc=\hcc\cap\gsigmacc$.
Let $\XNsigma$ be the type of $\gsigmacc$ with respect to $\hsigmacc$.
In this case, $\Pis:=\Deltas\cap\pi(\Pi)$ forms the set of all simple roots.
Set $\ell:=\#\Pis$ and set $\Deltasp:=\Deltas\cap\pi(\Deltap)$.
Let $\Deltaslong$ (resp.~$\Deltasshort$) denote the set of all long (resp.~short) roots in $\Deltas$.
If all roots of $\Deltas$ are of the same length, then we set
$\Deltaslong=\Deltas$ and $\Deltasshort=\emptyset$.

The Dynkin diagram automorphism $\sigma$ naturally acts on the set $\J$. 
Let $\Jsigma$ be the set of all equivalence classes of $\J$.
By definition, the set $\Jsigma$ consists of $\ell$ elements.
For $p\in\Jsigma$ and a fixed $i\in p$, we define $a_p\in\hsigmascc$ by letting $a_p(H):=\alpha_i(H)$ for all $H\in\hsigmacc$.
Then we can identify the set $\{a_p\}_{p\in\Jsigma}$ with $\Pis$.
For simplicity, we shall identify $\Jsigma$ with $\{1,\dots,\ell\}$, and write $\Pis=\{a_1,\dots,a_\ell\}$.
\medskip

The following is the complete list of non-trivial Dynkin diagram automorphisms $\sigma$ of order $r$.
Note that $r$ only takes the values $2$ or $3$.
\begin{enumerate}
\item $\XX_N=\XA_{2\ell-1}$ ($\ell\geq2$) and $r=2$.
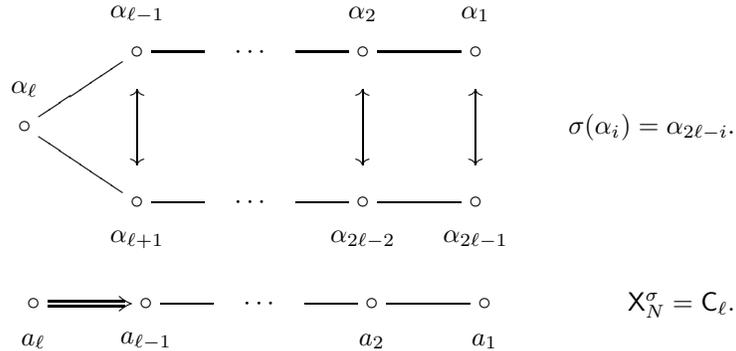
\[
\begin{array}{ccr}
\begin{xy}
(45,10)*+{\circ}="1", (45,15)*+{\alpha_1}="11", (45,5)*{}="111", (30,10)*+{\circ}="2",
(30,15)*+{\alpha_2}="22", (30,5)*{}="222", (20,10)*+{}="2.5", (15,10)*+{\cdots}="2.75",
(10,10)*+{}="3.5", (0,10)*+{\circ}="3", (0,15)*+{\alpha_{\ell-1}}="33", (0,5)*{}="333",
(45,-10)*+{\circ}="-1", (45,-15)*+{\alpha_{2\ell-1}}="-11", (45,-5)*{}="-111",
(30,-10)*+{\circ}="-2", (30,-15)*+{\alpha_{2\ell-2}}="-22", (30,-5)*{}="-222",
(20,-10)*+{}="-2.5", (15,-10)*+{\cdots}="-2.75", (10,-10)*+{}="-3.5", (0,-10)*+{\circ}="-3", (0,-15)*+{\alpha_{\ell+1}}="-33", (0,-5)*{}="-333",
(-15,0)*+{\circ}="4", (-15,5)*+{\alpha_\ell}="44",
{"1"\SelectTips{cm}{}\ar@{-}"2"}, {"3"\SelectTips{cm}{}\ar@{-}"4"}, {"2"\SelectTips{cm}{}\ar@{-}"2.5"},
{"3.5"\SelectTips{cm}{}\ar@{-}"3"}, {"-1"\SelectTips{cm}{}\ar@{-}"-2"}, {"-3"\SelectTips{cm}{}\ar@{-}"4"},
{"-2"\SelectTips{cm}{}\ar@{-}"-2.5"}, {"-3.5"\SelectTips{cm}{}\ar@{-}"-3"}, {"-3.5"\SelectTips{cm}{}\ar@{-}"-3"},
{"-333"\SelectTips{cm}{}\ar@{<->}"333"}, {"-222"\SelectTips{cm}{}\ar@{<->}"222"}, {"-111"\SelectTips{cm}{}\ar@{<->}"111"}
\end{xy}
&&\sigma(\alpha_i)=\alpha_{2\ell-i}.\\
&&\\
\begin{xy}
(45,0)*+{\circ}="1", (45,-5)*+{a_1}="11", (45,-5)*{}="111", (30,0)*+{\circ}="2", (30,-5)*+{a_2}="22",
(30,-5)*{}="222", (20,0)*+{}="2.5", (15,0)*+{\cdots}="2.75", (10,0)*+{}="3.5", (0,0)*+{\circ}="3",
(0,-5)*+{a_{\ell-1}}="33", (0,-5)*{}="333", (-15,0)*+{\circ}="4", (-15,-5)*+{a_\ell}="44",
{"1"\SelectTips{cm}{}\ar@{-}"2"}, {"3"\SelectTips{cm}{}\ar@{<=}"4"}, {"2"\SelectTips{cm}{}\ar@{-}"2.5"}, {"3.5"\SelectTips{cm}{}\ar@{-}"3"}
\end{xy}
&&\XNsigma=\XC_\ell.
\end{array}
\]

\item $\XX_N=\XA_{2\ell}$ ($\ell\geq 1$) and $r=2$.
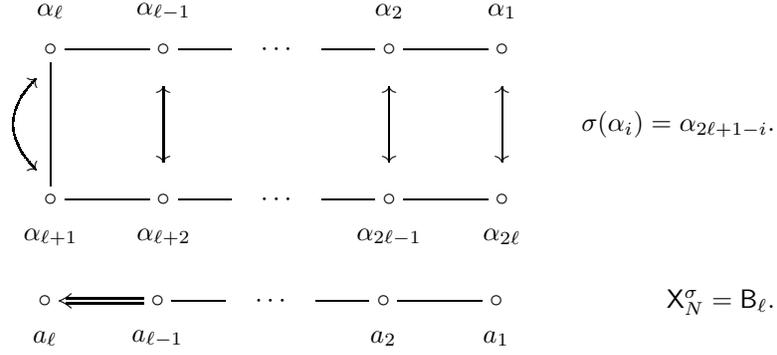
\[
\begin{array}{ccr}
\begin{xy}
(45,10)*+{\circ}="1", (45,15)*+{\alpha_1}="11", (45,5)*{}="111", (30,10)*+{\circ}="2", (30,15)*+{\alpha_2}="22",
(30,5)*{}="222", (20,10)*+{}="2.5", (15,10)*+{\cdots}="2.75", (10,10)*+{}="3.5", (0,10)*+{\circ}="3", (0,15)*+{\alpha_{\ell-1}}="33",
(0,5)*{}="333", (45,-10)*+{\circ}="-1", (45,-15)*+{\alpha_{2\ell}}="-11", (45,-5)*{}="-111", (30,-10)*+{\circ}="-2",
(30,-15)*+{\alpha_{2\ell-1}}="-22", (30,-5)*{}="-222", (20,-10)*+{}="-2.5", (15,-10)*+{\cdots}="-2.75",
(10,-10)*+{}="-3.5", (0,-10)*+{\circ}="-3", (0,-15)*+{\alpha_{\ell+2}}="-33", (0,-5)*{}="-333", (-15,10)*+{\circ}="4",
(-15,15)*+{\alpha_\ell}="44", (-17,6)*{}="444", (-15,-10)*+{\circ}="-4", (-15,-15)*+{\alpha_{\ell+1}}="-44", (-17,-6)*{}="-444",
{"1"\SelectTips{cm}{}\ar@{-}"2"}, {"3"\SelectTips{cm}{}\ar@{-}"4"}, {"2"\SelectTips{cm}{}\ar@{-}"2.5"},
{"3.5"\SelectTips{cm}{}\ar@{-}"3"}, {"-1"\SelectTips{cm}{}\ar@{-}"-2"}, {"-3"\SelectTips{cm}{}\ar@{-}"-4"},
{"-4"\SelectTips{cm}{}\ar@{-}"4"}, {"-2"\SelectTips{cm}{}\ar@{-}"-2.5"}, {"-3.5"\SelectTips{cm}{}\ar@{-}"-3"},
{"-3.5"\SelectTips{cm}{}\ar@{-}"-3"}, {"-333"\SelectTips{cm}{}\ar@{<->}"333"}, {"-222"\SelectTips{cm}{}\ar@{<->}"222"},
{"-111"\SelectTips{cm}{}\ar@{<->}"111"}, {"-444"\SelectTips{cm}{}\ar@/^3mm/"444"}, {"444"\SelectTips{cm}{}\ar@/^-3mm/"-444"}
\end{xy}
&&\sigma(\alpha_i)=\alpha_{2\ell+1-i}.\\
&&\\
\begin{xy}
(45,0)*+{\circ}="1", (45,-5)*+{a_1}="11", (45,-5)*{}="111", (30,0)*+{\circ}="2", (30,-5)*+{a_2}="22",
(30,-5)*{}="222", (20,0)*+{}="2.5", (15,0)*+{\cdots}="2.75", (10,0)*+{}="3.5", (0,0)*+{\circ}="3",
(0,-5)*+{a_{\ell-1}}="33", (0,-5)*{}="333", (-15,0)*+{\circ}="4", (-15,-5)*+{a_\ell}="44",
{"1"\SelectTips{cm}{}\ar@{-}"2"}, {"3"\SelectTips{cm}{}\ar@{=>}"4"}, {"2"\SelectTips{cm}{}\ar@{-}"2.5"}, {"3.5"\SelectTips{cm}{}\ar@{-}"3"}
\end{xy}
&&\XNsigma=\XB_\ell.
\end{array}
\]

\item $\XX_N=\XD_{\ell+1}$ ($\ell\geq 3$) and $r=2$.
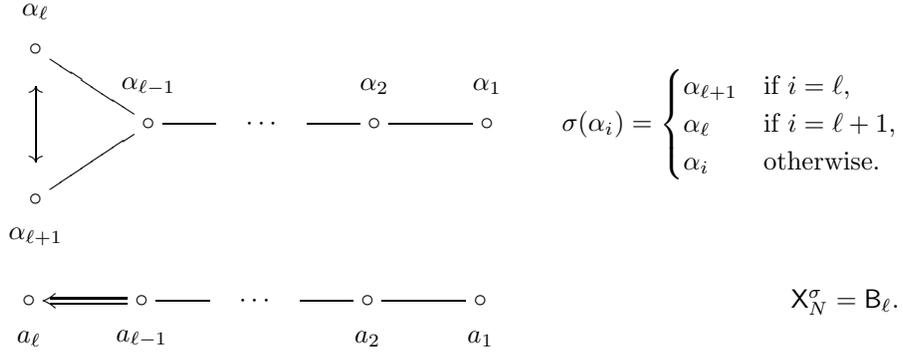
\[
\begin{array}{ccr}
\begin{xy}
(45,0)*+{\circ}="1", (45,05)*+{\alpha_1}="11", (30,0)*+{\circ}="2", (30,5)*+{\alpha_2}="22", (20,0)*+{}="2.5", (15,0)*+{\cdots}="2.75",
(10,0)*+{}="3.5", (0,0)*+{\circ}="3", (0,5)*+{\alpha_{\ell-1}}="33", (-15,10)*+{\circ}="4", (-15,15)*+{\alpha_{\ell}}="44",
(-15,-10)*+{\circ}="5", (-15,-15)*+{\alpha_{\ell+1}}="55", (-15,5)*{}="555", (-15,-5)*{}="-555",
{"1"\SelectTips{cm}{}\ar@{-}"2"}, {"3"\SelectTips{cm}{}\ar@{-}"4"}, {"3"\SelectTips{cm}{}\ar@{-}"5"},
{"2"\SelectTips{cm}{}\ar@{-}"2.5"}, {"3.5"\SelectTips{cm}{}\ar@{-}"3"}, {"-555"\SelectTips{cm}{}\ar@{<->}"555"}
\end{xy}
&& \sigma(\alpha_i) =
\begin{cases}
\alpha_{\ell+1} & \text{if }i=\ell,\\
\alpha_{\ell} & \text{if }i=\ell+1,\\
\alpha_i & \text{otherwise.}
\end{cases}\\
&&\\
\begin{xy}
(45,0)*+{\circ}="1", (45,-5)*+{a_1}="11", (45,-5)*{}="111", (30,0)*+{\circ}="2", (30,-5)*+{a_2}="22", (30,-5)*{}="222",
(20,0)*+{}="2.5", (15,0)*+{\cdots}="2.75", (10,0)*+{}="3.5", (0,0)*+{\circ}="3", (0,-5)*+{a_{\ell-1}}="33",
(0,-5)*{}="333", (-15,0)*+{\circ}="4", (-15,-5)*+{a_\ell}="44",
{"1" \SelectTips{cm}{} \ar @{-} "2"}, {"3" \SelectTips{cm}{} \ar @{=>} "4"}, {"2" \SelectTips{cm}{} \ar @{-} "2.5"}, {"3.5" \SelectTips{cm}{} \ar @{-} "3"}
\end{xy}
&&\XNsigma=\XB_\ell.
\end{array}
\]

\item $\XX_N=\XE_{6}$ and $r=2$.
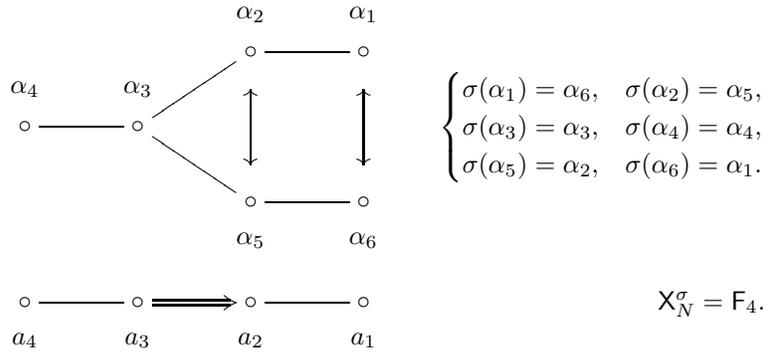
\[
\begin{array}{ccr}
\begin{xy}
(45,10)*+{\circ}="1", (45,15)*+{\alpha_1}="11", (45,5)*{}="111", (30,10)*+{\circ}="2", (30,15)*+{\alpha_2}="22", (30,5)*{}="222",
(10,10)*+{}="3.5", (15,0)*+{\circ} ="3", (15,5)*+{\alpha_{3}}="33", (45,-10)*+{\circ}="-1", (45,-15)*+{\alpha_{6}}="-11",
(45,-5)*{}="-111", (30,-10)*+{\circ}="-2", (30,-15)*+{\alpha_{5}}="-22", (30,-5)*{}="-222", (0,0)*+{\circ}="4", (0,5)*+{\alpha_4}="44",
{"1"\SelectTips{cm}{}\ar@{-}"2"}, {"3"\SelectTips{cm}{}\ar@{-}"4"}, {"3"\SelectTips{cm}{}\ar@{-}"2"}, {"3"\SelectTips{cm}{}\ar@{-}"-2"},
{"-1"\SelectTips{cm}{}\ar@{-}"-2"}, {"-222"\SelectTips{cm}{}\ar@{<->}"222"}, {"-111"\SelectTips{cm}{}\ar@{<->}"111"}
\end{xy}
&&
\begin{cases}
\sigma(\alpha_1) = \alpha_6, & \sigma(\alpha_2) = \alpha_5,\\
\sigma(\alpha_3) = \alpha_3, & \sigma(\alpha_4) = \alpha_4,\\
\sigma(\alpha_5) = \alpha_2, & \sigma(\alpha_6) = \alpha_1.
\end{cases}\\
&&\\
\begin{xy}
(45,0)*+{\circ}="1", (45,-5)*+{a_1}="11", (30,0)*+{\circ}="2", (30,-5)*+{a_2}="22", (20,0)*+{}="2.5",
(15,0)*+{\circ}="3", (0,0)*+{\circ}="4", (15,-5)*+{a_3}="33", (0,-5)*+{a_4}="44",
{"1"\SelectTips{cm}{}\ar@{-}"2"}, {"2"\SelectTips{cm}{}\ar@{<=}"3"}, {"3"\SelectTips{cm}{}\ar@{-}"4"}
\end{xy}
&&\XNsigma=\XF_4.
\end{array}
\]

\item $\XX_N=\XD_4$ and $r=3$.
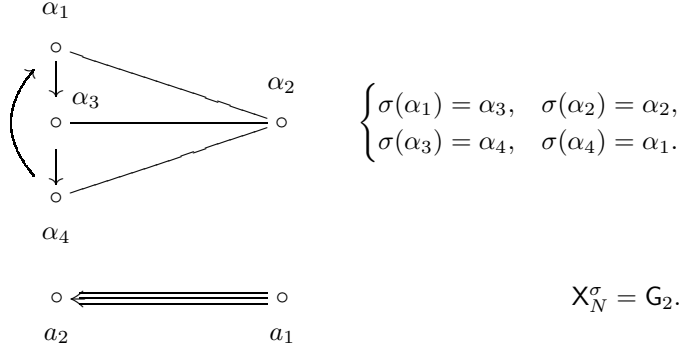
\[
\begin{array}{ccr}
\begin{xy}
(-5,0)*+{}, (0,10)*+{\circ}="2", (0,15)*+{\alpha_1}="22", (-3,7)*{}="222", (10,10)*+{}="3.5", (30,0)*+{\circ}="3",
(4,3)*+{\alpha_{3}}="33", (0,2.5)*+{}="9", (0,-2.5)*+{}="8", (0,-10)*+{\circ}="-2", (0,-15)*+{\alpha_{4}}="-22", (-3,-7)*{}="-222",
(0,0)*+{\circ}="4", (30,5)*+{\alpha_2}="44",
{"2"\SelectTips{cm}{}\ar@{->}"9"}, {"3"\SelectTips{cm}{}\ar@{-}"4"}, {"3"\SelectTips{cm}{}\ar@{-}"2"},
{"3"\SelectTips{cm}{}\ar@{-}"-2"}, {"8"\SelectTips{cm}{}\ar@{->}"-2"}, {"-222"\SelectTips{cm}{}\ar@/^3mm/"222"}
\end{xy}
&&
\begin{cases}
\sigma(\alpha_1) = \alpha_3, & \sigma(\alpha_2) = \alpha_2,\\
\sigma(\alpha_3) = \alpha_4, & \sigma(\alpha_4) = \alpha_1.
\end{cases}\\
&&\\
\begin{xy}
(-5,0)*+{}, (30,0)*+{\circ}="2", (30,-5)*+{a_1}="22", (20,0)*+{}="2.5", (0,0)*+{\circ}="4", (0,-5)*+{a_2}="44",
{"2"\SelectTips{cm}{}\ar@3{->}"4"}
\end{xy}
&&\XNsigma=\XG_2.
\end{array}
\]
\end{enumerate}

\subsection{Types of Twisted Roots} \label{sec:typ_root}
Each $\alpha\in\Delta$ satisfies one of the following four conditions,
see \cite[Section~1]{Mor81}.
\begin{description}
\item[(R-1)] $\alpha=\sigma(\alpha)$.
\item[(R-2)] $r=2$ with $\alpha\neq\sigma(\alpha)$ and $\alpha+\sigma(\alpha)\notin\Delta$.
\item[(R-3)] $r=2$ with $\alpha\neq\sigma(\alpha)$ and $\alpha+\sigma(\alpha)\in\Delta$.
\item[(R-4)] $r=3$ with $\alpha\neq\sigma(\alpha)$ and $\alpha\neq\sigma^2(\alpha)$.
\end{description}

If $r=1$, then all roots are of type (R-1).
Otherwise, we have to clarify the difference between $\Delta$ and $\Deltas$.
Hence, in order to avoid confusion,
we use {\it Greek characters} $\alpha,\beta$ to describe elements in $\Delta$,
and use {\it Alphabet characters} $a,b$ to describe elements in $\Deltas$ and $\pi(\Delta)$.

We have the following classification of elements in $\pi(\Delta)$ (cf.~\cite[Table~2]{Mor81}).
\[
\begin{array}{c||c|l}
\XX_N & r &\,\,\text{Types}\quad\text{Lengths}\\ \hline\hline
\XA_{2\ell-1},\,\XD_{\ell+1},\,\XE_6 & 2 &
\begin{array}{ll}
\text{(R-1)} & \text{long}\\ \hline
\text{(R-2)} & \text{short}
\end{array}\\ \hline
\XA_{2} & 2 &
\begin{array}{ll}
\text{(R-1)} & \text{extra long}\\ \hline
\text{(R-3)} & \text{short}
\end{array}\\ \hline
\XA_{2\ell\,\neq2} & 2 &
\begin{array}{ll}
\text{(R-1)} & \text{extra long}\\ \hline
\text{(R-2)} & \text{long}\\ \hline
\text{(R-3)} & \text{short}
\end{array}\\ \hline
\XD_{4} & 3 &
\begin{array}{ll}
\text{(R-1)} & \text{long}\\ \hline
\text{(R-4)} & \text{short}
\end{array}\\ \hline
\end{array}
\]

\begin{notation} \label{not:a'}
Suppose that $\XX_N=\XA_{2\ell}$ with $r=2$.
For $a\in\Deltas$, we shall denote $a$ or $2a$ by $a'$ so that $a'\in\pi(\Delta)$.
For the other types, we let $a':=a$ for each $a\in\Deltas$.
\end{notation}

\begin{example} \label{ex:A_4}
The case $(\XX_N,r)=(\XA_4,2)$.
The set of all positive roots are given by
$\Deltap = \{
\alpha_1,\alpha_2, \alpha_3,\alpha_4,
\alpha_1+\alpha_2, \alpha_2+\alpha_3, \alpha_3+\alpha_4,
\alpha_1+\alpha_2+\alpha_3, \alpha_2+\alpha_3+\alpha_4,
\alpha_1+\alpha_2+\alpha_3+\alpha_4 \}$.
The Dynkin diagram automorphism of order $2$ is given by $\sigma(\alpha_1)=\alpha_4$ and $\sigma(\alpha_2)=\alpha_3$.
Then one sees
\[
\pi(\Deltap) = 
\underset{\text{(R-3) short}}{\underline{\{a_2, a_1+a_2\}}} \sqcup 
\underset{\text{(R-2) long}}{\underline{\{a_1, a_1+2 a_2\}}} \sqcup 
\underset{\text{(R-1) extra long}}{\underline{\{2 a_2, 2(a_1+a_2)\}}},
\]
where $a_1=\pi(\alpha_1)=({\alpha_1+\alpha_4})/{2}$ and $a_2=\pi(\alpha_2)=({\alpha_2+\alpha_3})/{2}$.
Also, we get $\Pis=\{a_1,a_2\}$ and $\Deltasp=\{a_2, a_1+a_2\} \sqcup \{a_1,a_1+2a_2\} \subset \pi(\Delta)$.
This is the root system of type $\XNsigma=\XB_2$.
Set $a:=a_2\in\Deltas$.
Then by definition, $a'$ stands for $a_2$ or $2a_2$.
If $a'=a_2$ (resp.~$a'=2a_2$), then $a'$ is of type (R-3) (resp.~(R-1)).
\qed
\end{example}

Suppose that $r=1$ or $2$.
If $\alpha\leq \sigma(\alpha)$, then $-\sigma(\alpha) \leq -\alpha$.
On the other hand, when $r=3$, then the situation is more complicated.
\begin{example} \label{ex:D_4}
The case $(\XX_N,r)=(\XD_4,3)$.
Then the set of all positive roots are given by 
$\Deltap = \{
\alpha_1,\alpha_2,\alpha_3,\alpha_4,\alpha_1+\alpha_2,\alpha_2+\alpha_3, \alpha_2+\alpha_4,
\alpha_1+\alpha_2+\alpha_3, \alpha_2+\alpha_3+\alpha_4,
\alpha_1+\alpha_2+\alpha_4,
\alpha_1+\alpha_2+\alpha_3+\alpha_4,
\alpha_1+2\alpha_2+\alpha_3+\alpha_4 \}$.
The Dynkin diagram automorphism of order $3$ is given by 
$\sigma(\alpha_1)=\alpha_3$, $\sigma(\alpha_2)=\alpha_2$, $\sigma(\alpha_3)=\alpha_4$, and $\sigma(\alpha_4)=\alpha_1$.
Then we have
\[
\pi(\Deltap) = \Deltasp = 
\underset{\text{(R-4) short}}{\underline{\{ a_2, a_1+a_2, a_1+2 a_2 \}}} \sqcup 
\underset{\text{(R-1) long}}{\underline{\{ a_1, a_1+3 a_2, 2 a_1 + 3 a_2 \}}},
\]
where $a_1=\pi(\alpha_2)$ and $a_2=\pi(\alpha_1)$.
Also, we get $\Pis=\{a_1,a_2\}$.
This is the root system of type $\XNsigma=\XG_2$.
For $\alpha\in\Deltap$, it is easy to see that $\alpha\leq \sigma(\alpha) \leq \sigma^2(\alpha)$ if and only if
\[
\alpha \in \{ \alpha_1,\alpha_1+\alpha_2 \} \sqcup
\{ \alpha_2, \alpha_1+\alpha_2+\alpha_3+\alpha_4,\alpha_1+2\alpha_2+\alpha_3+\alpha_4 \}.
\]
On the other hand, for $\alpha\in\Deltap$, $\alpha < \sigma^2(\alpha) < \sigma(\alpha)$ if and only if $\alpha=\alpha_1+\alpha_2+\alpha_3$.
\qed
\end{example}

For $a' \in \pi(\Delta)$, there is $\alpha\in\Delta$ such that $a'=\pi(\alpha)$.
Since $\pi(\alpha)=\pi(\sigma(\alpha))=\pi(\sigma^2(\alpha))$,
we shall define something ``good'' from amongst $\alpha$, $\sigma(\alpha)$ and $\sigma^2(\alpha)$.
\begin{definition}
Let $a'\in\pi(\Delta)$.
Suppose that $a'=\pi(\alpha)$ for some $\alpha\in\Delta$.
\begin{enumerate}
\item Assume that $a'$ is of type (R-1).
We say that {\it $a'$ corresponds to $\alpha$}.
\item Assume that $a$ is of type (R-2) or (R-3).
We say that {\it $a$ corresponds to $\alpha$} if it satisfies $\alpha\leq \sigma(\alpha)$.
\item Assume that $a$ is of type (R-4).
We say that {\it $a$ corresponds to $\alpha$} if $\alpha$ is one of $\pm\alpha_1, \pm(\alpha_1+\alpha_2), \pm(\alpha_2+\alpha_3+\alpha_4)$.
For the notation, see Example~\ref{ex:D_4}.
\end{enumerate}
In these cases, we shall write $a'\ifff \alpha$.
\end{definition}

Suppose that $r=1$ or $2$.
It is easy to see that if $a'\ifff \alpha$, then we have $-a' \ifff-\sigma(\alpha)$.
\begin{example} \label{ex:A_4(2)}
The case $(\XX_N,r)=(\XA_4,2)$.
For each roots $a'\in\pi(\Delta)$, the corresponding roots $\alpha$ are given as follows.
\begin{center}
\scalebox{0.9}{$\displaystyle
\begin{array}{c||c|c|c|c|c|c}
a' & a_2 & a_1+a_2 & a_1 & a_1+2a_2 & 2a_2 & 2(a_1+a_2) \\ \hline
\alpha & \alpha_2 & \alpha_1+\alpha_2 & \alpha_1 & \alpha_1+\alpha_2+\alpha_3 
& \alpha_2+\alpha_3 & \alpha_1+\alpha_2+\alpha_3+\alpha_4 \\ \hline
\\ 
a' & -a_2 & -a_1-a_2 & -a_1 & -a_1-2a_2 & -2a_2 & -2(a_1+a_2)\\ \hline
\alpha & -\alpha_3 & -\alpha_3-\alpha_4 & -\alpha_4 & -\alpha_2-\alpha_3-\alpha_4
& -\alpha_2-\alpha_3 & -\alpha_1-\alpha_2-\alpha_3-\alpha_4 \\ \hline
\end{array}
$}
\end{center}
\qed
\end{example}

\begin{example} \label{ex:D_4(2)}
The case $(\XX_N,r)=(\XD_4,3)$.
For each (R-4) roots $a$, the corresponding roots $\alpha$ are given as follows.
\begin{center}
\scalebox{0.9}{$\displaystyle
\begin{array}{c||c|c|c|c|c|c}
a & a_2 & a_1+a_2 & a_1+2a_2 & -a_2 & -a_1-a_2 & -a_1-2a_2 \\ \hline
\alpha & \alpha_1 & \alpha_1+\alpha_2 & \alpha_2+\alpha_3+\alpha_4
& -\alpha_1 & -\alpha_1-\alpha_2 & -\alpha_2-\alpha_3-\alpha_4 \\ \hline
\end{array}
$}
\end{center}
Note that $-a \ifff -\alpha$, in this case.
\qed
\end{example}

\subsection{Chevalley Bases} \label{sec:chev-basis}
Let $\kappa$ denote the Killing form of $\gcc$.
For each $\alpha\in\Delta$,
there exists a unique $\dualroot{\alpha}\in\hcc$ such that $\kappa(\dualroot{\alpha},H)=\alpha(H)$ for all $H\in\hcc$.
Set $H_{\alpha}:=2\dualroot{\alpha}/(\alpha,\alpha)$ for each $\alpha\in\Delta$.
Here, $(\,,\,)$ is the standard invariant bilinear form on $\hcc^*$ induced by $\kappa$.
One notes that, for a {\it coroot} $\coroot{\alpha}:=2\alpha/(\alpha,\alpha)$ of $\alpha\in\Delta$,
we have $(\beta,\coroot{\alpha})=\beta(H_{\alpha})$ for all $\beta\in\Delta$.

It is known (see \cite[Chapter~2]{Ste} for example) that for any $\alpha\in\Delta$,
we can choose $X_\alpha\in \gcc_\alpha$ so that the set 
\[
\{ X_\alpha \in \gcc_\alpha \mid \alpha\in\Delta \} \sqcup
\{ H_{\alpha_i} \in\hcc \mid i\in \J \}
\]
forms a $\mathbb{Z}$-basis of $\gcc$, the so-called {\it Chevalley basis of $\gcc$}, satisfying
\[
\begin{gathered}
\kappa(X_\alpha,X_{-\alpha})=2/(\alpha,\alpha), \qquad 
[X_\alpha,X_{-\alpha}]=H_{\alpha},
\\
[X_\alpha,X_\beta]=N_{\alpha,\beta}X_{\alpha+\beta}, \quad \text{and} \quad 
N_{\alpha,\beta}=-N_{-\alpha,-\beta}
\end{gathered}
\]
for $\alpha,\beta\in\Delta$ with $\alpha+\beta\in\Delta$.
Here, $N_{\alpha,\beta}$ is the so-called {\it structure constant} which is necessarily an integer.
If $r=2$, then it is easy to see that $N_{\alpha,\beta}$ takes the values
$\pm1$ for all $\alpha,\beta\in\Delta$ with $\alpha+\beta\in\Delta$.
\medskip

We fix a Chevalley basis $\{X_\alpha, H_{\alpha_i}\}_{\alpha\in\Delta,i\in \J}$ of $\gcc$,
and let $\g_{\mathbb{Z}}$ denote the corresposnding $\mathbb{Z}$-form of $\gcc$. A given Dynkin diagram automorphism $\sigma$ induces a Lie algebra isomorphism $\sigma : \gcc \to \gcc$ and
a linear isomorphism $\sigma : \hcc^* \to \hcc^*$.
Moreover, we get an isomorphism $\sigma:\g_{\mathbb{Z}}\to\g_{\mathbb{Z}}$ of Lie algebras over $\mathbb{Z}$ satisfying
\begin{equation} \label{eq:k_a}
\sigma(X_\alpha)=k_\alpha X_{\sigma(\alpha)}
\quad \text{and} \quad
\sigma(H_{\alpha_i})=H_{\sigma(\alpha_i)}
\quad (\alpha\in\Delta,\,i\in\J)
\end{equation}
for some $k_\alpha=\pm1$.
By a suitable replacement (such as $X_\alpha$ with $\pm X_{\alpha}$),
we can re-choose the signs $k_\alpha$ ($\alpha\in\Delta$) satisfying the following (cf.~\cite[Proposition~3.1]{Abe} and \cite[Proposition~2.2]{Mor81}).
\begin{proposition} \label{prp:k_a}
For $\alpha\in\Delta$, we have $k_\alpha=k_{\sigma(\alpha)}$ and
\[
k_\alpha =
\begin{cases}
-1 & \text{if there exists $\beta\in \Delta$ such that $\alpha=\beta+\sigma(\beta)$}, \\
\phantom{-}1 & \text{otherwise.}
\end{cases}
\]
\end{proposition}

In the following, we fix and use the signs $k_\alpha$ ($\alpha\in\Delta$) as in Proposition~\ref{prp:k_a}.
Note that $k_\alpha=-1$ occurs only when $(\XX_N,r)=(\XA_{2\ell},2)$ and $\alpha$ is of type (R-1).
The following result follows from a direct calculation.
\begin{lemma} \label{prp:N-sigma}
For each $\alpha,\beta\in\Delta$ with $\alpha+\beta\in\Delta$,
we have $N_{\sigma(\alpha),\sigma(\beta)} = k_{\alpha+\beta}k_\alpha k_\beta N_{\alpha,\beta}$.
\end{lemma}

\section{Twisted Loop Algebras} \label{sec:aff}
In this section, we work over the field $\Kqq$ of rational numbers.
Let $\xi$ be a primitive $r$ th root of unity in the field $\mathbb{C}$ of complex numbers.
We denote by $\Kqq(\xi)$ the field generated by $\xi$ over $\Kqq$.
If $r=1$ or $2$, then $\xi$ takes value $\pm1$, and hence $\Kqq(\xi)=\Kqq$.
Note that $\xi\notin\Kqq$ occurs only when $(\XX_N,r)=(\XD_4,3)$.

\subsection{Twisted Loop Algebras} \label{sec:tw-loop-alg}
Recall that $\g_{\mathbb{Z}}$ is a $\mathbb{Z}$-form of $\g$, see Section~\ref{sec:chev-basis}.
Set $\gqq:=\g_{\mathbb{Z}}\otimes_{\mathbb{Z}}\Kqq$ and $\hqq:=\hcc \cap \gqq$,
Let $\Rqq:=\Kqq[z^{\pm1}]$ be the ring of Laurent polynomials in the variable $z$ with coefficients in $\Kqq$.
Set $\gidd:=\gqq\otimesqq\Rqq$, which is naturally a Lie algebra over $\Rqq$ (free of rank equal to the dimension of $\gcc$)
in addition to an infinite dimensional Lie algebra over $\Kqq$.
This is the so-called {\it loop algebra}.
In this subsection, we introduce the notion of a twisted version of loop algebras (cf.~\cite[Section~2]{CEGP}).

Let us denote by
\[
\Sqq:=\Kqq(\xi)[z^{\pm\frac{1}{r}}]
\]
the ring of Laurent polynomials in the variable $z^{\frac{1}{r}}$ with coefficients in $\Kqq(\xi)$.
We define the following $\Kqq(\xi)$-algebra automorphism.
\[
\sigmax:\Sqq\longrightarrow \Sqq;
\quad z^{\frac{n}{r}} \longmapsto \xi^{-n} z^{\frac{n}{r}}
\quad (n\in\mathbb{Z}).
\]
We also define the following $\Kqq[z^{\pm\frac{1}{r}}]$-algebra automorphism.
\[
\omegax:\Sqq\longrightarrow\Sqq;
\quad \xi^n \longmapsto \xi^{-n}
\quad (n\in\mathbb{Z}).
\]
Note that if $\xi\in\Kqq$, then this $\omegax$ is trivial.
Then $\Sqq$ is a Galois extension of $\Rqq$ with Galois group $\Gamma$ generated by $\{\sigmax,\omegax\}$
(see \cite{CHR} for Galois extension of rings).
An easy calculation shows that
\[
\Gamma=\langle\sigmax,\,\omegax\rangle
\cong
\begin{cases}
\mathbb{Z}/r\mathbb{Z} & \text{if } \xi\in \Kqq,\\
\mathfrak{S}_3 & \text{if } \xi\notin \Kqq.
\end{cases}
\]
Here, $\mathfrak{S}_3$ is the symmetric group on three letters.
Note that if $r=1$, then $\Sqq=\Rqq$ and $\Gamma$ is trivial.

If $\xi\in\Kqq$, then the group $\Gamma$ generated by the Dynkin diagram automorphism $\sigma\in\Aut_\mathbb{C}(\gcc)$
coincides with $\mathbb{Z}/r\mathbb{Z}$.
For the case when $\xi\notin\Kqq$ (necessarily $(\XX_N,r)=(\XD_4,3)$),
we let $\omega$ be the element of $\Aut_\mathbb{C}(\gcc)$ corresponding to the diagram automorphism (also denoted by $\omega$, see \S\ref{sec:twnotation}).
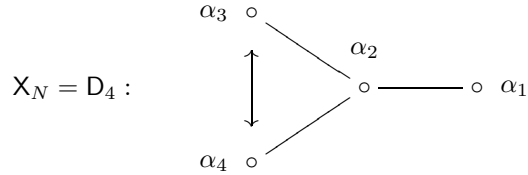
\[
\XX_N=\XD_4:\qquad
\begin{xy}
(20,0)*+{\alpha_1}="2.75", (15,0)*+{\circ}="3.5", (0,0)*+{\circ}="3", (0,5)*+{\alpha_{2}}="33", (-15,10)*+{\circ}="4",
(-20,10)*+{\alpha_{3}}="44", (-15,-10)*+{\circ}="5", (-20,-10)*+{\alpha_{4}}="55", (-15,5)*{}="555", (-15,-5)*{}="-555",
{"3"\SelectTips{cm}{}\ar@{-}"4"}, {"3"\SelectTips{cm}{}\ar@{-}"5"}, {"3.5"\SelectTips{cm}{}\ar@{-}"3"}, {"-555"\SelectTips{cm}{}\ar@{<->}"555"}
\end{xy}
\]
\[
\omega(\alpha_1)=\alpha_1,\quad\omega(\alpha_2)=\alpha_2,\quad\omega(\alpha_3)=\alpha_4,\quad\omega(\alpha_4)=\alpha_3.
\]
Then one sees the set $\{\sigma,\omega\}$ generates the group $\mathfrak{S}_3\,(\cong\Gamma)$.

The following is easy to see.
\begin{lemma} \label{prp:N_omega}
Suppose that $\xi\notin\Kqq$.
For each $\alpha,\beta\in\Delta$ with $\alpha+\beta\in\Delta$,
we have $N_{\omega(\alpha),\omega(\beta)}=N_{\alpha,\beta}$.
\end{lemma}

For simplicity, we set $\omega=\id_{\gcc}$ if $\xi\in\Kqq$.
Then we have $\Gamma=\langle\sigmax,\omegax\rangle\cong\langle\sigma,\omega\rangle$ for any case.
The group $\Gamma$ also acts on $\gqq\otimes_{\Kqq} \Sqq\,(\cong\gidd\otimes_{\Rqq} \Sqq)$ via
\[
\sigma(X\otimesqq s)=\sigma(X)\otimesqq\sigmax(s)
\quad \text{and} \quad 
\omega(X\otimesqq s)=\omega(X)\otimesqq\omegax(s),
\]
where $X\in\gqq$ and $s\in\Sqq$.
\begin{definition}
We let $\gtw$ denote the fixed-point subalgebra of $\gidd\otimes_{\Rqq}\Sqq$ under $\Gamma$,
and call it the {\it twisted loop algebra} defined over $\Kqq$.
\end{definition}

As in \cite[Chapter~7]{Kac}, the $\Kqq$-vector space
\[
\gtwh:=\gtw\oplus\Kqq\cc\oplus\Kqq\dd
\]
forms a Lie algebra over $\Kqq$ by letting
\[
\begin{gathered}\phantom{a}
[X\otimesqq\nu z^{\frac{n}{r}},Y\otimesqq\mu z^{\frac{m}{r}}] 
=[X,Y]\otimesqq\nu\mu z^{\frac{m+n}{r}}+\kappa(X,Y)\,n\,\delta_{m+n,0}\,\cc,
\\
[\cc,\nu z^{\frac{n}{r}}]=[\cc,\dd]=0,
\quad\text{and}\quad
[\dd,X\otimesqq\nu z^{\frac{n}{r}}]=nX\otimesqq\nu z^{\frac{n}{r}}
\end{gathered}
\]
for $X,Y\in\gqq$, $\nu,\mu\in\Kqq(\xi)$, $n,m\in\mathbb{Z}$.
Here, $\cc$ (resp.~$\dd$) is the so-called central element (resp.~degree derivation).
\begin{remark}\label{rem:g(n)}
If $r\neq3$, then it is easy to see that
\[
\gtw=\bigoplus_{n\in\mathbb{Z}}\gqq[\bar{n}]\otimesqq \Kqq z^{\frac{n}{r}},
\]
where $\gqq[\bar{n}]:=\{X\in\gqq\mid\sigma(X)=\xi^{n}X\}$.
\end{remark}

\subsection{Special Elements in Twisted Loop Algebras}
As in \cite[\S1]{Mor81}, we define the following subset $\Omega$ of the set $\pi(\Deltap)\times\mathbb{Z}$.
If $(\XX_N,r)=(\XA_{2\ell},2)$, then
\[
\Omega:=\{(a,n)\mid a\in\Deltas,\,n\in\mathbb{Z}\}\sqcup\{(2a,n)\mid a\in\Deltasshort,\,n\in2\mathbb{Z}+1\}.
\]
Otherwise,
\[
\Omega:=\{(a,n)\mid a\in\Deltasshort,\,n\in\mathbb{Z}\}\sqcup\{(a,n)\mid a\in\Deltaslong,\,n\in r\mathbb{Z}\}.
\]
Here, $2\mathbb{Z}+1:=\{2n+1\in\mathbb{Z}\mid n\in\mathbb{Z}\}$ and $r\mathbb{Z}:=\{rn\in\mathbb{Z}\mid n\in\mathbb{Z}\}$.

\begin{definition} \label{def:X_pm}
Take $(a',n)\in\Omega$ with $a'\ifff \alpha\in\Delta$.
We define $\tilde{X}_{(a',n)}\in\gqq\otimesqq \Sqq$ as follows.
\begin{itemize}
\item $\tilde{X}_{(a',n)}:=X_{\alpha}\otimesqq z^{\frac{n}{r}}$ if $a'$ is of type (R-1).
\item $\tilde{X}_{(a,n)}:=X_{\alpha}\otimesqq z^{\frac{n}{r}}+X_{\sigma(\alpha)}\otimesqq\xi^{-n}z^{\frac{n}{r}}$ if $a$ is of type (R-2) or (R-3).
\item $\tilde{X}_{(a,n)}:=X_{\alpha}\otimesqq z^{\frac{n}{r}}+X_{\sigma(\alpha)}\otimesqq\xi^{-n}z^{\frac{n}{r}}+X_{\sigma^2(\alpha)}\otimesqq\xi^{-2n}z^{\frac{n}{r}}$ if $a$ is of type (R-4).
\end{itemize}
\end{definition}

\begin{lemma}
For each $(a',n)\in\Omega$, the element $\tilde{X}_{(a',n)}$ belongs to $\gtw$.
\end{lemma}
\begin{proof}
Suppose that $a'\ifff\alpha\in\Delta$.
First, by Proposition~\ref{prp:k_a}, $\sigma(\tilde{X}_{(a',n)})$ is given as follows.
\begin{itemize}
\item $k_\alpha X_{\alpha}\otimesqq\xi^{-n}z^{\frac{n}{r}}$ if $a'$ is of type (R-1).
\item $X_{\sigma(\alpha)}\otimesqq\xi^{-n}z^{\frac{n}{r}}+X_{\sigma^2(\alpha)}\otimesqq\xi^{-2n}z^{\frac{n}{r}}$ if $a$ is of type (R-2) or (R-3).
\item $X_{\sigma(\alpha)}\otimesqq\xi^{-n}z^{\frac{n}{r}}+X_{\sigma^2(\alpha)}\otimesqq\xi^{-2n}z^{\frac{n}{r}}+X_{\sigma^3(\alpha)}\otimesqq\xi^{-3n}z^{\frac{n}{r}}$ if $a$ is of type (R-4).
\end{itemize}
Thus, if $(\XX_N,r)\neq(\XA_{2\ell},2)$, then we see that $\sigma(\tilde{X}_{(a',n)})=\tilde{X}_{(a',n)}$.
Suppose that $(\XX_N,r)=(\XA_{2\ell},2)$ and $a'$ is of type (R-1).
In this case, we have $k_\alpha=-1$ by Proposition~\ref{prp:k_a}.
Since $\xi=-1$ and $n\in2\mathbb{Z}+1$, we conclude that $\sigma(\tilde{X}_{(a',n)})=\tilde{X}_{(a',n)}$.

Next, suppose that $(\XX_N,r)=(\XD_4,3)$ and $\xi\notin\Kqq$.
If $a'=a$ is of type (R-1), then obviously we get $\omega(\tilde{X}_{(a,n)})=\tilde{X}_{(a,n)}$.
Thus, in the following, we treat the case when $a'=a$ is of type (R-4) and $a\in\Deltasp$.
We use the notations in Example~\ref{ex:D_4}.
If $a=a_2$ (i.e., $\alpha=\alpha_1$),
then by definition, we get 
$\tilde{X}_{(a_2,n)}
=
X_{\alpha_1}\otimesqq z^{\frac{n}{r}}+X_{\alpha_3}\otimesqq\xi^{-n}z^{\frac{n}{r}}+X_{\alpha_4}\otimesqq\xi^{-2n}z^{\frac{n}{r}}$,
and hence
\[
\omega(\tilde{X}_{(a_2,n)})
=
X_{\alpha_1}\otimesqq z^{\frac{n}{r}}+X_{\alpha_4}\otimesqq\xi^{n}z^{\frac{n}{r}}+X_{\alpha_3}\otimesqq\xi^{2n}z^{\frac{n}{r}}
=
\tilde{X}_{(a_2,n)}.
\]
For $a=a_1+a_2$, the situation is similar, and we obtain $\omega(\tilde{X}_{(a_1+a_2,n)})=\tilde{X}_{(a_1+a_2,n)}$.
If $a=a_1+2a_2$, then the corresponding root is $\alpha\ifff\alpha_2+\alpha_3+\alpha_4\in\Delta$,
and hence $\tilde{X}_{(a_1+2a_2,n)}$ is given as follows.
\[
X_{\alpha_2+\alpha_3+\alpha_4}\otimesqq z^{\frac{n}{r}}
+X_{\alpha_2+\alpha_4+\alpha_1}\otimesqq\xi^{-n}z^{\frac{n}{r}}+X_{\alpha_2+\alpha_1+\alpha_3}\otimesqq\xi^{-2n}z^{\frac{n}{r}}.
\]
Thus, $\omega(\tilde{X}_{(a_1+2a_2,n)})$ is given by
\[
X_{\alpha_2+\alpha_4+\alpha_3}\otimesqq z^{\frac{n}{r}}
+X_{\alpha_2+\alpha_3+\alpha_1}\otimesqq\xi^{n}z^{\frac{n}{r}}+X_{\alpha_2+\alpha_1+\alpha_4}\otimesqq\xi^{2n}z^{\frac{n}{r}}.
\]
Since $\xi^{n}=\xi^{-2n}$ and $\xi^{2n}=\xi^{-n}$, we have $\omega(\tilde{X}_{(a_1+2a_2,n)})=\tilde{X}_{(a_1+2a_2,n)}$.
\end{proof}

As a subspace of $\gtwh$, we set
\[
\htwh:=\hsigmaqq\oplus\Kqq\cc\oplus\Kqq\dd,
\]
where $\hsigmaqq:=\gsigmacc\cap\hqq$.
For each $p\in\Jsigma$, we regard $a_p$ as an element of the linear dual space $\htwh^*$ of $\htwh$ by letting $a_p(\cc)=a_p(\dd)=0$.
We define $\deltaa\in\htwh^*$ so that $\deltaa(\hsigmaqq)=0$, $\deltaa(\cc)=0$, and $\deltaa(\dd)=1$.

Let $\Deltash$ be the root system of $\gtwh$ with respect to $\htwh$.
One sees that the set of all ``real'' roots $\Deltashr$ in $\Deltash$ is given as follows (cf.~Theorem~\ref{prp:isom}).
\[
\Deltashr=\{a'+n\deltaa\in\htwh^*\mid(a',n)\in\Omega\}.
\]
\begin{remark} \label{rem:n-even}
If $r=1$, then $\Deltas=\Delta$ and $\Deltashr=\{\alpha+n\deltaa\mid\alpha\in\Delta,n\in\mathbb{Z}\}$.
Suppose that $\XX_N=A_{2\ell}$ with $r=2$.
Let $\hat{a}=a'+n\deltaa\in\Deltashr$ be a real root with $a'\in\pi(\Delta)$ and $n\in\mathbb{Z}$.
Then by Notation~\ref{not:a'}, $a'$ stands for $a$ or $2a$ for some $a\in\Deltas$.
However, if $n$ is even, then we can conclude that $a'=a$.
\end{remark}

\subsection{Affine Kac-Moody Algebras and Twisted Loop Algebras} \label{sec:KMtab}
Recall that $\gqq=\gqq(\XX_N)$ is a finite-dimensional split simple Lie algebra of type $\XX_N$ defined over $\Kqq$.
Let $A=(a_{ij})_{i,j\in\J}$ denote the Cartan matrix of type $\XX_N$, and
let $\{h_i,e_i,f_i\}$ be the Chevalley generators of $\gqq$, see Appendix~\ref{sec:KM-alg}.
In the following, we let $\gsigmaqq$ (resp.~$\hsigmaqq$) denote the fixed-point subalgebra of $\gqq$ (resp.~$\hqq$) under $\sigma$,
as in Section~\ref{sec:twnotation}.

Recall that $\Jsigma$ is the set of all equivalence classes for $\sigma$ on $\J$.
For each $p\in\Jsigma$, we define the following elements $H_p,E_p,F_p$ in $\gsigmaqq$:
If $a_{ij}=2\delta_{i,j}$ for all $i,j\in p$, then
\[
H_p:=\sum_{i\in p}h_i, \qquad 
E_p:=\sum_{i\in p}e_i, \qquad
F_p:=\sum_{i\in p}f_i.
\]
Otherwise,
\[
H_p:=2\,\sum_{i\in p}h_i, \qquad
E_p:=\sum_{i\in p}e_i, \qquad
F_p:=2\sum_{i\in p}f_i.
\]
Note that the latter case occurs when $(\XX_N,r)=(\XA_{2\ell},2)$.

One sees that $\As:=(a_q(H_p))_{p,q\in\Jsigma}$ is the Cartan matrix of type $\XNsigma$.
For $\lambda\in\hsigmasqq$, we let $\dualroot{\lambda}$ be an element of $\hsigmaqq$ such that
$\kappasqq(\dualroot{\lambda},H)=\lambda(H)$ for all $H\in\hsigmaqq$.
Here, $\kappasqq:=\kappa|_{\gsigmaqq\times\gsigmaqq}$ is the restricted Killing form of $\gsigmaqq$.
Then one sees that $H_p=2\dualroot{a_p}/{(a_p,a_p)}$.
In this setting, one can show the following.
\begin{lemma}
Let $\gqq(\XNsigma)$ be a finite-dimensional simple Lie algebra of type $\XNsigma$ defined over $\Kqq$ with
 Chevalley generators $\{h_p^\sigma,e_p^\sigma,f_p^\sigma\}$.
Then the map defined by $h_p^\sigma\mapsto H_p$, $e_p^\sigma\mapsto E_p$, and $f_p^\sigma\mapsto F_p$ ($p\in\Jsigma$)
gives an isomorphism $\gqq(\XNsigma)\to\gsigmaqq$ of Lie algebras over $\Kqq$.
\end{lemma}

If $\xi\notin\Kqq$ (i.e., $(\XX_N,r)=(\XD_4,3)$), then we let $-a_0:=a_1+2a_2$.
If $\xi\in\Kqq$, then we let $-a_0\in\hsigmasqq$ be the highest weight of the $\gqq[\bar0]$-module $\gqq[-\bar1]$
(for the notation, see Remark~\ref{rem:g(n)}).
By \cite[Propositions~7.9~and~7.10]{Kac}, we have the following.
\medskip

\begin{center}
\scalebox{0.95}{$\displaystyle
\begin{array}{l||c|c|c}
\XX_N        &r&-a_0 \ifff \text{the corresponding root}                                        &\text{Type}  \\ \hline\hline
\XA_{2\ell-1}&2&a_1+2a_2+\cdots+2a_{\ell-1}+a_\ell \ifff \alpha_1+\cdots+\alpha_{2\ell-2}       &\text{(R-2)} \\ \hline
\XA_{2\ell}  &2&2a_1+\cdots+2a_\ell \ifff \alpha_1+\cdots+\alpha_{2\ell}                        &\text{(R-1)} \\ \hline
\XD_{\ell+1} &2&a_1+\cdots+a_\ell \ifff \alpha_1+\cdots+\alpha_{\ell}                           &\text{(R-2)} \\ \hline
\XE_{6}      &2&2a_1+3a_2+2a_3+a_4 \ifff \alpha_1+2\alpha_2+2\alpha_3+\alpha_4+\alpha_5+\alpha_6&\text{(R-2)} \\ \hline
\XD_{4}      &3&a_1+2a_2 \ifff \alpha_1+\alpha_2+\alpha_3                                       &\text{(R-4)} \\ \hline
\end{array}
$}
\end{center}
\medskip

Set $H_0:=2\dualroot{a_0}/(a_0,a_0)$.
We choose $E_0\in\gqq(\XX_N)_{a_0}$ and $F_0\in\gqq(\XX_N)_{-a_0}$ so that
$[E_0,F_0]=H_0$ and $E_0\otimesqq z^{\frac{1}{r}}, F_0\otimesqq z^{-\frac{1}{r}}$ belong to $\gtwh$.
Set $\Jsigmah:=\{0\}\sqcup\Jsigma$.
Then $\Ash:=(a_q(H_p))_{p,q\in\Jsigmah}$ forms a symmetrizable Cartan matrix.

For each $p\in\Jsigmah$, we define elements $\hat{H}_p\in\gtwh$ and $\hat{a}_p\in\htwh^*$ so that
\[
\hat{H}_p:=H_p\otimesqq1+\delta_{p,0}\,\frac{2}{(a_0,a_0)}\cc, \qquad
\hat{a}_p:=a_p+\delta_{p,0}\,\deltaa.
\]
Then the triple 
$(\htwh,\{\hat{a}_p\}_{p\in\Jsigmah},\{\hat{H}_p\}_{p\in\Jsigmah})$ is a realization of the generalized Cartan matrix $\Ash$ of type $\XX_N^{(r)}$.
For each $p\in\Jsigmah$, we also define elements in $\gtwh$ as follows.
\[
\hat{E}_p:=E_p\otimesqq z^{\frac{1}{r}\delta_{p,0}}, \qquad 
\hat{F}_p:=F_p\otimesqq z^{-\frac{1}{r}\delta_{p,0}}.
\]
By definition, $\hat{H}_p,\hat{E}_p$ and $\hat{F}_p$ belong to $\gtwh$.

In the following, we use the label $\XX_N^{(r)}$ as in Kac's list \cite[TABLE~Aff $r$]{Kac} ($r=1$, $2$, or $3$).
Then we have the following result.
\begin{theorem} \label{prp:isom}
Let $\ghatqq(\XX_N^{(r)})$ be the affine Kac-Moody algebra of type $\XX_N^{(r)}$ defined over $\Kqq$
with Chevalley generators $\{\hat{h}_p,\hat{e}_p,\hat{f}_p\}$.
Then the map
\[
\varphi:\ghatqq(\XX_N^{(r)})\longrightarrow\gtwh; \quad
\hat{h}_p\mapsto\hat{H}_p,\,\,\hat{e}_p\mapsto\hat{E}_p,\,\,\hat{f}_p\mapsto\hat{F}_p
\quad (p\in\Jsigmah)
\]
is an isomorphism of Lie algebras over $\Kqq$.
\end{theorem}
\begin{proof}
If $\xi\in\Kqq$, then the proof is essentially the same as Kac's \cite[Theorem~8.3]{Kac} (see also \cite[Theorem~7.4]{Kac}).
If $\xi\notin\Kqq$, then by comparing the dimensions of real/imaginary root spaces concretely, we can also show that $\varphi$ is bijective.
\end{proof}

Since the center of the derived subalgebra $\ghatqq(\XX_N^{(r)})'$ of $\ghatqq(\XX_N^{(r)})$ is one-dimensional,
the sequence $0\to\Kqq\to\ghatqq(\XX_N^{(r)})'\overset{\varphi}{\longrightarrow}\gtw\to0$ is exact.
Thus, we have the following.
\begin{corollary} \label{prp:centext}
$\ghatqq(\XX_N^{(r)})'/\Kqq\cong\gtw$.
\end{corollary}

\subsection{Chevalley Pairs} \label{sec:chev-pair}
We let $\{\hat{h}_p,\hat{e}_p,\hat{f}_p\}$ denote the Chevalley generators of $\ghatqq(\XX_N^{(r)})$ as before.
Let $\hhatqq$ denote the Cartan subalgebra of $\ghatqq(\XX_N^{(r)})$ generated by $\hat{h}_p$'s.
We identify $\Deltashr$ with the set of all real roots of $\ghatqq(\XX_N^{(r)})$ with respect to $\hhatqq$ via the isomorphism $\varphi$
given in Theorem~\ref{prp:isom}.
In this subsection, we translate the notion of Chevalley pairs (see Definition~\ref{def:chev-pair}) of $\ghatqq(\XX_N^{(r)})$ into $\gtwh$.

For a real root $\hat{a}\in\Deltashr$, we define $H_{\hat{a}}\in\hhatqq$ so that
$H_{\hat{a}}=w(\hat{h}_p)$ for some $w\in\WWh$ and $p\in\Jsigmah$ satisfying $\hat{a}=w(\hat{a}_p)$.
Here, $\WWh$ is the Weyl group of $\ghatqq(\XX_N^{(r)})$ with respect to $\hhatqq$.
Let $(\,,\,)$ be the standard invariant form induced from the Killing form $\kappahatqq$ of $\ghatqq(\XX_N^{(r)})$.
As before, for each $\hat{a}\in\Deltashr$, we define $\dualroot{\hat{a}}\in\hhatqq$ so that
$\kappahatqq(\dualroot{\hat{a}},h)=\hat{a}(h)$ for all $h\in\hhatqq$.
Then we have $H_{\hat{a}}=2\dualroot{\hat{a}}/(\hat{a},\hat{a})$ and $\hat{\lambda}(H_{\hat{a}})=2(\hat{\lambda},\hat{a})/(\hat{a},\hat{a})$ for all $\hat{\lambda}\in\hhatqq^*$.
\begin{lemma} \label{prp:invform}
For $\hat{a}=a'+n\deltaa\in\Deltashr$ with $a'\ifff\alpha\in\Delta$, we get
\[
(\hat{a},\hat{a}) =
\begin{cases}
(\alpha,\alpha)  &\text{if $a'$ is of type {\rm (R-1)},} \\
(\alpha,\alpha)/2&\text{if $a$ is of type {\rm (R-2)},} \\
(\alpha,\alpha)/4&\text{if $a$ is of type {\rm (R-3)},} \\
(\alpha,\alpha)/3&\text{if $a$ is of type {\rm (R-4)}.}
\end{cases}
\]
\end{lemma}
\begin{proof}
First, note that $(\hat{a},\hat{a})=(a',a')=\kappasqq(\dualroot{a'},\dualroot{a'})$.
Since $(\,,\,)$ is invariant under the action of the Weyl group $\WWh$,
it is enough to show the claim in the case of $a'=a_p$ and $\alpha=\alpha_i$ for some $p\in\Jsigmah$ and $i\in p$.
We may suppose that $p\neq0$.
If $a_p$ is of type (R-1), then the claim is trivial.
If $a_p$ is of type (R-2), then we see $(\alpha_i,\sigma(\alpha_i))=0$.
Note that $(\sigma(\alpha_i),\sigma(\alpha_i))=(\alpha_i,\alpha_i)$.
We have $\dualroot{a_p}=\tfrac{1}{2}(\dualroot{\alpha_i}+\dualroot{\sigma(\alpha_i)})$.
Indeed, for any $H\in\hsigma$,
\[
\kappasqq(\tfrac{1}{2}(\dualroot{\alpha_i}+\dualroot{\sigma(\alpha_i)}),H)
=\tfrac{1}{2}(\alpha_i(H)+\sigma(\alpha_i)(H))=\alpha_i(H).
\]
Hence, we conclude that $(a_p,a_p)=\kappasqq(\dualroot{a_p},\dualroot{a_p})=\tfrac{1}{2}(\alpha_i,\alpha_i)$.
If $a_p$ is of type (R-3), then we see that $\dualroot{a_p}=\tfrac{1}{2}(\dualroot{\alpha_i}+\dualroot{\sigma(\alpha_i)})$, as before.
In this case, notice that $\XX_N=\XA_N$ and $i-\sigma(i)=\pm1$.
Hence, $(\alpha_i,\alpha_i)=(\sigma(\alpha_i),\sigma(\alpha_i))=-2(\alpha_i,\sigma(\alpha_i))$.
By using this, we have
\begin{eqnarray*}
(a_p,a_p) 
&=& 
\kappasqq(\dualroot{a_p},\dualroot{a_p})
\\ &=&
\tfrac{1}{4}((\alpha_i,\alpha_i)+(\sigma(\alpha_i),\alpha_i)+(\alpha_i,\sigma(\alpha_i))+(\sigma(\alpha_i),\sigma(\alpha_i)))
\\ &=&
\tfrac{1}{4}(\alpha_i,\alpha_i).
\end{eqnarray*}
If $a_p$ is of type (R-4), then one sees that $\dualroot{a_p}=\frac{1}{3}(\dualroot{\alpha}+\dualroot{\sigma(\alpha)}+\dualroot{\sigma^2(\alpha)})$.
We shall use the notations in Example~\ref{ex:D_4}.
Then $(a_p,a_p)=\frac{1}{9}((\alpha_1,\alpha_1)+(\alpha_3,\alpha_3)+(\alpha_4,\alpha_4))=\frac{1}{3}(\alpha_1,\alpha_1)$.
Thus, we are done.
\end{proof}

\begin{proposition} \label{prp:H_a}
For any $\hat{a}=a'+n\deltaa\in\Deltashr$ with $a'\ifff\alpha\in\Delta$, we have
\[
\varphi(H_{\hat{a}}) =
\begin{cases}
H_{\alpha}\otimesqq1+\tfrac{2n}{(\alpha,\alpha)}\cc & \text{if $a'$ is of type {\rm (R-1)}}, \\
(H_{\alpha}+H_{\sigma(\alpha)})\otimesqq1+\tfrac{4n}{(\alpha,\alpha)}\cc & \text{if $a$ is of type {\rm (R-2)}},\\
2(H_{\alpha}+H_{\sigma(\alpha)})\otimesqq1+\tfrac{8n}{(\alpha,\alpha)}\cc & \text{if $a$ is of type {\rm (R-3)}}, \\
(H_{\alpha}+H_{\sigma(\alpha)}+H_{\sigma^2(\alpha)})\otimesqq1+\tfrac{6n}{(\alpha,\alpha)}\cc & \text{if $a$ is of type {\rm (R-4)}}.
\end{cases}
\]
\end{proposition}
\begin{proof}
Let $\hat{\lambda}\in\hat{\h}^*$. 
Suppose that $\hat{\lambda}$ corresponds to $\lambda+\nu\cc^*+\mu\deltaa\in\htwh^*$ for some $\nu,\mu\in\Kqq(\xi)$
and $\lambda\in\hsigmasqq$,
where $\cc^*$ is the linear dual element of $\cc$.
Then by Lemma~\ref{prp:invform}, we have
\[
\frac{2(\hat{\lambda},\hat{a})}{(\hat{a},\hat{a})}=\frac{2((\lambda,a')+\nu n)}{(\hat{a},\hat{a})} =
\begin{cases}
\frac{2(\lambda,a')}{(\alpha,\alpha)}+\frac{2\nu n}{(\alpha,\alpha)} & \text{if $a'$ is of type (R-1),} \\
\frac{4(\lambda,a)}{(\alpha,\alpha)}+\frac{4\nu n}{(\alpha,\alpha)} & \text{if $a$ is of type (R-2),} \\
\frac{8(\lambda,a)}{(\alpha,\alpha)}+\frac{8\nu n}{(\alpha,\alpha)} & \text{if $a$ is of type (R-3),} \\ 
\frac{6(\lambda,a)}{(\alpha,\alpha)}+\frac{6\nu n}{(\alpha,\alpha)} & \text{if $a$ is of type (R-4).}
\end{cases}
\]
We shall use the notations in the proof of Lemma~\ref{prp:invform}.
Then we see that $\dualroot{a'}=\dualroot{\alpha}$
(resp.~$\tfrac{1}{2}(\dualroot{\alpha}+\dualroot{\sigma(\alpha)})$,
$\tfrac{1}{3}(\dualroot{\alpha}+\dualroot{\sigma(\alpha)}+\dualroot{\sigma^2(\alpha)})$)
if $a'$ is of type (R-1) (resp.~$a$ is of (R-2 or 3), (R-4)).
Since $H_{\alpha}=\tfrac{2}{(\alpha,\alpha)}\dualroot{\alpha}$, we get
\[
\frac{(\lambda,a')}{(\alpha,\alpha)}
=\frac{\lambda(\dualroot{a'})}{(\alpha,\alpha)}
=\begin{cases}
\frac{1}{2}\lambda(H_{\alpha}) & \text{if $a'$ is of type (R-1)}, \\
\frac{1}{4}\lambda(H_{\alpha}+H_{\sigma(\alpha)}) & \text{if $a$ is of type (R-2) or (R-3)}, \\
\frac{1}{6}\lambda(H_{\alpha}+H_{\sigma(\alpha)}+H_{\sigma^2(\alpha)}) & \text{if $a$ is of type (R-4)}.
\end{cases}
\]
On the other hand, we see that $\hat{\lambda}(\cc)=\nu$.
Hence, the claim follows.
\end{proof}

\begin{notation} \label{not:e_a}
Suppose that $(\XX_N,r)\neq(\XD_4,3)$.
For $a\in\Deltas$, we set
\[
\xi_a:=
\begin{cases}
1 & \text{if } a\in\Deltasp, \\
\xi & \text{otherwise}.
\end{cases}
\]
If $a$ is of type (R-3), then we set
\[
\epsilon_a:=
\begin{cases}
1 & \text{if } a\in\Deltasp, \\
2 & \text{otherwise}.
\end{cases}
\]
Note that $\xi_a\xi_{-a}=\xi$ and $\epsilon_a\epsilon_{-a}=2$.
\end{notation}

Using the isomorphism $\varphi:\ghatqq(\XX_N^{(r)})\to\gtwh$ of Lie algebras given in Theorem~\ref{prp:isom},
we define $X_{\hat{a}}\in\ghatqq(\XX_N^{(r)})$ for each $\hat{a}=a'+n\deltaa\in\Deltashr$ as follows.
\begin{equation} \label{eq:tildeX_hata}
\varphi(X_{\hat{a}}) = 
\begin{cases}
\tilde{X}_{(a',n)} & \text{if $a'$ is of type (R-1)}, \\
\xi_a^{-n}\tilde{X}_{(a,n)} & \text{if $a$ is of type (R-2)}, \\
\epsilon_a\xi_a^{-n}\tilde{X}_{(a,n)} & \text{if $a$ is of type (R-3)}, \\
\tilde{X}_{(a,n)} & \text{if $a$ is of type (R-4)}.
\end{cases}
\end{equation}
For the notation $\tilde{X}_{(a',n)}$, see Definition~\ref{def:X_pm}.
\begin{proposition} \label{prp:1}
For $\hat{a}\in\Deltashr$,
the pair $(X_{\hat{a}},X_{-\hat{a}})$ forms a Chevalley pair of $\ghatqq(\XX_N^{(r)})$.
\end{proposition}
\begin{proof}
For simplicity, we let $\tilde{Y}_{\hat{a}}$ denote the right hand side of \eqref{eq:tildeX_hata}.
We will show that the pair $(\tilde{Y}_{\hat{a}},\tilde{Y}_{-\hat{a}})$ satisfies the conditions of Chevalley pairs (Definition~\ref{def:chev-pair}).
To see this, it is enough to show that $[\tilde{Y}_{\hat{a}},\tilde{Y}_{-\hat{a}}]=\varphi(H_{\hat{a}})$ for any $\hat{a}=a'+n\deltaa\in\Deltashr$.
We assume that $a'\ifff \alpha\in\Delta$.
Then $\tilde{Y}_{-\hat{a}}$ is explicitly given as follows.
\begin{itemize}
\item $X_{-\alpha}\otimesqq z^{-\frac{n}{r}}$ if $a'$ is of type (R-1).
\item $X_{-\sigma(\alpha)}\otimesqq\xi_{-a}^{n}z^{-\frac{n}{r}}+X_{-\alpha}\otimesqq\xi_{-a}^{n}\xi^{n}z^{-\frac{n}{r}}$ if $a$ is of type (R-2).
\item $\epsilon_{-a}(X_{-\sigma(\alpha)}\otimesqq\xi_{-a}^{n}z^{-\frac{n}{r}}+X_{-\alpha}\otimesqq\xi_{-a}^{n}\xi^{n}z^{-\frac{n}{r}})$ if $a$ is of type (R-3).
\item $X_{-\alpha}\otimesqq z^{-\frac{n}{r}}+X_{-\sigma(\alpha)}\otimesqq\xi^{n}z^{-\frac{n}{r}}+X_{-\sigma^2(\alpha)}\otimesqq\xi^{2n}z^{-\frac{n}{r}}$ if $a$ is of type (R-4).
\end{itemize}
Since $[X_\alpha,X_{-\alpha}]=H_\alpha$ and $\kappa_\Kqq(X_\alpha,X_{-\alpha})=2/(\alpha,\alpha)$, we have
\[
[X_\alpha\otimesqq\xi_a^{-n}z^{\frac{n}{r}},\,X_{-\alpha}\otimesqq\xi_{-a}^n\xi^nz^{-\frac{n}{r}}]
=
H_\alpha\otimesqq\xi_a^{-n}\xi_{-a}^n\xi^n+n\tfrac{2}{(\alpha,\alpha)}\cc.
\]
Then by Proposition~\ref{prp:H_a}, the claim follows.
\end{proof}

\section{Twisted Loop Groups} \label{sec:tw_loop}
Throughout the rest of the paper, we work over a field $\K$ of characteristic not equal to $2$ (resp.~$3$)
when we consider the case $r=2$ (resp.~$r=3$).
We take and fix a primitive $r$ th root of unity $\xi$ in a fixed algebraic closure $\overline{\K}$ of $\K$,
and denote by $\K(\xi)$ the subfield of $\overline{\K}$ generated by $\xi$ over $\K$.
Note that $\xi\notin\K$ may occur only when $(\XX_N,r)=(\XD_4,3)$.

\subsection{Twisted Loop Groups} \label{sec:tw_chev}
As in Section~\ref{sec:tw-loop-alg}, we put
\[
\Rkk:=\K[z^{\pm1}]\quad \text{and} \quad \Skk:=\K(\xi)[z^{\pm\frac{1}{r}}].
\]
Also, we define automorphisms $\sigmax\in\Aut_{\K(\xi)}(\Skk)$ and $\omegax\in\Aut_{\K[z^{\pm\frac{1}{r}}]}(\Skk)$
satisfying $\sigmax(z^{\frac{n}{r}})=\xi^{-n}z^{\frac{n}{r}}$ and $\omegax(\xi^n)=\xi^{-n}$
for all $n\in\mathbb{Z}$.
Then as before, we have $\Gamma=\Gal(\Skk/\Rkk)=\langle\sigmax,\omegax\rangle\cong\langle\sigma,\omega\rangle$.
We denote by $\Skk^\omegax=\K[z^{\pm\frac{1}{r}}]$ the fixed-point subalgebra of $\Skk$ under $\omegax$.

Let $\Gcd$ be the Chevalley-Demazure simply-connected group scheme of type $\XX_N$ defined over $\mathbb{Z}$,
and let $\G$ denote the base change of $\Gzz$ to $\K$.
The $\Rkk$-valued points $\Gidd:=\G(\Rkk)$ of $\G$ is the so-called {\it loop group}.
In this section, using Abe's construction \cite{Abe}, we introduce the notion of a twisted version of loop groups.
\medskip

For $\alpha\in\Delta$ and $s\in \Skk$, we let $x_\alpha(s)$ denote the associated {\it unipotent element} (cf.~\eqref{eq:unip-elem}) of $\G(\Skk)$.
Since $\Skk$ is an Euclidean domain, we see that the group $\G(\Skk)$ is generated by these $x_\alpha(s)$'s, see \cite[Chapter~8]{Ste}.
As in Section~\ref{sec:tw-loop-alg}, we define the following action of $\Gamma$ on $\G(\Skk)$.
For $\alpha\in\Delta$ and $s\in \Skk$,
\[
\sigma(x_\alpha(s))=x_{\sigma(\alpha)}(k_\alpha\,\sigmax(s))
\quad \text{and} \quad
\omega(x_\alpha(s))=x_{\omega(\alpha)}(\omegax(s)).
\]
For the notation $k_\alpha$, see \eqref{eq:k_a}.

\begin{definition}\label{def:tw_loop_grp} 
We let $\Gtw$ denote the fixed-point subgroup of $\G(\Skk)$ under $\Gamma$,
and call it the {\it twisted loop group} associated to $\XX_N^{(r)}$ defined over $\K$.
\end{definition}

As in \cite[\S2]{Abe}, we put
\begin{equation} \label{eq:mathfrakA}
\frakA:=\{
\chi=(\chi^{(1)},\chi^{(2)})\in\Skk\times\Skk\mid 
\chi^{(1)}\,\sigmax(\chi^{(1)})=\chi^{(2)}+\sigmax(\chi^{(2)})
\}.
\end{equation}
For $a\in\Deltas$ with $a\ifff\alpha\in\Delta$, $u\in\Rkk$, $s\in\Skk$, $\tilde{s}\in\Skk^\omegax$,
and $\chi=(\chi^{(1)},\chi^{(2)})\in\frakA$, we define
\begin{description}
\item[(G-1)] $\xsx_{a}(u):=x_{\alpha}(u)$ if $a$ is of type (R-1).
\item[(G-2)] $\xsx_{a}(s):=x_{\alpha}(s)x_{\sigma(\alpha)}(\sigmax(s))$ if $a$ is of type (R-2).
\item[(G-3)] $\xsx_{a}(\chi):=
x_{\alpha}(\chi^{(1)})x_{\sigma(\alpha)}(\sigmax(\chi^{(1)}))x_{\alpha+\sigma(\alpha)}(N_{\sigma(\alpha),\alpha}\chi^{(2)})$ if $a$ is of type (R-3).
\item[(G-4)] $\xsx_{a}(\tilde{s}):=x_{\alpha}(\tilde{s})x_{\sigma(\alpha)}(\sigmax(\tilde{s}))x_{\sigma^2(\alpha)}(\sigmax^2(\tilde{s}))$ if $a$ is of type (R-4).
\end{description}

\begin{lemma}\label{prp:xsx}
These elements belong to $\Gtw$.
\end{lemma}
\begin{proof}
First, we show $\sigma(\xsx_{a}(\chi))=\xsx_{a}(\chi)$ for type (R-3).
For others, the proof is easy.
Let $\alpha\in\Delta$ be the corresponding root $a\ifff\alpha$.
By the commutator formula (see \eqref{eq:cij}),
we have 
$[x_\alpha(\chi^{(1)}),\,x_{\sigma}(\sigmax(\chi^{(1)}))]
=x_{\alpha+\sigma(\alpha)}(N_{\alpha,\sigma(\alpha)}\chi^{(1)}\sigmax(\chi^{(1)}))$.
Then by Proposition~\ref{prp:k_a},
\begin{eqnarray*}
&& \sigma(\xsx_{a}(\chi))
\\ &=&
x_{\sigma(\alpha)}(k_\alpha\sigmax(\chi^{(1)}))
x_{\alpha}(k_{\sigma(\alpha)}\chi^{(1)}) 
x_{\alpha+\sigma(\alpha)}(k_{\alpha+\sigma(\alpha)}N_{\sigma(\alpha),\alpha}\sigmax(\chi^{(2)}))
\\ &=&
x_{\sigma(\alpha)}(\sigmax(\chi^{(1)}))
x_{\alpha}(\chi^{(1)}) 
x_{\alpha+\sigma(\alpha)}(-N_{\sigma(\alpha),\alpha}\sigmax(\chi^{(2)}))
\\ &=&
x_{\alpha+\sigma(\alpha)}(N_{\alpha,\sigma(\alpha)}s\sigmax(\chi^{(1)}))^{-1}
x_{\alpha}(\chi^{(1)})
x_{\sigma(\alpha)}(\sigmax(\chi^{(1)}))
x_{\alpha+\sigma(\alpha)}(-N_{\sigma(\alpha),\alpha}\sigmax(\chi^{(2)}))
\\ &=&
x_{\alpha}(\chi^{(1)}) 
x_{\sigma(\alpha)}(\sigmax(\chi^{(1)}))
x_{\alpha+\sigma(\alpha)}(-N_{\alpha,\sigma(\alpha)}\chi^{(1)}\sigmax(\chi^{(1)})-N_{\sigma(\alpha),\alpha}\sigmax(\chi^{(2)}))
\\ &=&
x_{\alpha}(\chi^{(1)}) 
x_{\sigma(\alpha)}(\sigmax(\chi^{(1)}))
x_{\alpha+\sigma(\alpha)}(N_{\sigma(\alpha),\alpha}(\chi^{(1)}\sigmax(\chi^{(1)})-\sigmax(\chi^{(2)}))).
\end{eqnarray*}
Since $\chi^{(1)}\sigmax(\chi^{(1)})-\sigmax(\chi^{(2)})=\chi^{(2)}$, we are done.

Next, we show $\omega(\xsx_{a}(\tilde{s}))=\xsx_{a}(\tilde{s})$ for type (R-4).
Suppose that $a$ is of type (R-4) with $a\ifff\alpha\in\Delta$.
By definition, $\xsx_{a_2}(\tilde{s})$, $\xsx_{a_1+a_2}(\tilde{s})$ and $\xsx_{a_1+2a_2}(\tilde{s})$
are respectively given as follows.
\begin{itemize}
\item
$x_{\alpha_1}(\tilde{s})
x_{\alpha_3}(\sigmax(\tilde{s}))
x_{\alpha_4}(\sigmax^2(\tilde{s}))$,
\item
$x_{\alpha_1+\alpha_2}(\tilde{s})
x_{\alpha_2+\alpha_3}(\sigmax(\tilde{s}))
x_{\alpha_2+\alpha_4}(\sigmax^2(\tilde{s}))$, and
\item
$x_{\alpha_2+\alpha_3+\alpha_4}(\tilde{s})
x_{\alpha_1+\alpha_2+\alpha_4}(\sigmax(\tilde{s}))
x_{\alpha_1+\alpha_2+\alpha_3}(\sigmax^2(\tilde{s}))$.
\end{itemize}
Thus, $\omega(\xsx_{a_2}(\tilde{s}))$, $\omega(\xsx_{a_1+a_2}(\tilde{s}))$, and $\omega(\xsx_{a_1+2a_2}(\tilde{s}))$ are respectively calculated as follows.
\begin{itemize}
\item $x_{\alpha_1}(\omegax(\tilde{s}))x_{\alpha_4}(\omegax\sigmax(\tilde{s}))x_{\alpha_3}(\omegax\sigmax^2(\tilde{s}))$,
\item $x_{\alpha_1+\alpha_2}(\omegax(\tilde{s})) x_{\alpha_2+\alpha_4}(\omegax\sigmax(\tilde{s}))x_{\alpha_2+\alpha_3}(\sigmax^2(\tilde{s}))$, and
\item $x_{\alpha_2+\alpha_3+\alpha_4}(\omegax(\tilde{s}))x_{\alpha_1+\alpha_2+\alpha_3}(\omegax\sigmax(\tilde{s}))
x_{\alpha_1+\alpha_2+\alpha_4}(\omegax\sigmax^2(\tilde{s}))$.
\end{itemize}
Since $\omegax\sigmax\omegax=\sigmax^2$ and $\omegax(\tilde{s})=\tilde{s}$, we see that $\omega(\xsx_{a}(\tilde{s}))=\xsx_a(\tilde{s})$.
\end{proof}

The set $\frakA$ forms a (non commutative) group by letting
\begin{equation} \label{eq:dotplus}
\chi\dotplus\phi:=\big(\chi^{(1)}+\phi^{(1)},\,\chi^{(2)}+\phi^{(2)}+\sigmax(\chi^{(1)})\phi^{(1)}\big),
\end{equation}
for $\chi=(\chi^{(1)},\chi^{(2)}),\,\phi=(\phi^{(1)},\phi^{(2)})\in\frakA$.
The unit element is given by $0:=(0,0)$ and the inverse element of $\chi=(\chi^{(1)},\chi^{(2)})$ is given by
$\dotminus\chi:=(-\chi^{(1)},\sigmax(\chi^{(2)}))$.
For an (R-3) type root $a\in\Deltas$, one sees
\begin{equation}\label{eq:dotplusx}
\xsx_a(\chi)\xsx_a(\phi)=\xsx_a(\chi\dotplus\phi)
\quad \text{and} \quad
\xsx_a(\chi)^{-1}=\xsx_a(\dotminus\chi)
\end{equation}
for $\chi,\phi\in\frakA$.

\subsection{Special Elements in Twisted Loop Groups}
As in \cite[\S2]{Abe}, we define
\begin{equation} \label{eq:hits}
s\hits\chi:=(s\chi^{(1)},\,s\sigmax(s)\chi^{(2)})
\end{equation}
for $s\in \Skk$ and $\chi=(\chi^{(1)},\chi^{(2)})\in\frakA$.
One sees that this $\hits$ defines an action of $\Skk$ on the group $\frakA$.
Set
\[
\frakA^*:=\{\zeta=(\zeta^{(1)},\zeta^{(2)})\in\frakA\mid\zeta^{(2)}\in\Skk^\times\},
\]
where $\Skk^\times$ is the multiplicative group of $\Skk$.

In the following, we use the following usual notations.
\[
w_\alpha(t):=x_{\alpha}(t)x_{-\alpha}(-t)x_{\alpha}(t),
\qquad
h_\alpha(t):=w_\alpha(t)w_{\alpha}(-1)
\]
for $\alpha\in\Delta$ and $t\in \Rkk^\times$.
One easily sees that $w_\alpha(-t)=w_\alpha(t)^{-1}$.
It is known that for all $\alpha,\beta\in\Delta$ and $t\in\Rkk^\times$,
we have $w_\alpha(1)h_\beta(t)w_\alpha(-1)=h_{s_\alpha(\beta)}(t)$,
where $s_\alpha(\beta):=\beta-(\beta,\alpha^\vee)\alpha$ (see \cite[Lemma~20(a)]{Ste} for example).

For $a\in\Deltas$ with $a\ifff \alpha\in\Delta$, $t\in\Rkk^\times$, $q\in\Skk^\times$, $\tilde{q}\in (\Skk^\omegax)^\times$,
and $\zeta=(\zeta^{(1)},\zeta^{(2)})\in\frakA^*$, 
we define the following elements in $\Gtw$.
\begin{description}
\item[(W-1)] $\wsw_{a}(t):=w_\alpha(t)$ if $a$ is of type (R-1).
\item[(W-2)] $\wsw_{a}(q):=\xsx_{a}(q)\xsx_{-a}(-\sigmax(q)^{-1})\xsx_{a}(q)$ if $a$ is of type (R-2).
\item[(W-3)] $\wsw_{a}(\zeta):=
\xsx_{a}(\zeta)\xsx_{-a}(-{\sigmax(\zeta^{(2)})}^{-1}\hits\zeta)\xsx_{a}(\zeta^{(2)}\sigmax(\zeta^{(2)})^{-1}\hits\zeta)$ if $a$ is of type (R-3).
\item[(W-4)] $\wsw_a(\tilde{q}):=\xsx_{a}(\tilde{q})\xsx_{-a}(-\tilde{q}^{-1})\xsx_{a}(\tilde{q})$ if $a$ is of type (R-4).
\end{description}

\begin{lemma} \label{prp:wsw}
Let $a\in\Deltas$ with $a\ifff\alpha\in\Delta$, $q\in\Skk^\times$, and $\tilde{q}\in (\Skk^\omegax)^\times$.
\begin{enumerate}
\item If $a$ is of type (R-2), then $\wsw_{a}(q)=w_{\alpha}(q)w_{\sigma(\alpha)}(\sigmax(q))$ and $\wsw_{a}(-q)=\wsw_{a}(q)^{-1}$.
\item If $a$ is of type (R-4), then $\wsw_{a}(\tilde{q})=w_{\alpha}(\tilde{q})w_{\sigma(\alpha)}(\sigmax(\tilde{q}))w_{\sigmax^2(\alpha)}(\sigma^2(\tilde{q}))$
and $\wsw_{a}(-\tilde{q})=\wsw_{a}(\tilde{q})^{-1}$.
\end{enumerate}
\end{lemma}
\begin{proof}
First, we show the claim for the case when $a$ is of type (R-2).
Since $\alpha\pm\sigma(\alpha)\notin\Delta$ and $-a\ifff -\sigma(\alpha)$, we have
\begin{eqnarray*}
\wsw_{a}(q)
&=&
x_{\alpha}(q)x_{\sigma(\alpha)}(\sigmax(q))x_{-\sigma(\alpha)}(-\sigmax(q)^{-1})x_{-\alpha}(-q^{-1})x_{\alpha}(q)x_{\sigma(\alpha)}(\sigmax(q))
\\ &=&
x_{\alpha}(q)x_{-\sigma(\alpha)}(-\sigmax(q)^{-1})x_{\alpha}(q)x_{\sigma(\alpha)}(\sigmax(q))x_{-\alpha}(-q^{-1})x_{\sigma(\alpha)}(\sigmax(q))
\\ &=&
w_\alpha(q)w_{\sigma(\alpha)}(\sigmax(q)).
\end{eqnarray*}
Moreover, we have $\wsw_a(-q)=w_\alpha(q)^{-1}w_{\sigma(\alpha)}(\sigmax(q))^{-1}=\wsw_a(q)^{-1}$.

Next, suppose that $a$ is of type (R-4).
By definition, $\wsw_a(\tilde{q})$ is given as follows
\[
\begin{gathered}
x_{\alpha}(\tilde{q})x_{\sigma(\alpha)}(\sigmax(\tilde{q}))x_{\sigma^2(\alpha)}(\sigmax^2(\tilde{q}))
\cdot
x_{-\alpha}(-\tilde{q}^{-1})x_{-\sigma(\alpha)}(-\sigmax(\tilde{q}^{-1}))x_{-\sigma^2(\alpha)}(-\sigmax^2(\tilde{q}^{-1}))
\\
\cdot x_{\alpha}(\tilde{q})x_{\sigma(\alpha)}(\sigmax(\tilde{q}))x_{\sigma^2(\alpha)}(\sigmax^2(\tilde{q})).
\end{gathered}
\]
Since $\alpha\pm\sigma(\alpha)\notin\Delta$ and $\alpha\pm\sigma^2(\alpha)\notin\Delta$, the claim follows.
\end{proof}

It is known that the elements $h_\alpha(t)=w_\alpha(t)w_{\alpha}(-1)$ ($\alpha\in\Delta, t\in \Rkk^\times$) satisfy
$h_\alpha(\tau)h_\beta(\theta)=h_\beta(\theta)h_\alpha(\tau)$ and
$h_\alpha(\tau)h_\alpha(\theta)=h_\alpha(\tau\theta)$ for $\alpha,\beta\in\Delta$, $\tau,\theta\in\K^\times$.
Moreover, if $H_\alpha=\sum_{i\in \J} n_i H_{\alpha_i}$ for some $n_i\in \mathbb{Z}$,
then one sees that $h_\alpha(\tau)=\prod_{i\in \J} h_{\alpha_i}(\tau^{n_i})$ for all $\tau\in\K^\times$.
In particular, $h_\alpha(\tau^{-1})=h_\alpha(\tau)^{-1}$ and $h_{-\alpha}(\tau)=h_\alpha(\tau)^{-1}$.

For $a\in\Deltas$ with $a\ifff \alpha\in\Delta$, $t\in\Rkk^\times$, $q\in\Skk^\times$, $\tilde{q}\in(\Skk^\omegax)^\times$, $\zeta,\gamma\in\frakA^*$, 
we define the following elements in $\Gtw$.
\begin{description}
\item[(H-1)] $\hsh_{a}(t):=h_\alpha(t)$ if $a$ is of type (R-1).
\item[(H-2)] $\hsh_{a}(q):=\wsw_{a}(q)\wsw_{a}(-1)$ if $a$ is of type (R-2).
\item[(H-3)] $\hsh_{a}(\zeta,\gamma):=\wsw_{a}(\zeta)\wsw_{a}(\gamma)$ if $a$ is of type (R-3).
\item[(H-4)] $\hsh_{a}(\tilde{q}):=\wsw_{a}(\tilde{q})\wsw_{a}(-1)$ if $a$ is of type (R-4).
\end{description}

\begin{lemma} \label{prp:hsh}
Let $a\in\Deltas$ with $a\ifff\alpha\in\Delta$, $q\in\Skk^\times$, and $\tilde{q}\in(\Skk^\omegax)^\times$.
\begin{enumerate}
\item If $a$ is of type (R-2), then
$\hsh_{a}(q)=h_{\alpha}(q)h_{\sigma(\alpha)}(\sigmax(q))$ and $\hsh_{a}(q^{-1}) = \hsh_{a}(q)^{-1}$.
\item If $a$ is of type (R-4), then
$\hsh_{a}(\tilde{q})=h_{\alpha}(\tilde{q})h_{\sigma(\alpha)}(\sigmax(\tilde{q}))h_{\sigma^2(\alpha)}(\sigmax^2(\tilde{q}))$
and $\hsh_{a}(\tilde{q}^{-1})=\hsh_{a}(\tilde{q})^{-1}$.
\end{enumerate}
\end{lemma}
\begin{proof}
We only show lemma for (R-2).
For (R-4), the proof is essentially the same.
Suppose that $a$ is of type (R-2).
By Lemma~\ref{prp:wsw}, we have
\begin{eqnarray*}
\hsh_a(q)
&=&
w_\alpha(q)w_{\sigma(\alpha)}(\sigmax(q))w_\alpha(-1)w_{\sigma(\alpha)}(-1)\\
&=&
h_\alpha(q)w_\alpha(1)h_{\sigma(\alpha)}(\sigmax(q))w_{\sigma(\alpha)}(1)w_\alpha(-1)w_{\sigma(\alpha)}(-1)\\
&=&
h_\alpha(q)w_\alpha(1)h_{\sigma(\alpha)}(\sigmax(q))w_\alpha(-1)w_\alpha(1)w_{\sigma(\alpha)}(1)w_\alpha(-1)w_{\sigma(\alpha)}(-1)\\
&=&
h_\alpha(q)w_\alpha(1)h_{\sigma(\alpha)}(\sigmax(q))w_\alpha(-1).
\end{eqnarray*}
Since $s_\alpha(\sigma(\alpha))=\sigma(\alpha)$, we get $w_\alpha(1)h_{\sigma(\alpha)}(\sigmax(q))w_\alpha(-1)=h_{\sigma(\alpha)}(\sigmax(q))$.
This proves the first assertion.
By this result, $\hsh_a(q^{-1})=\hsh_a(q)^{-1}$ is trivial.
\end{proof}

As in \cite[\S2]{Abe}, we begin by introducing the following notation.
For $\zeta=(\zeta^{(1)},\zeta^{(2)}),\gamma=(\gamma^{(1)},\gamma^{(2)})\in\frakA^*$, we set
\begin{equation} \label{eq:c(xi,eta)}
\ccc(\zeta,\gamma):=\zeta^{(2)}\,\sigmax(\gamma^{(2)})^{-1}.
\end{equation}
Note that in \cite{Abe}, the right hand side was denoted by $c(\zeta,\gamma)$.
However, we use $\ccc(\zeta,\gamma)$ to avoid confusion with our setting.
By \cite[\S2 (C)]{Abe}, we have the following result.
\begin{lemma} \label{prp:mult-h}
For an (R-3) type root $a\in\Deltas$ and $\zeta_1,\dots,\zeta_k,\gamma_1,\dots,\gamma_k\in\frakA^*$,
we have $\hsh_a(\zeta_1,\gamma_1)\cdots\hsh_a(\zeta_k,\gamma_k)=1$ if and only if $\ccc(\zeta_1,\gamma_1)\cdots\ccc(\zeta_k,\gamma_k)=1$.
\end{lemma}

For $(\XX_N,r)=(\XA_{2\ell},2)$ and the type (R-3) root $a_{\ell}\ifff\alpha_{\ell}$, one sees that
\begin{equation} \label{eq:h_{a_ell}}
\hsh_{a_\ell}(\zeta,\gamma)=h_{\alpha_\ell}(\sigmax(\ccc(\zeta,\gamma)))\,h_{\alpha_{\ell+1}}(\ccc(\zeta,\gamma)),
\end{equation}
where $\zeta,\gamma\in\frakA^*$.

As in Section~\ref{sec:KMtab}, we shall write $-a_0$ as
\begin{equation} \label{eq:-a_0}
-a_0=c_1a_1+c_2a_2+\cdots+c_\ell a_\ell
\end{equation}
for some non-negative integers $c_1,c_2,\dots,c_\ell$.
Then we have the following lemma.
\begin{lemma}\label{prp:twloop_h=1}
For fixed $\tau_0,\tau_1,\dots,\tau_\ell\in\K^\times$ and $\zeta,\gamma\in\frakA^*$, we have the following.
\begin{enumerate}
\item If $(\XX_N,r)\neq(\XA_{2\ell},2)$, then $\hsh_{a_0}(\tau_0)\hsh_{a_1}(\tau_1)\cdots\hsh_{a_\ell}(\tau_\ell)=1$ if and only if 
\[
\tau_p =
\begin{cases}
\tau_0^{rc_p}&\text{if $a_p$ is of type (R-1),} \\
\tau_0^{c_p} &\text{otherwise}
\end{cases}
\]
for all $1\leq p\leq \ell$.
\item If $(\XX_N,r)=(\XA_{2\ell},2)$, then $\hsh_{a_0}(\tau_0)\hsh_{a_1}(\tau_1)\cdots\hsh_{a_{\ell-1}}(\tau_{\ell-1})\hsh_{a_\ell}(\zeta,\gamma)=1$ if and only if 
\[
\ccc(\zeta,\gamma)=\sigmax(\ccc(\zeta,\gamma))
\quad \text{and} \quad
\tau_p=\ccc(\zeta,\gamma)
\]
for all $0\leq p\leq\ell-1$.
\end{enumerate}
\end{lemma}
\begin{proof}
First, suppose that $(\XX_N,r)=(\XA_{2\ell-1},2)$.
In this case, $-a_0=a_1+2a_2+\cdots+2a_{\ell-1}+a_\ell$ ($\ifff\alpha_1+\cdots+\alpha_{2\ell-2}$) is of type (R-2).
Then by Lemma~\ref{prp:hsh}, the element $\hsh_{a_0}(\tau_0)$ is described as follows.
\[
h_{\alpha_1}(\tau_0^{-1})h_{\alpha_2}(\tau_0^{-2})\cdots
h_{\alpha_{\ell-1}}(\tau_0^{-2})h_{\alpha_{\ell}}(\tau_0^{-1})h_{\alpha_{\ell+1}}(\tau_0^{-2})\cdots h_{\alpha_{2\ell-1}}(\tau_0^{-2}).
\]
Thus, the product $\hsh_{a_0}(\tau_0)\hsh_{a_1}(\tau_1)\cdots \hsh_{a_\ell}(\tau_\ell)$ is given as
\[
\begin{gathered}
h_{\alpha_1}(\tau_0^{-1}\tau_1) h_{\alpha_2}(\tau_0^{-2}\tau_2) \cdots h_{\alpha_{\ell-1}}(\tau_0^{-2}\tau_{\ell-1}) \\
\cdot h_{\alpha_{\ell}}(\tau_0^{-1}\tau_\ell)h_{\alpha_{\ell+1}}(\tau_0^{-2}\tau_{\ell-1})\cdots h_{\alpha_{2\ell-1}}(\tau_0^{-2}\tau_1).
\end{gathered}
\]
Hence, by \cite[Lemma~28(c)]{Ste}, $\hsh_{a_0}(\tau_0)\hsh_{a_1}(\tau_1)\cdots \hsh_{a_\ell}(\tau_\ell)=1$ if and only if
$\tau_0^{-1}\tau_1 = \tau_0^{-2}\tau_2 = \cdots = \tau_0^{-2}\tau_{\ell-1} = \tau_0^{-1}\tau_\ell = 1$.
Thus, we are done.

Next, for $(\XX_N,r)=(\XD_{\ell+1},2)$, $(\XE_6,2)$, and $(\XD_4,3)$,
the product $\hsh_{a_0}(\tau_0)\hsh_{a_1}(\tau_1)\cdots \hsh_{a_\ell}(\tau_\ell)$ is calculated respectively as follows.
\begin{itemize}
\item
$h_{\alpha_1}(\tau_0^{-2}\tau_1)
h_{\alpha_2}(\tau_0^{-2}\tau_2)
\cdots 
h_{\alpha_{\ell-1}}(\tau_0^{-2}\tau_{\ell-1})
h_{\alpha_{\ell}}(\tau_0^{-1}\tau_\ell)
h_{\alpha_{\ell+1}}(\tau_0^{-1}\tau_{\ell})$,
\item
$h_{\alpha_1}(\tau_0^{-2}\tau_1)
h_{\alpha_2}(\tau_0^{-3}\tau_2)
h_{\alpha_{3}}(\tau_0^{-4}\tau_{3})
h_{\alpha_{4}}(\tau_0^{-2}\tau_4)
h_{\alpha_{5}}(\tau_0^{-3}\tau_{2})
h_{\alpha_{6}}(\tau_0^{-2}\tau_{1})$,
and
\item
$h_{\alpha_1}(\tau_0^{-2}\tau_2)
h_{\alpha_2}(\tau_0^{-3}\tau_1)
h_{\alpha_3}(\tau_0^{-2}\tau_2)
h_{\alpha_4}(\tau_0^{-2}\tau_2)$.
\end{itemize}
Hence, the claim follows.

Finally, we suppose that $(\XX_N,r)=(\XA_{2\ell},2)$.
Then by the equation \eqref{eq:h_{a_ell}}, the product $\hsh_{a_0}(\tau_0)\hsh_{a_1}(\tau_1)\cdots\hsh_{a_\ell}(\zeta,\gamma)$ is given as
\[
\begin{gathered}
h_{\alpha_1}(\tau_0^{-1}\tau_1)\cdots h_{\alpha_{\ell-1}}(\tau_0^{-1}\tau_{\ell-1})
\cdot h_{\alpha_{\ell}}(\tau_0^{-1}\sigmax(\ccc(\zeta,\gamma))) h_{\alpha_{\ell+1}}(\tau_0^{-1}\ccc(\zeta,\gamma))\\
\cdot h_{\alpha_{\ell+2}}(\tau_0^{-1}\tau_{\ell-1}) \cdots h_{\alpha_{2\ell}}(\tau_0^{-1}\tau_1).
\end{gathered}
\]
Thus, we are done.
\end{proof}

\subsection{The case $(\XA_{2},2)$}
In this subsection, we shall consider the case $(\XX_N,r)=(\XA_{2},2)$.
In this case, we may regard $\G$ as the special linear group scheme $\SL_3$ of degree $3$ defined over $\K$,
and $\Skk=\K[z^{\pm\frac{1}{2}}]$ (see \cite[Chapter~3]{Ste} for example).
Then as in \cite[\S2]{Abe}, the twisted loop group $\Gtw$ is explicitly given as
\[
\SU_3(\Skk):=\{C\in\SL_3(\Skk) \mid {}^t C\,J\,\sigmax(C)=J\},
\,\,
J:=
\begin{pmatrix}
0 & 0 & -1\\
0 & 1 & 0\\
-1 & 0 & 0
\end{pmatrix}.
\]
Here, ${}^t C$ is the transpose matrix of $C$
and $\sigmax(C):=(\sigmax(s_{ij}))_{1\leq i,j\leq 3}$ for $C=(s_{ij})_{1\leq i,j\leq 3}\in\SL_3(\Skk)$.
By definition, we have
\[
\xsx_{a_1}(\chi)=
\begin{pmatrix}
1 & \chi^{(1)} & \sigmax(\chi^{(2)}) \\
0 & 1 & \sigmax(\chi^{(1)}) \\
0 & 0 & 1
\end{pmatrix},
\quad
\xsx_{-a_1}(\chi)=
\begin{pmatrix}
1 & 0 & 0 \\
\sigmax(\chi^{(1)}) & 1 & 0 \\
\sigmax(\chi^{(2)}) & \chi^{(1)} & 1
\end{pmatrix}
\]
for $\chi=(\chi^{(1)},\chi^{(2)})\in\frakA$.
Moreover, we have
\[
\wsw_{a_1}(\zeta)=
\begin{pmatrix}
0 & 0 & \sigmax(\zeta^{(2)}) \\
0 & \frac{-\zeta^{(2)}}{\sigmax(\zeta^{(2)})} & 0 \\
\frac{1}{\zeta^{(2)}} & 0 & 0
\end{pmatrix},
\,\,
\hsh_{a_1}(\zeta,\gamma)=
\begin{pmatrix}
\frac{\sigmax(\zeta^{(2)})}{\gamma^{(2)}} & 0 & 0 \\
0 & \frac{-\zeta^{(2)}\gamma^{(2)}}{\sigmax(\zeta^{(2)}\gamma^{(2)})} & 0 \\
0 & 0 & \frac{\sigmax(\gamma^{(2)})}{\zeta^{(2)}}
\end{pmatrix}
\]
for $\zeta=(\zeta^{(1)},\zeta^{(2)}), \gamma=(\gamma^{(1)},\gamma^{(2)})\in\mathfrak{A}^*$.

Let $\EE{\Skk}$ denote the subgroup of $\SU_3(\Skk)$ generated by the set 
$\{\xsx_{a}(\chi) \mid a=\pm a_1,\,\chi\in\frakA\}$.
The purpose of this subsection is to prove the following theorem.
\begin{theorem} \label{thm:SU_3=elem}
$\SU_3(\Skk)$ coincides with $\EE{\Skk}$.
\end{theorem}

To prove this,
we prepare some technical notations.
For $n\in\mathbb{Z}$, $m\in2\mathbb{Z}+1$, $\zeta=(\zeta^{(1)},\zeta^{(2)})\in\frakA^*$, and $\tau\in\K^\times$, we set
\[
\xsx'_{\pm a_1}(\tau z^{\frac{m}{2}}):=\xsx_{\pm a_1}\big((0,-\tau z^{\frac{m}{2}})\big), \quad
\wsw'_{a_1}(\zeta):=\wsw_{a_1}\big((\tfrac{\zeta^{(1)}}{\zeta^{(2)}},\,\tfrac{1}{\zeta^{(2)}})\big),
\]
and
\begin{eqnarray*}
\hsh'_{a_1}(\tau z^{\frac{n}{2}}) 
&:=& \hsh_{a_1}\big((1,\tfrac{1}{2}),(0,\tfrac{1}{2}z^{-\frac{1}{2}})\big)^{n-1}
\cdot\hsh_{a_1}\big((1,\tfrac{1}{2}),(0,\tfrac{1}{2}\tau z^{-\frac{1}{2}})\big)
\\ &=&
\begin{pmatrix}
\tau z^{\frac{n}{2}} & 0 & 0 \\
0 & (-1)^n & 0 \\
0 & 0 & (-1)^n \tau^{-1} z^{-\frac{n}{2}}
\end{pmatrix}.
\end{eqnarray*}
Note that, $\tfrac12\in\K$, since $\K$ is supposed to be of characteristic not equal to $2$.

For $0\neq s=\sum_{n}s_nz^{\frac{n}{2}}\in\Skk$,
we set
\[
\lenM(s):=\mathsf{max}\{n\in\mathbb{Z}\mid s_n\neq0\}, \quad
\lenm(s):=\mathsf{min}\{n\in\mathbb{Z}\mid s_n\neq0\},
\]
and $\plen(s):=\lenM(s)-\lenm(s)$.
Also, we put $\lenM(0)=\lenm(0)=\plen(0):=0$.
For $C=(s_{ij})_{1\leq i,j\leq3}\in\SU_3(\Skk)$,
we set $\lenM(C)_{ij}:=\lenM(s_{ij})$ and $\lenm(C)_{ij}:=\lenm(s_{ij})$, and $\plen(C)_{ij}:=\plen(s_{ij})$ for each $1\leq i,j\leq 3$.
\begin{lemma} \label{lem:SU_3-0}
For any $C\in\SU_3(\Skk)$,
there exists $E\in\EE{\Skk}$ such that 
$\lenM(C)_{11}=\lenM(E\cdot C)_{31}$, $\lenM(C)_{31}=\lenM(E\cdot C)_{11}$, and $\plen(E\cdot C)_{31}\leq\plen(E\cdot C)_{11}$.
\end{lemma}
\begin{proof}
If $\plen(C)_{31}\leq\plen(C)_{11}$,
then we just take $E$ as the identity matrix.
Otherwise, we put $E:=\hsh'_{a_1}(\tfrac{1}{2})\cdot\wsw'_{a_1}((1,\tfrac{1}{2}))\in\EE{\Skk}$.
Then we have
\[
E\cdot C=
\begin{pmatrix}
\phantom{-}s_{31} & \phantom{-}s_{32} & \phantom{-}s_{33} \\
-s_{21} & -s_{22} & -s_{23} \\
\phantom{-}s_{11} & \phantom{-}s_{12} & \phantom{-}s_{13}
\end{pmatrix},
\]
where $C=(s_{ij})_{1\leq i,j\leq3}$.
Thus, we are done.
\end{proof}

In the following we fix $C=(s_{ij})_{1\leq i,j\leq3}\in\SU_3(\Skk)$ which satisfies $\plen(C)_{31}\leq\plen(C)_{11}$.
\begin{lemma} \label{lem:SU_3-1}
There exists $E\in\EE{\Skk}$ such that one of the following holds:
\begin{enumerate}
\item $\plen(E\cdot C)_{31}\leq\plen(E\cdot C)_{11}$ and $\lenM(E\cdot C)_{11}=\lenM(E\cdot C)_{31}$.
\item $\plen(E\cdot C)_{31}\leq\plen(E\cdot C)_{11}$ and $\plen(E\cdot C)_{11}<\plen(C)_{11}$.
\end{enumerate}
\end{lemma}
\begin{proof}
Let us write
$s_{11}=\sum_{n=m}^M\nu_nz^{\frac{n}{2}}$ and $s_{31}=\sum_{n=m'}^{M'}\mu_nz^{\frac{n}{2}}$,
where $m=\lenm(s_{11})$, $M=\lenM(s_{11})$, $m'=\lenm(s_{31})$, and $M'=\lenM(s_{31})$.
\medskip

\noindent{\bf (I)~Case $M\equiv M'\pmod2$:}
Since $M'-M\in2\mathbb{Z}$, we can consider the element $E:=\hsh'_{a_1}(z^{\frac{M'-M}{4}})\in\EE{\Skk}$.
Then
\[
E \cdot C=
\begin{pmatrix}
\cdots+\nu_{M} z^{\frac{M+M'}{4}} & s'_{12} & s'_{13} \\ 
s'_{21} & s'_{22} & s'_{23} \\
\cdots+(-1)^{\frac{M'-M}{2}}\mu_{M'} z^{\frac{M+M'}{4}} & s'_{32} & s'_{33}
\end{pmatrix}
\]
for some $s'_{ij}\in\Skk$.
Thus, we have $\lenM(E\cdot C)_{11}=\lenM(E\cdot C)_{31}$.
One sees that $\plen(E\cdot C)_{31}=\plen(C)_{31}$ and $\plen(E\cdot C)_{11}=\plen(C)_{11}$.
\medskip

\noindent{\bf (II)~Case $M\not\equiv M'\pmod2$:}
Since $M'-M+1\in2\mathbb{Z}$, we see $z^{\frac{M'-M+1}{4}}\in\Skk^\times$.
For simplicity, we put $\sgn:=(-1)^{\frac{M'-M+1}{2}}$.
Since $(0,-\sgn\frac{\nu_M}{\mu_{M'}}z^{\frac{1}{2}})\in\frakA$,
we can consider the element $E:=\xsx'_{a_1}(\sgn\tfrac{\nu_M}{\mu_{M'}}z^{\frac{1}{2}})\cdot\hsh'_{a_1}(z^{\frac{M'-M+1}{4}})\in\EE{\Skk}$.
Then
\begin{eqnarray*}
&& E\cdot C
\\ &=&
\begin{pmatrix}
1 & 0 & -\sgn\frac{\nu_M}{\mu_{M'}}z^{\frac{1}{2}} \\
0 & 1 & 0 \\
0 & 0 & 1
\end{pmatrix}
\begin{pmatrix}
\cdots+\nu_{M-1} z^{\frac{M+M'-1}{4}}+\nu_{M} z^{\frac{M+M'+1}{4}} & s_{12}' & s_{13}' \\
s_{21}' & s_{22}' & s_{23}' \\ 
\cdots+\sgn\mu_{M'-1}z^{\frac{M+M'-3}{4}}+\sgn\mu_{M'}z^{\frac{M+M'-1}{4}} & s_{32}' & s_{33}'
\end{pmatrix}
\\ &=&
\begin{pmatrix}
\cdots+(\nu_{M-1}-\frac{\nu_M}{\mu_{M'}}\mu_{M'-1})z^{\frac{M+M'-1}{4}} & s_{12}'' & s_{13}'' \\
s_{21}'' & s_{22}'' & s_{23}'' \\
\cdots+\sgn\mu_{M'-1}z^{\frac{M+M'-3}{4}}+\sgn\mu_{M'}z^{\frac{M+M'-1}{4}} & s_{32}'' & s_{33}''
\end{pmatrix}
\end{eqnarray*}
for some $s'_{ij},s_{ij}''\in\Skk$.
If $\nu_{M-1}-\frac{\nu_M}{\mu_{M'}}\mu_{M'-1}\neq0$, then we have $\lenM(E\cdot C)_{11}=\lenM(E\cdot C)_{13}$.
Using Lemma~\ref{lem:SU_3-0} (if necessary), we are done.

Suppose that $\nu_{M-1}-\frac{\nu_M}{\mu_{M'}}\mu_{M'-1}=0$.
Then we have
\[
E\cdot C =
\begin{pmatrix}
\cdots+(\nu_{M-2}-\frac{\nu_M}{\mu_{M'}}\mu_{M'-2})z^{\frac{M+M'-3}{4}} & s_{12}'' & s_{13}'' \\
s_{21}'' & s_{22}'' & s_{23}'' \\
\cdots+\sgn\mu_{M'-1}z^{\frac{M+M'-3}{4}}+\sgn\mu_{M'}z^{\frac{M+M'-1}{4}} & s_{32}'' & s_{33}''
\end{pmatrix}.
\]
In this case, we have $\plen(E\cdot C)_{11}\leq\plen(C)_{11}$,
since $\plen(C)_{31}=M'-m'\leq\plen(C)_{11}=M-m$.
If $\plen(E\cdot C)_{11}<\plen(C)_{11}$, then we are done.
Otherwise, for simplicity, we set $\nu'_{n-2}:=\nu_{n-2}-\frac{\nu_M}{\mu_{M'}}\mu_{n-M+M'-2}$ for each $n$.
If $\nu'_{M-2}=0$, then the algorithm~(I) stated above works.
Suppose that $\nu'_{M-2}\neq0$.
Put $E':=\xsx'_{-a_1}(\sgn\tfrac{\mu_{M'}}{\nu'_{M-2}}z^{\frac{1}{2}})\cdot E\in\EE{\Skk}$.
Then we have
\begin{eqnarray*}
&& E'\cdot C
\\ &=&
\begin{pmatrix}
1 & 0 & 0\\
0 & 1 & 0 \\
\frac{-\sgn\mu_{M'}}{\nu'_{M-2}}z^{\frac{1}{2}} & 0 & 1
\end{pmatrix}
\begin{pmatrix}
\cdots+\nu'_{M-3}z^{\frac{M+M'-5}{4}}+\nu'_{M-2}z^{\frac{M+M'-3}{4}} & s_{12}'' & s_{13}'' \\
s_{21}'' & s_{22}'' & s_{23}'' \\
\cdots+\sgn\mu_{M'-1}z^{\frac{M+M'-3}{4}}+\sgn\mu_{M'}z^{\frac{M+M'-1}{4}} & s_{32}'' & s_{33}''
\end{pmatrix}
\\ &=&
\begin{pmatrix}
\cdots+\nu'_{M-3}z^{\frac{M+M'-5}{4}}+\nu'_{M-2}z^{\frac{M+M'-3}{4}} & s_{12}''' & s_{13}''' \\
s_{21}''' & s_{22}''' & s_{23}''' \\
\cdots+\sgn(\mu_{M'-1}-\tfrac{\mu_{M'}}{\nu'_{M-2}}\nu'_{M-3})z^{\frac{M+M'-3}{4}} & s_{32}''' & s_{33}'''
\end{pmatrix}.
\end{eqnarray*}
for some $s_{ij}'''\in\Skk$.
If $\mu_{M'-1}-\tfrac{\mu_{M'}}{\nu'_{M-2}}\nu'_{M-3}\neq0$, then $M(E'\cdot C)_{11}=M(E'\cdot C)_{13}$.
Using Lemma~\ref{lem:SU_3-0} (if necessary), we are done.
Otherwise, we just repeat the algorithm above.

Since the number of non-zero coefficients $\nu_n$ (resp.~$\mu_n$) of $s_{11}$ (resp.~$s_{31}$) are finite,
this algorithm leads us to the desired result.
\end{proof}

In the following, we assume that our $C=(s_{ij})_{1\leq i,j\leq3}$ satisfies 
$\plen(C)_{31}\leq\plen(C)_{11}$ and $M:=\lenM(C)_{11}=\lenM(C)_{31}$.
Let us write 
\[
s_{11}=\sum_{n=m}^M\nu_nz^{\frac{n}{2}},
\quad
s_{21}=\sum_{n=m'}^{M'}\iota_nz^{\frac{n}{2}}
\quad \text{and}\quad
s_{31}=\sum_{n=m''}^M\mu_nz^{\frac{n}{2}}.
\]
with $m=\lenm(s_{11})$, $m'=\lenm(s_{21})$, $m''=\lenm(s_{31})$ and $M'=\lenM(s_{21})$.
\begin{lemma} \label{lem:SU_3-2}
We have $M=M'$, $\iota_M^2=2\nu_M\mu_M$, and $m\leq\mathsf{min}\{m',m''\}$.
\end{lemma}
\begin{proof}
By the definition of $\sigmax$, we have the following equations:
\[
\begin{gathered}
s_{21}\sigmax(s_{21})
= (-1)^{m'}\iota_{m'}^2z^{m'}+\cdots+(-1)^{M'}\iota_{M'}^2z^{M'}, \\
s_{31}\sigmax(s_{11})+\sigmax(s_{31})s_{11} \\
=((-1)^{m}+(-1)^{m''})\nu_m\mu_{m''}z^{\frac{m+m''}{2}}+\cdots+(-1)^{M}2\nu_M\mu_Mz^M.
\end{gathered}
\]
Since $C\in\SU_3(\Skk)$, we have
$s_{21}\sigmax(s_{21})=s_{31}\sigmax(s_{11})+\sigmax(s_{31})s_{11}$.
Thus, $M'$ should coincide with $M$ and $\iota_M^2=2\nu_M\mu_M$.

Next, we show the last claim.
Since $\plen(s_{31})\leq\plen(s_{11})$, we have $m\leq m''$.
In the following, we show $m\leq m'$.
\medskip

\noindent{\bf (I)~Case $m\equiv m''\pmod2$:}
In this case, the lowest term of $s_{31}\sigmax(s_{11})+\sigmax(s_{31})s_{11}$ is given as 
\[
((-1)^{m}+(-1)^{m''})\nu_m\mu_{m''}z^{\frac{m+m''}{2}}
\,\,=\,\,
2\nu_m\mu_{m''}z^{\frac{m+m''}{2}}
\quad\neq0.
\]
Thus, we have $2m'=m+m''$.
This implies $2m\leq m+m''=2m'$, and hence $m\leq m'$.
\medskip

\noindent{\bf (II)~Case $m\not\equiv m''\pmod2$:}
In this case, $s_{31}\sigmax(s_{11})+\sigmax(s_{31})s_{11}$ is given as
\[
2\big(((-1)^{m+1}\nu_{m+1}\mu_{m''}+(-1)^{m''+1}\nu_m\mu_{m''+1})z^{\frac{m+m''+1}{2}}+\cdots+(-1)^{M}\nu_M\mu_Mz^{M}\big).
\]
Thus, there exists $k\geq1$ such that $2m'=m+m''+k$.
Since $m''\geq m$,
we have
$2m'=m+m''+k\geq 2m+k\geq 2m$, and hence $m'\geq m$.
\end{proof}

The following is a kind of ``Euclidean algorithm'' for $\SU_3(\Skk)$.
\begin{proposition}\label{prp:Euc-SU_3}
For any $C\in\SU_3(\Skk)$,
there exists $E\in\EE{\Skk}$ such that $\plen(E\cdot C)_{11}<\plen(C)_{11}$.
\end{proposition}
\begin{proof}
By Lemmas~\ref{lem:SU_3-0}, \ref{lem:SU_3-1} and \ref{lem:SU_3-2},
we may assume that
our $C=(s_{ij})_{1\leq i,j\leq 3}$ satisfies $M:=\lenM(s_{11})=\lenM(s_{21})=\lenM(s_{31})$,
$\lenm(s_{11})\leq \lenm(s_{21})$, and $\lenm(s_{11})\leq \lenm(s_{31})$.
One sees that $\chi:=(-\frac{2\nu_M}{\iota_M},\,\,\frac{2\nu_M^2}{\iota_M^2})$ belongs to $\frakA$.
Thus, we may consider $E:=\xsx_{a_1}(\chi)\in\EE{\Skk}$, and get
\[
E\cdot C 
=\begin{pmatrix}
s_{11}-\frac{2\nu_M}{\iota_M}s_{21}+\frac{2\nu_M^2}{\iota_M^2}s_{31} & s'_{12} & s'_{13} \\
s_{21}-\frac{2\nu_M}{\iota_M}s_{31} & s'_{22} & s'_{23} \\
s_{31} & s'_{32} & s'_{33}
\end{pmatrix}.
\]
By Lemma~\ref{lem:SU_3-2},
the coefficient of $z^{\frac{M}{2}}$ in 
the $(1,1)$-component of $E\cdot C$ is calculated as
\[
\nu_M-\frac{2\nu_M}{\iota_M}\iota_M+\frac{2\nu_M^2}{\iota_M^2}\mu_M
\,\,=\,\,
-\nu_M+\frac{2\nu_M^2}{2\nu_M\mu_M}\mu_M
\,\,=\,\,0.
\]
Thus, we have $\lenM(E\cdot C)_{11}<M=\lenM(C)_{11}$.

On the other hand, we have $\lenm(C)_{11}\leq \lenm(E\cdot C)_{11}$,
since $\lenm(C)_{11}\leq \lenm(C)_{21}$ and $\lenm(C)_{11}\leq \lenm(C)_{31}$.
Therefore, we conclude that
$\plen(E\cdot C)_{11}
=\lenM(E\cdot C)_{11}-\lenm(E\cdot C)_{11}
\leq \lenM(E\cdot C)_{11}-\lenm(C)_{11}
<\lenM(C)_{11}-\lenm(C)_{11}
=\plen(C)_{11}$.
\end{proof}

\begin{proof}[Proof of Theorem~\ref{thm:SU_3=elem}]
By repeatedly applying Proposition~\ref{prp:Euc-SU_3},
for an arbitrary $C\in\SU_3(\Skk)$,
there exists $E\in\EE{\Skk}$ such that $\plen(E\cdot C)_{11}=0$.
Since $E\cdot C\in\SU_3(\Skk)$ and the $(1,1)$-component of $E\cdot C$ is zero,
one easily sees that $E\cdot C$ is of the form
\[
E\cdot C
=\begin{pmatrix}
0 & 0 & s'_{13} \\
0 & s'_{22} & s'_{23} \\
s'_{31} & s'_{32} & s'_{33}
\end{pmatrix}
\]
with
$s'_{13}\in\Skk^\times$,
$s'_{31}=\sigmax(s'_{13})^{-1}$,
$s'_{22}=-\sigmax(s'_{13}){s'_{13}}^{-1}$,
$s'_{23}=-\sigmax(s'_{13}s'_{32})$,
and
$s'_{32}s'_{13}\sigmax(s'_{32}s'_{13})=s'_{33}\sigmax(s'_{13})+\sigmax(s'_{33})s'_{13}$.
Put
$E':=\hsh'_{a_1}(\tfrac{1}{2})\cdot\wsw'_{a_1}((1,\tfrac{1}{2}))$
and
$E'':=\hsh'_{a_1}(\sigmax(s'_{13})^{-1})\cdot\xsx'_{a_1}(s'_{13}s'_{23})$.
Then by a direct calculation shows that
\[
E'\cdot (E\cdot C)
=
\begin{pmatrix}
\sigmax(s'_{13})^{-1} & s'_{32} & s'_{33} \\
0 & \sigmax(s'_{13}){s'_{13}}^{-1} & \sigmax(s'_{13}s'_{32}) \\
0 & 0 & s'_{13}
\end{pmatrix}
=E''.
\]
Therefore, we conclude that $C=E^{-1}\cdot E'^{-1}\cdot E''\in\EE{\Skk}$.
\end{proof}

\subsection{Structure of Twisted Loop Groups}
At the end of this section, we describe the structure of the twisted loop groups $\Gtw$.

Let $\E{\Skk}$ be the subgroup of $\Gtw$ generated by the following set.
\begin{itemize}
\item $\{\xsx_a(s),\,\xsx_b(\chi) \mid s\in\Skk,\,a\in\Deltaslong,\,\chi\in\frakA,\,b\in\Deltasshort\}$, if $(\XX_N,r)=(\XA_{2\ell},2)$.
\item $\{\xsx_a(u),\,\xsx_b(\tilde{s}) \mid u\in\Rkk,\,a\in\Deltaslong,\,\tilde{s}\in\Skk^\omegax,\,b\in\Deltasshort\}$, if $(\XX_N,r)=(\XD_4,3)$.
\item $\{\xsx_a(u),\,\xsx_b(s) \mid u\in\Rkk,\,a\in\Deltaslong,\,s\in\Skk,\,b\in\Deltasshort\}$, otherwise.
\end{itemize}
Note that $\E{\Skk}$ coincides with $\EE{\Skk}$ (defined in the previous section) when $(\XX_N,r)=(\XA_{2},2)$.
Using Theorem~\ref{thm:SU_3=elem}, we have the following.
\begin{theorem}\label{prp:elem}
$\Gtw$ coincides with $\E{\Skk}$.
\end{theorem}
\begin{proof}
Let $\Tss$ and $\Bss$ be the maximal torus and Borel subgroup of $\Gcd$ corresponding to $\hcc$ and $\bcc$ respectively (see Section~\ref{sec:intro}).
We will for convenience denote  $\Tss_\K$ and $\Bss_\K$  simply by  $\Tss$ and $\Bss$.
Let $\Uss$ be the unipotent radical of $\Bss$.
Then $\Bss=\Tss\ltimes\Uss$ (semi-direct product).
For $w\in\WW$, we put $\Uss_w=\Uss\cap w^{-1}\Uss^-w$,
where $(\Bss^-,\Uss^-)$ is the opposite of $(\Bss,\Uss)$ and $\WW$ is the Weyl group of $\G$ with respect to $\Tss$.

First, suppose that $\xi\in\K$.
We consider the fraction field $K:=\K((z^{\frac1r}))$ of
the ring of formal power series $\K[[z^{\frac1r}]]$ in the variable $z^{\frac1r}$ over $\K$.
It is known that
\[
\G(K)=\bigsqcup_{w\in\WW}\Uss(K)\,\Tss(K)\,w\Uss_w(K)
\qquad (\text{disjoint union}),
\]
called a Bruhat decomposition (see \cite[Theorem~4]{Ste}).
Then we obtain:
\begin{eqnarray*}
\Gtw &=&
\G(\Skk)^\Gamma
\\ &=&
\G(K)^\Gamma\cap\G(\Skk)
\\ &=&
\big(\bigsqcup_{w\in\WW}\Uss(K)\,\Tss(K)\,w\Uss_w(K)\big)^\Gamma\cap\G(\Skk)
\\ &=&
\big(\bigsqcup_{w'\in\WW^\Gamma}\Uss(K)^\Gamma\,\Tss(K)^\Gamma\,w'\Uss_{w'}(K)^\Gamma\big)\cap\G(\Skk)
\\ &=&
\bigsqcup_{w'\in\WW^\Gamma}\big(\Uss(K)^\Gamma\,\Tss(K)^\Gamma\,w'\Uss_{w'}(K)^\Gamma w'^{-1}\cdot w'\cap\G(\Skk)\big)
\\ &\subset &
\bigsqcup_{w'\in\WW^\Gamma}\big(\Uss(K)^\Gamma\,\Tss(K)^\Gamma\,\Uss^-(K)^\Gamma\cap\G(\Skk)\big)\cdot w',
\end{eqnarray*}
Here, $\WW^\Gamma$ is the fixed-point subgroup of $\WW$ under $\Gamma$.
We note that if we take
$g \in (\Uss(K)\Tss(K)w\Uss_w(K))^\Gamma$ and write $g = uhwv$ for some $u \in \Uss(K), h \in \Tss(K), v \in \Uss_w(K)$, then
we obtain
$\gamma^*(u)\gamma^*(h)\gamma^*(w)\gamma^*(v) = \gamma^*(g) = g = uhwv$ for all $\gamma^* \in \Gamma$ and
$\gamma^*(u) = u, \gamma^*(h) = h, \gamma^*(w)=w, \gamma^*(v) = v$
by the uniqueness of expression (cf.~\cite[Theorem 4$'$]{Ste}).
This shows the forth equality in the display just above.

Let $g\in\Uss(K)^\Gamma\,\Tss(K)^\Gamma\,\Uss^-(K)^\Gamma\cap\G(\Skk)$ and write $g=uhv$ for some
\[
u\in\Uss(K)^\Gamma,\quad
h\in\Tss(K)^\Gamma,
\quad\text{and}\quad
v\in\Uss^-(K)^\Gamma.
\]
For each $p\in\Jsigma$,
we let $\Uss_{a_p}(K)^\Gamma$, $\Uss'_{a_p}(K)^\Gamma$, $\Vss_{-a_p}(K)^\Gamma$, and $\Vss_{-a_p}'(K)^\Gamma$ be the subgroup of $\G(K)^\Gamma$
corresponding to $\{a_p\,(,2a_p)\}$, $\Deltasp\setminus\{a_p\,(,2a_p)\}$,
$\{-a_p\,(,-2a_p)\}$ and $-\Deltasp\setminus\{-a_p\,(,-2a_p)\}$, respectively.
Here, we define $\{ a_p\,(,2a_p) \}$ to be $\{ a_p, 2a_p \}$ if $2a_p \in \Delta^\sigma$, and $\{ a_p \}$ if $2a_p \not\in \Delta^\sigma$.
We note that the subgroup corresponding to a set of roots
means the subgroup generated by the root subgroups parametrized by it.
Also, we define $\Tss_{a_p}(K)^\Gamma$ and $\Tss'_{a_p}(K)^\Gamma$ to be the subgroups of $\Tss(K)^\Gamma$
corresponding to $\{a_p\}$ and $\Pis\setminus\{a_p\}$, respectively (in the sense of tori).
Then, $g$ can be expressed as $g=u_{a_p} u_{a_p}' h_{a_p} h_{a_p}' v_{-a_p} v_{-a_p}'$,
where $u=u_{a_p}u'_{a_p}$, $h=h_{a_p}h'_{a_p}$, and $v=v_{-a_p}v'_{-a_p}$ for $a_p\in\Pis$,
and where $u_{a_p}\in\Uss_{a_p}(K)^\Gamma$, $u'_{a_p}\in\Uss'_{a_p}(K)^\Gamma$,
$h_{a_p}\in\Tss_{a_p}(K)^\Gamma$, $h'_{a_p}\in\Tss'_{a_p}(K)^\Gamma$,
$v_{-a_p}\in\Vss_{-a_p}(K)^\Gamma$, and $v'_{-a_p}\in\Vss'_{-a_p}(K)^\Gamma$.

Let $V^\lambda$ be a finite dimensional irreducible module of $\gcc$
generated by a maximal (or minimal) vector $v_0^\lambda$ with a highest (or lowest)
weight $\lambda$,
and $\mathcal{U}_{\mathbb{Z}}$ be a Chevalley $\mathbb{Z}$-form in the universal enveloping algebra of $\gcc$,
which is defined by a fixed Chevalley basis of $\gcc$.
Let $V^\lambda_{\mathbb{Z}}$ denote the $\mathcal{U}_{\mathbb{Z}}$-submodule of $V^\lambda$
generated by $v_0^\lambda$ (cf.~Appendix~\ref{sec:adm-pair}).
Note that there is an action of $\G(K)$ on $V^\lambda_K:=K\otimes_{\mathbb{Z}}V^\lambda_{\mathbb{Z}}$.
We can choose a direct sum of such finite dimensional irreducible modules if necessarily.

We understand that our $g$ belongs to
\begin{eqnarray*}
&&\Uss(K)^\Gamma \Tss(K)^\Gamma \Uss^{-}(K)^\Gamma\\
&=& \Uss_{a_p}(K)^\Gamma \Uss_{a_p}'(K)^\Gamma \Tss_{a_p}(K)^\Gamma \Tss_{a_p}'(K)^\Gamma
\Vss_{-a_p}(K)^\Gamma \Vss_{-a_p}'(K)^\Gamma \\
&=& \Uss_{a_p}(K)^\Gamma \Tss_{a_p}(K)^\Gamma \Vss_{-a_p}(K)^\Gamma \cdot \Uss_{a_p}'(K)^\Gamma
\Tss_{a_p}'(K)^\Gamma \Vss_{-a_p}'(K)^\Gamma \\
&\subset&
\left\langle\Uss_{a_p}(K)^\Gamma,\,\Tss_{a_p}(K)^\Gamma,\,\Vss_{-a_p}(K)^\Gamma \right\rangle
\cdot
\Uss'_{a_p}(K)^\Gamma\,\Tss'_{a_p}(K)^\Gamma\,\Vss'_{-a_p}(K)^\Gamma.
\end{eqnarray*}
If we suppose $g=g_{a_p}g'_{a_p}$ satisfying
\[
g_{a_p}\in\langle\Uss_{a_p}(K)^\Gamma,\,\Tss_{a_p}(K)^\Gamma,\,\Vss_{-a_p}(K)^\Gamma\rangle
\]
and
\[
g'_{a_p}\in\Uss'_{a_p}(K)^\Gamma\,\Tss'_{a_p}(K)^\Gamma\,\Vss'_{-a_p}(K)^\Gamma,
\]
then $g_{a_p}$ can be viewed as an element of
\[
{\rm(i)}\ \SL_2(K)^\Gamma;\quad
{\rm(ii)}\ (\SL_2(K)\times\SL_2(K))^\Gamma;
\]
\[
{\rm(iii)}\ (\SL_2(K)\times\SL_2(K)\times \SL_2(K))^\Gamma;\quad
{\rm(iv)}\  \SL_3(K)^\Gamma,
\]
whose matrix entries are all in $\Skk$ (or sometimes $\Rkk$).
This means that $g_{a_p}$ can be identified with an element of one of the following groups:
\[
{\rm(i)}\ \SL_2(\Rkk);\quad {\rm(ii)}\ \SL_2(\Skk);\quad {\rm(iii)}\ \SL_2(\Skk);\quad {\rm(iv)}\ \SU_3(\Skk).
\]
One may here recall how to construct twisted Chevalley groups by using the so-called ``foldings'' of root systems (cf. Section 3, or \cite{Abe}).
In this sense, one knows that there are four types of roots, namely
\[
{\rm(i)}\ a_p = \{ \alpha \};\quad
{\rm(ii)}\ a_p = \{ \alpha,\beta \}\ (\alpha \perp \beta);
\]
\[
{\rm(iii)}\ a_p = \{ \alpha, \beta, \gamma \}\ (\text{mutually orthogonal});\quad
{\rm(iv)}\ a_p =\{ \alpha,\beta \},\  2a_p = \{ \alpha+\beta \}.
\]

For the case when $(\XX_N,r)\neq(\XA_{2\ell},2)$, multiplying
\[
u''_j:=u_{a_p}(s_j)\in\Uss_{a_p}(K)^\Gamma
\quad\text{and}\quad
v''_j:=v_{-a_p}(s'_j)\in\Vss_{-a_p}(K)^\Gamma
\]
continuously, for some $s_j, s_j' \in \Skk$ (or sometimes $\Rkk$), from the left hand side,
we may use a ``standard'' Euclidean algorithm.
Then, we can obtain and assume $u_{a_p}=v_{-a_p}=1$.
Next, we will select some $w''\in \WW^\Gamma$ to have
\[
g'=w''gw''^{-1}\,\,\in\Uss(K)^\Gamma\,\Tss(K)^\Gamma\,\Uss^-(K)^\Gamma
\]
instead of $g$ and to repeat again as above.
We notice
\[
\bigcap_{w'\in\WW^\Gamma}w'\,\Uss^-(K)^\Gamma\,w'^{-1}=\{1\}.
\]
Then, it is possible to continue our process for $g'$,
and finally we obtain $v=1$ at least (cf.~Remark~\ref{rem:g->g'}).
In case of $(\XX_N,r)=(\XA_{2\ell},2)$,
we can also use the same method as above.
Then, it is enough to concentrate $(\XX_N,r)=(\XA_{2},2)$ and $\Gtw=\SU_3(\Skk)$.
In this case, we already know $\Gtw=\E{\Skk}$ by Theorem~\ref{thm:SU_3=elem}.
Therefore, in any case, we reach $v=1$.
Hence, we have
\[
g=uh\in\Bss(K)^\Gamma\cap\G(\Skk)=\Bss(\Skk)^\Gamma=\Uss(\Skk)^\Gamma\,\Tss(\Skk)^\Gamma,
\]
and hence we have $g\in\E{\Skk}$.
Therefore, $\Gtw$ coincides with $\E{\Skk}$.

If $\xi\notin\K$, then the proof is essentially the same as above.
That is, first we put $\K'=\K(\xi)$,
then one can take $\Gamma$-fixed points, and get the desired result.
\end{proof}

\begin{remark} \label{rem:g->g'}
Here is the way to understand our inductive method from $g$ to $g'$.
For $b\in\Deltasp$, we define $\Uss_b(K)^\Gamma$ (resp.~$\Vss_{-b}(K)^\Gamma$)
to be the subgroup of $\Uss(K)^\Gamma$ (resp.~$\Uss^{-}(K)^\Gamma$)
corresponding to $\{b\}$ (resp.~$\{-b\}$).
We already wrote $g=uhv$ and, more precisely we write
\[
u=x_1x_2\cdots x_k,\quad v=y_1y_2\cdots y_k,
\]
where $x_p\in\Uss_{b_p}(K)^\Gamma$, $y_p\in\Vss_{-b_p}(K)^\Gamma$, $b_1=a_p\in\Pis$
and
\[
u_{a_p}=x_1,\quad u'_{a_p}=x_2\cdots x_k,\quad v_{-a_p}=y_1,\quad v'_{-a_p}=y_2\cdots y_k.
\]
Using a ``standard'' Euclidean algorithm, we can obtain and assume $u_{a_p}=x_1=1$ and $v_{-a_p}=y_1=1$ as before.
Here, we can also assume $g=uhv$, satisfying $u=x_2\cdots x_k$ and $v=y_2\cdots y_k$,
from the beginning.
Choosing a suitable element $w''\in\WW^\Gamma$, we obtain
\[
g'=w''gw''^{-1}\,\,\in\Uss(K)^\Gamma\,\Tss(K)^\Gamma\,\Uss^{-}(K)^\Gamma,
\]
where $g'=u'h'v'=(x_2'\cdots x_k')h'(y_2'\cdots y_k')$, satisfying
\[
u'=x_2'\cdots x_k'\in\Uss(K)^\Gamma,
\quad
h'\in\Tss(K)^\Gamma,
\quad
v'=y_2'\cdots y_k'\in\Uss^-(K)^\Gamma,
\]
and where
\[
x_p'\in\Uss_{b_p'}(K)^\Gamma,
\quad
y_p'\in\Vss_{-b_p'}(K)^\Gamma,
\quad
b_2'=a_q\in\Pis,
\quad 
b_p'\in\Deltasp.
\]
Hence, we can repeat our process (cf.~Remark~\ref{rem:add-cld}).
Since the number of components in the expression of $v$ is decreasing,
finally we finish our process, and reach $v=1$ at least.
\end{remark}

\begin{remark} \label{rem:add-cld}
Set $\Deltasm:=-\Deltasp$.
A subset $\Psi\subset\Deltasm$ is called {\it (additively) closed} if $b+b'\in\Deltas$ for $b,b'\in\Psi$ always implies $b+b'\in\Psi$.
Note that $\Deltasm$ is closed.
We also find that $w'(\Psi)$ is closed if $\Psi\subset\Deltasm$ is closed and if $w'(\Psi)\subset\Deltasm$ for $w'\in\WW^\Gamma$.
Furthermore, we can see that $\Psi':=\Psi\setminus\{-a\}$ is closed if $\Psi$ is closed and if there is an element $a\in\Pis$ satisfying $-a\in\Psi$.
Hence, the recursive process in Remark~\ref{rem:g->g'} works in this sense.
\end{remark}

We can also obtain the following result
(cf.~\cite{Moody-Teo,Mor79,Mor81,Peterson-Kac,Tit2,Tit}):
\begin{proposition} \label{prp:Iwa-Matsu}
The twisted loop group $\Gtw$ admits the Iwahori-Matsumoto decomposition.
\[
\Gtw=\E{\Skk}=\bigsqcup_{w'\in\WW_{\mathsf{aff}}}B_Iw'B_I \quad \text{(disjoint union)}.
\]
Here, $\WW_{\mathsf{aff}}$ is the corresponding affine Weyl group and $B_I$ is the standard Iwahori subgroup (cf.~\cite{IwaMat}).
\end{proposition}

\section{Affine Kac-Moody Groups and Twisted Loop Groups}\label{sec:aff-KM-grp}
In the following, we consider the simply-connected affine Kac-Moody group $\Gkms$
of type $\XX_N^{(r)}$ defined over our field $\K$.
For the precise  definition, see Appendix~\ref{sec:sc_KM}.
We show that $\Gkms$ can be realized as a central extension of the twisted loop group $\Gtw$.

\subsection{Induced Group Homomorphisms} \label{sec:grp_homo}
In Theorem~\ref{prp:isom}, we have seen that $\varphi:\ghatqq(\XX_N^{(r)})\to\gtwh$ is an isomorphism of Lie algebras.
Since $\Gkms$ is simply-connected, one sees that this $\varphi$ induces a group homomorphism $\Phi:\Gkms\to\G(\Skk)$,
where $\Skk$ is given in Section~\ref{sec:tw_chev}.

For the case when $(\XX_N,r)=(\XA_{2\ell},2)$,
it is better to introduce the following notations in addition to (G-1)--(G-4), (W-1)--(W-4) and (H-1)--(H-4).
\[
\xsx_{2a}(s):=x_{\alpha+\sigma(\alpha)}(s),\quad
\wsw_{2a}(s):=w_{\alpha+\sigma(\alpha)}(s),\quad
\hsh_{2a}(s):=h_{\alpha+\sigma(\alpha)}(s)
\]
for $2a+n\deltaa\in\Deltashr$ with $2a\ifff\alpha+\sigma(\alpha)\in\Delta$ and $s\in\Skk$.
\begin{remark} \label{rem:x_2}
In the above case, $a$ is of type (R-3) and $n$ is an odd number,
and $\xsx_{2a}(z^{\frac{n}{2}})=x_{\alpha+\sigma(\alpha)}(z^{\frac{n}{2}})=\xsx_a\big((0,N_{\sigma(\alpha),\alpha}z^{\frac{n}{2}})\big)$.
By Lemma~\ref{prp:xsx}, this is an element of the twisted loop group $\Gtw$.
\end{remark}
Then by Proposition~\ref{prp:1} (see also \eqref{eq:tildeX_hata}),
the induced group homomorphism $\Phi:\Gkms\to\G(\Skk)$ is explicitly described as follows.
\begin{equation} \label{eq:Phi(x)}
\Phi(x_{\hat{a}}(\nu)) =
\begin{cases}
\xsx_{a'}(\nu z^{\frac{n}{r}}) & \text{if $a'$ is of type {\rm (R-1)}}, \\
\xsx_{a}(\xi_a^{-n} \nu z^{\frac{n}{r}}) & \text{if $a$ is of type {\rm (R-2)}}, \\
\xsx_{a}\big(\epsilon_a \xi_a^{-n} \nu z^{\frac{n}{2}} \hits (1,\tfrac{1}{2})\big) & \text{if $a$ is of type {\rm (R-3)}}, \\
\xsx_{a}(\nu z^{\frac{n}{r}}) & \text{if $a$ is of type {\rm (R-4)}}
\end{cases}
\end{equation}
for $\hat{a}=a'+n\deltaa\in\Deltashr$ and $\nu\in\K$.
Here, $(1,\tfrac{1}{2})$ is, of course, regarded as an element of $\frakA$.
For the constants $\xi_a$ and $\epsilon_a$, see Notation~\ref{not:e_a}.
\medskip

For simplicity, we put
\[
\chi_{\hat{a},\nu}:=\epsilon_a\xi_a^{-n}\nu z^{\frac{n}{2}}\hits(1,\tfrac{1}{2})
=
\begin{cases}
(\nu z^{\frac{n}{2}},\,\frac{1}{2}\xi^{-n}\nu^2 z^n) & \text{if } a\in \Deltasp, \\
(2\xi^{-n}\nu z^{\frac{n}{2}},\, 2\xi^{-n}\nu^2 z^n) & \text{otherwise}.
\end{cases}
\]
For the definition of $\hits$, see \eqref{eq:hits}.
\begin{remark}
Suppose that $a$ is of type (R-3).
It is easy to see that $\ccc(\chi_{\hat{a},\nu},\chi_{\hat{a},-1})=\nu^2$ (for the notation, see \eqref{eq:c(xi,eta)}).
For $\tau\in\K^\times$, 
one sees that $\chi_{\hat{a},\tau}\in\frakA^*$ and
$\hsh_a(\chi_{\hat{a},\tau},\chi_{\hat{a},-1})^{-1}=\hsh_a(\chi_{\hat{a},\tau^{-1}},\chi_{\hat{a},-1})$ by Lemma~\ref{prp:mult-h}.
\end{remark}
\begin{lemma} \label{prp:phi(w)}
For $\hat{a}=a'+n\deltaa\in\Deltashr$ and $\tau\in\K^\times$, we get
\[
\Phi(w_{\hat{a}}(\tau)) = \begin{cases}
\wsw_{a'}(\tau z^{\frac{n}{r}}) & \text{if $a'$ is of type {\rm (R-1)}}, \\
\wsw_{a}(\xi_a^{-n} \tau z^{\frac{n}{r}}) & \text{if $a$ is of type {\rm (R-2)}}, \\
\wsw_{a}(\chi_{\hat{a},\tau}) & \text{if $a$ is of type {\rm (R-3)}}, \\
\wsw_{a}(\tau z^{\frac{n}{r}}) & \text{if $a$ is of type {\rm (R-4)}}.
\end{cases}
\]
\end{lemma}
\begin{proof}
By definition, $w_{\hat{a}}(\tau)=x_{\hat{a}}(\tau)x_{-\hat{a}}(-\tau^{-1})x_{\hat{a}}(\tau)$.
Thus, if $a'$ is of type (R-1) or (R-4), then nothing to do.
First, we suppose that $a$ is of type (R-2).
Then we have
\[
\Phi(w_{\hat{a}}(\tau))=
\xsx_a(\xi_a^{-n}\tau z^{\frac{n}{r}})\xsx_{-a}(-\xi_{-a}^{n}\tau^{-1}z^{-\frac{n}{r}})\xsx_a(\xi_a^{-n}\tau z^{\frac{n}{r}}).
\]
On the other hand, by definition,
\[
\wsw_a(\xi_a^{-n})=
\xsx_a(\xi_a^{-n}\tau z^{\frac{n}{r}})\xsx_{-a}(-\xi_{a}^{n} \xi^{n}\tau^{-1} z^{-\frac{n}{r}})\xsx_a(\xi_a^{-n}\tau z^{\frac{n}{r}}).
\]
Since $\xi_a\xi_{-a}=\xi$ and $\xi_a=\pm1$, this proves the claim.

Next, suppose that $a$ is of type (R-3).
In this case, $\Phi(w_{\hat{a}}(\tau))$ is explicitly given as $\xsx_a(\chi_{\hat{a},\tau})\xsx_{-a}(\chi_{-\hat{a},-\tau^{-1}})\xsx_a(\chi_{\hat{a},\tau})$.
By the definition of $\hits$, we get
\begin{eqnarray*}
-\sigmax(\chi_{\hat{a},\tau}^{(2)})^{-1} \hits \chi_{\hat{a},\tau}
&=& \begin{cases}
(-2\xi^{n}\tau^{-1} z^{-\frac{n}{2}},\,2\xi^{n}\tau^{-2}z^{-n}) & \text{if } a\in \Deltasp, \\
(-\tau^{-1}z^{-\frac{n}{2}},\,\frac{1}{2}\xi^{n}\tau^{-2}z^{-n}) & \text{otherwise.} \end{cases}
\\ &=&
\epsilon_{-a}\xi_{-a}^{n}(-\tau^{-1})z^{-\frac{n}{2}}\hits(1,\tfrac{1}{2}).
\end{eqnarray*}
Here, $\chi_{\hat{a},\tau}=(\chi_{\hat{a},\tau}^{(1)},\chi_{\hat{a},\tau}^{(2)})$.
Since $\xi=-1$,
we conclude that
$-\sigmax(\chi_{\hat{a},\tau}^{(2)})^{-1} \hits \chi_{\hat{a},\tau}=\chi_{-\hat{a},-\tau^{-1}}$.
Moreover, since $\chi_{\hat{a},\tau}^{(2)}\,\sigmax(\chi_{\hat{a},\tau}^{(2)})^{-1}=1$,
we get $\chi_{\hat{a},\tau}^{(2)}\,\sigmax(\chi_{\hat{a},\tau}^{(2)})^{-1}\hits\chi_{\hat{a},\tau}=\chi_{\hat{a},\tau}$.
Thus, we are done.
\end{proof}

\begin{lemma} \label{prp:phi(h)}
For $\hat{a}=a'+n\deltaa\in\Deltashr$ and $\tau\in\K^\times$, we get
\[
\Phi(h_{\hat{a}}(\tau))=\begin{cases}
\hsh_{a'}(\tau) & \text{if $a'$ is of type {\rm (R-1)}}, \\
\hsh_{a}(\tau) & \text{if $a$ is of type {\rm (R-2)} or {\rm (R-4)}}, \\
\hsh_a(\chi_{\hat{a},\tau},\,\chi_{\hat{a},-1}) & \text{if $a$ is of type {\rm (R-3)}.} \end{cases}
\]
\end{lemma}
\begin{proof}
If $a$ is of type (R-1), then the claim is trivial.
Suppose that $a$ is of type (R-2).
Then by Lemmas~\ref{prp:wsw}~and~\ref{prp:hsh}, we have
\begin{eqnarray*}
\Phi(h_{\hat{a}}(\tau))
&=&
\wsw_a(\xi_a^{-n}\tau z^{\frac{n}{r}})\wsw_a(-\xi_a^{-n}z^{\frac{n}{r}})
\\ &=&
\wsw_a(\xi_a^{-n}\tau z^{\frac{n}{r}})\wsw_a(-1)\wsw_a(-1)^{-1}\wsw_a(\xi_a^{-n}z^{\frac{n}{r}})^{-1}
\\ &=&
\hsh_a(\xi_a^{-n}\tau z^{\frac{n}{r}})\hsh_a(\xi_a^{-n}z^{\frac{n}{r}})^{-1}
\\ &=&
\hsh_a(\tau).
\end{eqnarray*}
For the case (R-4), the proof is the same as above.
If $a$ is of type (R-3), then by definition 
$\Phi(h_{\hat{a}}(\tau))=\wsw_a(\chi_{\hat{a},\tau})\wsw_a(\chi_{\hat{a},-1})=\hsh_a(\chi_{\hat{a},\tau},\chi_{\hat{a},-1})$.
\end{proof}

Let $\TThk$ be the subgroup of $\Gkms$ generated by $\{h_{\hat{a}}(\tau) \mid \hat{a}\in\Deltashr,\tau\in\K^\times\}$.
\begin{lemma}\label{prp:cent}
The kernel of $\Phi$ is included in $\TThk$.
\end{lemma}
\begin{proof}
For any $x\in\Ker(\Phi)$, we can uniquely write as $x=ywy'$ for some $y,y'\in \BBhk^+$ and $w\in \WWh$ by Proposition~\ref{prp:Bruhat}.
Then we have $1=\Phi(x)=\Phi(y)\Phi(w)\Phi(y')$.
By Proposition~\ref{prp:Iwa-Matsu}, the expression is unique, and hence $\Phi(w)$ should be trivial.
We conclude that $w=1$, see Lemma~\ref{prp:phi(w)}.
Thus, $\Ker(\Phi)$ is included in $\BBhk$.

Again, we take $x\in\Ker(\Phi)$ and express $x = u h$ for some $u\in \UUhk^+$ and $h\in \TThk$.
In particular, $\Phi(u)$ is trivial.
On the other hand, again by Proposition~\ref{prp:Iwa-Matsu}, one sees that the restriction map $\Phi|_{\UUhk}$ is injective.
Hence, we conclude that $u$ is trivial, and $\Ker(\Phi) \subset \TThk$.
\end{proof}

\subsection{Central Extension} \label{sec:cent_ext}
In this subsection, we will show that $\Gkms$ is a one-dimensional central extension of $\Gtw$.

\begin{proposition} \label{prp:phi-surj}
The map $\Phi: \Gkms \to \Gtw$ is surjective.
\end{proposition}
\begin{proof}
By Theorem~\ref{prp:elem}, it is enough to show that the image of $\Phi$ coincides with the subgroup $\E{\Skk}$ of $\Gtw$.
If $r=1$ or $3$, then the claim is trivial.
Assume that $r=2$.
When $(\XX_N,r)\neq(\XA_{2\ell},2)$, we can easily see that $\Phi$ is surjective by definition of $\E{\Skk}$.
Thus, in the following, we will show the claim for $(\XX_N,r)=(\XA_{2\ell},2)$.
In this case, note that $\xi=-1\in\K$, and hence $\Skk=\K[z^{\pm\frac{1}{2}}]$.

We take $\xsx_a(\chi)\in\E{\Skk}$ for some $a\in\Deltasshort$ with $a\ifff\alpha\in\Delta$ and $\chi=(s,s')\in\frakA$.
Note that $2a\ifff \alpha+\sigma(\alpha)\in\Delta$.
We shall write $s=\sum_n s_n z^{\frac{n}{2}}$ and $s'=\sum_n s'_n z^{\frac{n}{2}}$ with some $s_n,s'_n\in\K$.
We define an element $y$ of $\Gkms$ so that
\[
y:=\prod_n x_{a+n\deltaa}(s_n),
\]
where the product is (necessarily) finite, and is taken in the canonical order ($\cdots < -1 < 0 < 1 < 2 < \cdots$).
Then by \eqref{eq:dotplusx}, we can find $g\in\Skk$ such that $\phi:=(s,g)$ belongs to $\mathfrak{A}$ and
\[
\Phi(y)=\prod_n\xsx_{a}\big((s_n z^{\frac{n}{2}},\,\tfrac{1}{2}\xi^ns_n^2z^n)\big)
=\xsx_a\big((\sum_n s_nz^{\frac{n}{2}},\,g)\big)
=\xsx_a(\phi).
\]
We shall denote $g=\sum_n g_n z^{\frac{n}{2}}$ for some $g_n\in\K$ ($n\in\mathbb{Z}$).
If there exists an odd number $m\in2\mathbb{Z}+1$ such that the coefficient $g_m$ of $z^{\frac{m}{2}}$ in $g$ is non-zero, then
\begin{eqnarray*}
\Phi \big(y\cdot x_{2a+m\deltaa}(-g_m N_{\sigma(\alpha),\alpha})\big)
&=&
\xsx_a(\phi)\cdot\xsx_a\big((0,-g_m z^{\frac{m}{2}})\big)
\\ &=&
\xsx_a(\phi\dotplus(0,-g_m z^{\frac{m}{2}}))
\\ &=&
\xsx_a\big((s,g-g_m z^{\frac{m}{2}})\big)
\end{eqnarray*}
(cf. Remark~\ref{rem:x_2}).
Thus, we define an element $y'$ of $\Gkms$ as follows.
\[
y':=\underset{\text{with }g_m\neq0}{\prod_{m\in2\mathbb{Z}+1}}x_{2a+m\deltaa}(-g_m N_{\sigma(\alpha),\alpha}).
\]
Then we can choose $g'\in\Skk$ so that $g'$ has no odd terms, $\phi':=(s,g')$ is in $\frakA$ and $\Phi(y\cdot y')=\xsx_{a}(\phi')$.
Finally, we define an element $y''$ of $\Gkms$ as follows.
\[
y'':=\prod_{m\in 2\mathbb{Z}+1}x_{2a+m\deltaa}(s'_m N_{\sigma(\alpha),\alpha}).
\]
We put $g'':=g'+s'_{\mathsf{odd}}$, where $s'_{\mathsf{odd}}:=\sum_{m\in2\mathbb{Z}+1}s'_m z^{\frac{m}{2}}$.
Then $\phi'':=(s,g'')\in\frakA$ and $\Phi(y\cdot y'\cdot y'')=\xsx_{a}(\phi'')$.
By the definition of $\frakA$, the following relation holds.
\[
g''_{\mathsf{even}}=\tfrac{1}{2}(s_{\mathsf{even}}^2-s_{\mathsf{odd}}^2)=s'_{\mathsf{even}},
\]
where $s'_{\mathsf{even}}:=\sum_{n\in 2\mathbb{Z}} s'_n z^{\frac n 2}$ etc.
Since $g''_{\mathsf{odd}}=s'_{\mathsf{odd}}$, we conclude that $g''=s'$ and $\chi=\phi''$.
In this way, we see $\Phi(y\cdot y'\cdot y'')=\xsx_a(\chi)$.
This proves the claim.
\end{proof}

As in \eqref{eq:-a_0}, we shall write $-a_0=c_1a_1+c_2a_2+\cdots+c_\ell a_\ell$
for some non-negative integers $c_1,c_2,\dots,c_\ell$.
For $(\XX_N,r)\neq(\XA_{2\ell},2)$, we set
\[
\ZZhk := \{ \prod_{p=0}^\ell h_{\hat{a}_p}(\tau_p) \in \TThk \mid \tau_0\in \K^\times,\,\, 
\tau_p = 
\begin{cases}
\tau_0^{r c_p} & \text{if $a_p$ is of type (R-1),} \\
\tau_0^{c_p}   & \text{otherwise.}
\end{cases}
\,\}.
\]
Otherwise, we set
\[
\ZZhk:=\{\prod_{p=0}^\ell h_{\hat{a}_p}(\tau_\ell^2)\in\TThk \mid \tau_\ell\in\K^\times\}.
\]
Obviously, this $\ZZhk$ is a subgroup of the center $Z(\Gkms)$ of the Kac-Moody group $\Gkms$ (as an abstract group) and satisfies $\ZZhk \cong \K^\times$.
In this way, we may identify $\K^\times$ as a subgroup of the center of $\Gkms$.
The non-twisted version (i.e.,~$r=1$) of the following theorem is well-known (cf.~\cite{Chen,Peterson-Kac}).
\begin{theorem} \label{prp:main}
The kernel of $\Phi$ coincides with $\ZZhk$, and hence the sequence $1 \to \K^\times \to \Gkms \overset{\Phi}{\longrightarrow} \Gtw \to 1$ is exact.
\end{theorem}
\begin{proof}
By Lemma~\ref{prp:cent}, any element $h$ in $\Ker(\Phi)$ can be expressed as $h=\prod_{p=0}^\ell h_{\hat{a}_p}(\tau_p) \in \Ker(\Phi)$ for some $\tau_p\in\K^\times$.
First, suppose that $(\XX_N,r)\neq(\XA_{2\ell},2)$.
By Lemma~\ref{prp:phi(h)}, we have
\[
1=\Phi(h)=\hsh_{a_0}(\tau_0) \hsh_{a_1}(\tau_1) \cdots \hsh_{a_\ell}(\tau_\ell).
\]
Then, by Lemma~\ref{prp:twloop_h=1}, $h\in \ZZhk$.

Next, suppose that $(\XX_N,r)=(\XA_{2\ell},2)$.
In this case, we get
\[
1=\Phi(h)=
\hsh_{a_0}(\tau_0)\hsh_{a_1}(\tau_1)\cdots\hsh_{a_{\ell-1}}(\tau_{\ell-1})\cdot\hsh_{a_\ell}(\chi_{a_\ell,\tau_\ell},\,\chi_{a_\ell,-1}).
\]
Then, by Lemma~\ref{prp:twloop_h=1}, we see that $\tau_p=\ccc(\chi_{a_\ell,\tau_\ell},\,\chi_{a_\ell,-1})=\tau_\ell^2$.
Hence, $h\in\ZZhk$.
Thus, we conclude that $\Ker(\Phi)\subset\ZZhk$.
One can easily see the converse $\ZZhk\subset\Ker(\Phi)$ in a similar way.
\end{proof}

We have the following corollary.
\begin{corollary} \label{prp:main1}
$\Gkms/\K^\times\cong\Gtw$.
\end{corollary}

\section{Twisted Affine Kac-Moody Groups}\label{sec:tw-aff-KM-grp}
In this section, 
we define a $\Gamma$-action on the simply-connected affine Kac-Moody group $\Gkmzeta$ of type $\XX_N^{(1)}$ defined over $\K(\xi)$,
and show that the fixed-point subgroup of $\Gkmzeta$ under $\Gamma$ coincides with
the simply-connected affine Kac-Moody group $\Gkms$ of type $\XX_N^{(r)}$ defined over $\K$.

\subsection{Twisted Affine Kac-Moody Groups}
In this subsection, we define an action of the group $\Gamma$ on the affine Kac-Moody group $\Gkmzeta$.
For simplicity, we set $\Deltahr:=\Deltashrid=\{\alpha+n\deltaa\mid \alpha\in \Delta,\,n\in\mathbb{Z}\}$.
For a real root $\hat{\alpha}=\alpha+n\deltaa\in \Deltahr$, we let 
\[
\sigma(\hat{\alpha}):=\sigma(\alpha)+n\deltaa
\quad \text{and} \quad
\omega(\hat{\alpha}):=\omega(\alpha)+n\deltaa.
\]
Note that if $\xi\in\K$, then this $\omega$ is trivial.

Let $\FDeltazeta$ denote the free group with free generating set $\{\hat{x}_{\hat{\alpha}}(\nu)\mid\hat{\alpha}\in\Deltahr,\,\nu\in\K(\xi)\}$.
We define group homomorphisms $\sigmah$ and $\omegah$ form $\FDeltazeta$ to $\Gkmzeta$ as follows.
\begin{equation} \label{eq:sig1}
\sigmah(\hat{x}_{\hat{\alpha}}(\nu)):=x_{\sigma(\hat{\alpha})}(k_\alpha \xi^{-n} \nu)
\quad \text{and} \quad
\omegah(\hat{x}_{\hat{\alpha}}(\nu)):=x_{\omega(\hat{\alpha})}(\omegaxx(\nu)),
\end{equation}
where $\hat{\alpha}=\alpha+n\deltaa\in\Deltahr$ and $\nu\in\K(\xi)$.
Here, $\omegaxx:\K(\xi)\to\K(\xi)$ is a $\K$-automorphism defined by $\omegaxx(\xi)=\xi^{-1}$.
\begin{lemma} \label{prp:sigma(w,h)}
For $\tau\in\K(\xi)^\times$ and $\hat{\alpha}=\alpha+n\deltaa\in\Deltahr$, we have
the following equations.
$\sigmah(\hat{w}_{\hat{\alpha}}(\tau)) = w_{\sigma(\hat{\alpha})}(k_\alpha \xi^{-n} \tau)$,
$\sigmah(\hat{h}_{\hat{\alpha}}(\tau)) = h_{\sigma(\hat{\alpha})}(\tau)$,
$\omegah(\hat{w}_{\hat{\alpha}}(\tau)) = w_{\omega(\hat{\alpha})}(\omegaxx(\tau))$, and
$\omegah(\hat{h}_{\hat{\alpha}}(\tau)) = h_{\omega(\hat{\alpha})}(\omegaxx(\tau))$.
\end{lemma}
\begin{proof}
By Proposition~\ref{prp:k_a}, we note that $k_{\alpha}=k_{-\alpha}=\pm1$.
Thus, we have
\begin{eqnarray*}
\sigmah(\hat{w}_{\hat{\alpha}}(\tau))
&=&
\sigmah(\hat{x}_{\hat{\alpha}}(\tau))
\sigmah(\hat{x}_{-\hat{\alpha}}(-\tau^{-1}))
\sigmah(\hat{x}_{\hat{\alpha}}(\tau))
\\ &=&
x_{\sigma(\hat{\alpha})}(k_\alpha \xi^{-n} \tau)
x_{-\sigma(\hat{\alpha})}(k_{-\alpha} \xi^{n} (-\tau^{-1}))
x_{\sigma(\hat{\alpha})}(k_\alpha \xi^{-n} \tau)
\\ &=&
x_{\sigma(\hat{\alpha})}(k_\alpha \xi^{-n} \tau)
x_{-\sigma(\hat{\alpha})}(-(k_\alpha \xi^{-n} \tau)^{-1})
x_{\sigma(\hat{\alpha})}(k_\alpha \xi^{-n} \tau)
\\ &=&
w_{\sigma(\hat{\alpha})}(k_\alpha \xi^{-n} \tau).
\end{eqnarray*}
Also,
\begin{eqnarray*}
\sigmah(\hat{h}_{\hat{\alpha}}(\tau))
&=&
\sigmah(\hat{w}_{\hat{\alpha}}(\tau))
\sigmah(\hat{w}_{\hat{\alpha}}(-1))
\\ &=&
w_{\sigma(\hat{\alpha})}(k_\alpha \xi^{-n} \tau)
w_{\sigma(\hat{\alpha})}(-1)
w_{\sigma(\hat{\alpha})}(1)
w_{\sigma(\hat{\alpha})}(-k_\alpha \xi^{-n})
\\ &=&
h_{\sigma(\hat{\alpha})}(k_\alpha \xi^{-n} \tau)
h_{\sigma(\hat{\alpha})}(k_\alpha \xi^{-n})^{-1}
\\ &=&
h_{\sigma(\hat{\alpha})}(\tau).
\end{eqnarray*}
Thus, we are done.
For $\omegah$, the proof is essentially the same as above.
\end{proof}

Then by Proposition~\ref{prp:str-const}, we get the following.
\begin{proposition} \label{prp:Gal-act}
The induced automorphisms $\sigmah$ and $\omegah$ on $\Gkmzeta$ are well-defined.
The order of $\sigmah$ (resp.~$\omegah$) coincides with the order of $\sigma$ (resp.~$\omega$).
\end{proposition}
\begin{proof}
We have to check that $\sigmah$ and $\omegah$ preserve the relations (SC-A), (SC-B), (SC-B${}'$), and (SC-C), see Appendix~\ref{sec:sc_KM}.
First, we shall show the claim for $\sigmah$.
For the relations (SC-A) and (SC-C) are trivial.
Hence, we will show that $\sigmah$ preserves (SC-B) and (SC-B${}'$).

For (SC-B), we will show the following equation.
\begin{equation} \label{eq:comm-sigma}
\begin{gathered}\phantom{a}
[ x_{\sigma(\hat{\alpha})}(k_\alpha \xi^{-n} \nu),\, x_{\sigma(\hat{\beta})}(k_\beta \xi^{-m} \mu)]
\\ =
\prod_{\sigma(i\hat{\alpha}+j\hat{\beta})\in Q_{\hat{\alpha},\hat{\beta}}} x_{\sigma(\hat{\alpha}+\hat{\beta})}
\big( c^{i,j}_{\hat{\alpha},\hat{\beta}} k_{\alpha+\beta} \, \xi^{-in-jm} \nu^i \mu^j \big),
\end{gathered}
\end{equation}
where $\nu,\mu\in\K$ and $\hat{\alpha}=\alpha+n\deltaa,\hat{\beta}=\beta+m\deltaa\in\Deltahr$.
If $\sigma=\id$, then nothing to do.
Thus, in the following, we assume that $\sigma\neq\id$.
Suppose that $\XX_N\neq\XA_{2\ell}$.
In this case, it is easy to see that $k_\alpha=k_\beta=k_{\alpha+\beta}=1$ and
$c^{i,j}_{\sigma(\hat{\alpha}),\sigma(\hat{\beta})}=c^{i,j}_{\hat{\alpha},\hat{\beta}}$ by Proposition~\ref{prp:k_a} and Lemma~\ref{prp:N-sigma}.
Then the equation \eqref{eq:comm-sigma} holds.
Suppose that $\XX_N=\XA_{2\ell}$.
In this case, a concrete calculation shows that 
\[
c^{i,j}_{\hat{\alpha},\hat{\beta}}=c^{i,j}_{\alpha,\beta} =
\begin{cases}
N_{\alpha,\beta} & \text{if } i=j=1, \\
0 & \text{otherwise}.
\end{cases}
\]
Then again by Lemma~\ref{prp:N-sigma}, 
we get $c^{i,j}_{\sigma(\hat{\alpha}),\sigma(\hat{\beta})}k_\alpha k_\beta k_{\alpha+\beta}=c^{i,j}_{\hat{\alpha},\hat{\beta}}$.
Hence, the equation \eqref{eq:comm-sigma} also holds.

For (SC-B${}'$), we will show the following equation.
\begin{equation} \label{eq:conj-sigma}
\begin{gathered}
w_{\sigma(\hat{\alpha})}(k_\alpha \xi^{-n} \tau) \cdot x_{\sigma(\hat{\beta})}(k_\beta \xi^{-m} \nu)) 
\cdot w_{\sigma(\hat{\alpha})}(k_\alpha \xi^{-n} \tau)^{-1}
\\ =
x_{s_{\sigma(\hat{\alpha})}(\sigma(\hat{\beta}))}
\big( \eta_{\sigma(\hat{\alpha}),\sigma(\hat{\beta})} k_{s_{\alpha}(\beta)}\,\xi^{-(m-n\beta(H_{\alpha}))} \nu \tau^{-\beta(H_{\alpha})} \big),
\end{gathered}
\end{equation}
where $\nu\in\K$, $\tau\in\K^\times$ and $\hat{\alpha}=\alpha+n\deltaa,\hat{\beta}=\beta+m\deltaa\in\Deltahr$.
Here, we used $\sigma(\beta)(H_{\sigma(\alpha)})=(\sigma(\beta),\sigma(\alpha)^\vee)=(\beta,\alpha^\vee)=\beta(H_{\alpha})$ and
$s_{\sigma(\hat{\alpha})}(\sigma(\hat{\beta}))=\sigma(s_\alpha(\beta))+(m-n\beta(H_{\alpha}))\deltaa$.
If $\sigma=\id$, then nothing to do.
Thus, assume that $\sigma\neq\id$.
As before, if $\XX_N\neq\XA_{2\ell}$, then the equation \eqref{eq:conj-sigma} is easy to see.
Suppose that $\XX_N=\XA_{2\ell}$.
In this case, a concrete calculation shows that 
\[
\eta_{\hat{\alpha},\hat{\beta}}=\eta_{\alpha,\beta}=
\begin{cases}
\pm1 & \text{if $\beta(H_{\alpha})=0$\, (i.e., $s_\alpha(\beta)=\beta$),} \\
\pm N_{\pm \alpha,\beta} & \text{if $\beta(H_{\alpha})=\pm1$\, (i.e., $s_\alpha(\beta)=\beta\pm\alpha$).}
\end{cases}
\]
Then by Lemma~\ref{prp:N-sigma}, one sees that the equation \eqref{eq:conj-sigma} holds.

Finally, using Lemma~\ref{prp:N_omega}, we can also show that $\omegah$ preserves the relations (SC-A), (SC-B), (SC-B${}'$), and (SC-C).
\end{proof}

By Proposition~\ref{prp:Gal-act},
we can consider the subgroup $\langle\sigmah,\omegah\rangle$ of the automorphism group of $\Gkmzeta$ generated by $\sigmah$ and $\omegah$,
and we get an isomorphism $\Gamma\cong\langle\sigmah,\omegah\rangle$ via $\sigma\mapsto \sigmah$ and $\omega\mapsto \omegah$.
Under this identification, $\Gamma$ acts on the group $\Gkmzeta$.
We let $\Gkmtwzeta$ denote the fixed-point subgroup of $\Gkmzeta$ under $\Gamma$.

\subsection{Special Elements in Twisted Affine Kac-Moody Groups}
As in Section~\ref{sec:tw_chev}, we define some special elements in $\Gkmzeta$.

For $\nu\in\K(\xi)$ and $\hat{a}=a'+n\deltaa\in\Deltashr$ with $a'\ifff\alpha\in\Delta$,
we define an element $\xs_{\hat a}(\nu)$ of $\Gkmzeta$ as follows.
\begin{description}
\item[($\hat{\mbox{G}}$-1)] $\xs_{\hat{a}}(\nu) :=
x_{\hat{\alpha}}(\nu)$ if $a'$ is of type (R-1).
\item[($\hat{\mbox{G}}$-2)] $\xs_{\hat{a}}(\nu) :=
x_{\hat{\alpha}}(\nu)x_{\sigma(\hat{\alpha})}(\xi^{-n}\nu)$ if $a$ is of type (R-2).
\item[($\hat{\mbox{G}}$-3)] $\xs_{\hat{a}}(\nu) := 
x_{\hat{\alpha}}(\nu)x_{\sigma(\hat{\alpha})}(\xi^{-n}\nu)
x_{\hat{\alpha}+\sigma(\hat{\alpha})}(\tfrac{1}{2}N_{\sigma(\alpha),\alpha} \xi^{-n}\nu^2)$ if $a$ is of type (R-3).
\item[($\hat{\mbox{G}}$-4)] $\xs_{\hat a}(\nu) :=
x_{\hat{\alpha}}(\nu)x_{\sigma(\hat{\alpha})}(\xi^{-n}\nu)x_{\sigma^2(\hat{\alpha})}(\xi^{-2n}\nu)$ if $a$ is of type (R-4).
\end{description}

\begin{lemma}
These elements belong to $\Gkmtwzeta$.
\end{lemma}
\begin{proof}
If $\sigmah=\id$, then nothing to do.
We shall show the claim one-by-one for $\sigmah\neq\id$.
In the following, let $\alpha\in\Delta$ be the corresponding root $a\ifff\alpha\in\Delta$.

Suppose that $a$ is of type (R-1).
By definition, $\sigmah(\xs_{\hat{a}}(\nu))=x_{\alpha+n\deltaa}(k_\alpha \xi^{-n}\nu)$.
First, assume that $(\XX_N,r)\neq(\XA_{2\ell},2)$.
Then by Proposition~\ref{prp:k_a}, $k_\alpha=1$.
Since $\hat{a}\in\Deltashr$ and $a\in\Deltaslong$, the integer $n$ should be divided by $r$.
Thus, $\xi^{-n}=1$, and hence $\sigmah(\xs_{\hat{a}}(\nu))= \xs_{\hat{a}}(\nu)$.
Next, assume that $(\XX_N,r)=(\XA_{2\ell},2)$.
By Proposition~\ref{prp:k_a}, we have $k_\alpha=-1$.
Hence, we have to show that $\xi^{-n}=(-1)^{-n}=-1$.
Since $\hat{a}\in\Deltashr$ and $a'=2a$ is of type (R-1), we see $n\in2\mathbb{Z}+1$. 
Thus we are done.

Suppose that $a$ is of type (R-2).
Since $k_\alpha=1$, we get
\[
\sigmah(\xs_{\hat{a}}(\nu))=x_{\sigma(\alpha)+n\deltaa}(\xi^{-n}\nu)x_{\alpha+n\deltaa}(\nu).
\] 
This proves the claim, since $\alpha+\sigma(\alpha)\notin\Delta$ and the product is commutative.

Suppose that $a$ is of type (R-3).
By Proposition~\ref{prp:k_a}, we get $k_\alpha=1$ and $k_{\alpha+\sigma(\alpha)}=-1$.
Since $\hat{\alpha}+\sigma(\hat{\alpha})=\alpha+\sigma(\alpha)+2n\deltaa$, we have
\begin{eqnarray*}
&& \sigmah(\xs_{\hat{a}}(\nu))
\\ &=&
x_{\sigma(\hat{\alpha})}(\xi^{-n}\nu)
x_{\hat{\alpha}}( \xi^{-2n}\nu)
x_{\hat{\alpha}+\sigma(\hat{\alpha})}(-\tfrac{1}{2}N_{\sigma(\alpha),\alpha}\xi^{-n}\nu^2)
\\ &=&
x_{\sigma(\hat{\alpha})+\hat{\alpha}} (c^{1,1}_{\sigma(\alpha),\alpha}\xi^{-n}\nu^2)
x_{\hat{\alpha}}(\nu)
x_{\sigma(\hat{\alpha})}(\xi^{-n}\nu)
x_{\hat{\alpha}+\sigma(\hat{\alpha})}(-\tfrac{1}{2}N_{\sigma(\alpha),\alpha}\xi^{-n}\nu^2)
\\ &=&
x_{\hat{\alpha}}(u)
x_{\sigma(\hat{\alpha})}(\xi^{-n}\nu)
x_{\hat{\alpha}+\sigma(\hat{\alpha})}\big((c^{1,1}_{\sigma(\hat{\alpha}),\hat{\alpha}}-\tfrac{1}{2}N_{\sigma(\alpha),\alpha})\xi^{-n}\nu^2\big).
\end{eqnarray*}
Here, we used the commutation law, see \eqref{eq:cij}.
Since $c^{1,1}_{\sigma(\hat{\alpha}),\hat{\alpha}} = c^{1,1}_{\sigma(\alpha),\alpha} = N_{\sigma(\alpha),\alpha}$, we are done.

Suppose that $a$ is of type (R-4).
Then by the same calculation for (R-2) shows $\sigmah(\xs_{\hat{a}}(\nu)) = \xs_{\hat{a}}(\nu)$.
Since $\omega\sigma\omega=\sigma^2$ and $\omegaxx(\xi^{-n}\nu) = \xi^{n}\nu$, we get 
\[
\omegah(\xs_{\hat{a}}(\nu)) =
x_{\omega(\hat{\alpha})}(\omegaxx(\nu))\,x_{\omega(\sigma(\hat{\alpha}))}(\xi^{n}\nu)\,x_{\omega(\sigma^2(\hat{\alpha}))}(\xi^{2n}\nu)
=\xs_{\hat{a}}(\nu),
\]
see the proof of Lemma~\ref{prp:xsx}.
This completes the proof.
\end{proof}

For $\tau\in\K(\xi)^\times$ and $\hat{a}=a'+n\deltaa\in\Deltashr$ with $a'\ifff\alpha\in\Delta$, 
we define elements $\ws_{\hat{a}}(t)$ and $\hs_{\hat{a}}(\tau)$ of $\Gkmzeta$ as follows.
\begin{description}
\item[($\hat{\mbox{W}}$-1)] $\ws_{\hat{a}}(\tau):=w_{\hat{\alpha}}(\tau)$ if $a'$ is of type (R-1).
\item[($\hat{\mbox{W}}$-2)] $\ws_{\hat{a}}(\tau):=\xs_{\hat{a}}(\tau)\xs_{-\hat{a}}(-\xi^{n}\tau^{-1})\xs_{\hat{a}}(\tau)$ if $a$ is of type (R-2).
\item[($\hat{\mbox{W}}$-3)] $\ws_{\hat{a}}(\tau):=\xs_{\hat{a}}(\tau)\xs_{-\hat{a}}(-2N_{\sigma(\alpha),\alpha}\xi^n\tau^{-1})\xs_{\hat{a}}(\tau)$ if $a$ is of type (R-3).
\item[($\hat{\mbox{W}}$-4)] $\ws_{\hat{a}}(\tau):=\xs_{\hat{a}}(\tau)\xs_{-\hat{a}}(-\tau^{-1})\xs_{\hat{a}}(\tau)$ if $a$ is of type (R-4).
\item[($\hat{\mbox{H}}$)] $\hs_{\hat{a}}(\tau):=\ws_{\hat{a}}(\tau)\ws_{\hat{a}}(-1)$ for all types.
\end{description}
Note that if $a$ is of type (R-3), then $N_{\sigma(\alpha),\alpha}=\pm1$, and hence $(N_{\sigma(\alpha),\alpha})^{-1}=N_{\sigma(\alpha),\alpha}$.

Then a direct calculation shows the following.
\begin{lemma} \label{prp:ws,hs}
Elements $\xs_{\hat{a}}(\nu)$, $\ws_{\hat{a}}(\tau)$, and $\hs_{\hat{a}}(\tau)$ satisfy the relations 
(SC-A), (SC-B), (SC-B${}'$), and (SC-C) in $\Gkmtwzeta$.
\end{lemma}

\subsection{Galois Descent Formalism}
In this subsection, we will show that
the fixed-point subgroup $\Gkmtwzeta$ of $\Gkmzeta$ under $\Gamma$ is isomorphic to $\Gkms$.

If we consider the case when $r=1$, then we have constructed the following surjective group homomorphism (see \eqref{eq:Phi(x)}).
\[
\Gkmzeta\longrightarrow\G(\K(\xi)\otimeskk\Rkk);
\quad x_{\hat{\alpha}}(\nu)\longmapsto x_{\alpha}(\nu\otimeskk z^{n}),
\]
where $\hat{\alpha}=\alpha+n\deltaa\in\Deltahr$ and $\nu\in\K(\xi)$.
Since $\G$ is a group scheme, we can consider the natural group isomorphism $\G(\K(\xi)\otimeskk\Rkk)\to\G(\Skk)$
induced from the canonical algebra isomorphism $\K(\xi)\otimeskk\Rkk\to\Skk=\K(\xi)[z^{\pm\frac{1}{r}}]$
given by $\nu\otimes z^n\mapsto \nu z^{\frac{n}{r}}$ ($\nu\in\K(\xi)$, $n\in\mathbb{Z}$).
Then by compositing above maps, we get the following surjective group homomorphism.
\[
\Psi:\Gkmzeta\longrightarrow\G(\K(\xi)\otimeskk\Rkk)\cong\G(\Skk);
\quad x_{\hat{\alpha}}(\nu)\longmapsto x_{\alpha}(\nu z^{\frac{n}{r}}),
\]
where $\hat{\alpha}=\alpha+n\deltaa\in\Deltahr$ and $\nu\in\K(\xi)$.
It is easy to see that this preserves the $\Gamma$-action, see \eqref{eq:sig1} and Section~\ref{sec:tw_chev}.
Hence, by taking the fixed-point functor $(-)^\Gamma$ to $\Psi$, we have the following group homomorphism.
\[
\Psis:\Gkmtwzeta\longrightarrow\Gfix=\Gtw.
\]
Since $(-)^\Gamma$ is left exact, we have $\Ker(\Psis)\cong\K^\times$ by Theorem~\ref{prp:main}.
\medskip

On the other hand, we shall consider the following group homomorphism $\Theta:\FDeltas\to\Gkmtwzeta$.
\[
\Theta(\hat{x}_{\hat{a}}(\nu)) := \begin{cases}
\xs_{\hat{a}}(\nu) & \text{if $a'$ is of type (R-1)}, \\
\xs_{\hat{a}}(\xi_a^{-n}\nu) & \text{if $a$ is of type (R-2)}, \\
\xs_{\hat{a}}(\epsilon_a\xi_a^{-n}\nu) & \text{if $a$ is of type (R-3)}, \\
\xs_{\hat{a}}(\nu) & \text{if $a$ is of type (R-4)},
\end{cases}
\]
where $\hat{a}=a'+n\deltaa\in\Deltashr$ and $\nu\in\K$.
For $\xi_a$ and $\epsilon_a$, see Notation~\ref{not:e_a}.

\begin{lemma} \label{prp:Theta}
The induced map $\Theta:\Gkms\to\Gkmtwzeta$ is well-defined.
\end{lemma}
\begin{proof}
For each $\hat{a}\in\Deltashr$ and $\tau\in\K^\times$.
One easily sees that
\[
\Theta(\hat{w}_{\hat{a}}(\tau)) = \begin{cases}
\ws_{\hat{a}}(\xi_a^{-n}\tau) & \text{if $a$ is type of (R-2)}, \\
\ws_{\hat{a}}(\tau) & \text{otherwise} \end{cases}
\]
and $\Theta(\hat{h}_{\hat{a}}(\tau))=\hs_{\hat{a}}(\tau)$.
Thus, by Lemma~\ref{prp:ws,hs}, one easily sees that
the map $\Theta$ preserves the relations (SC-A), (SC-B), (SC-B${}'$), and (SC-C).
\end{proof}

\begin{lemma}
The following diagram is commutative.
\[
\begin{xy}
(0,8)*++{\Gkmtwzeta}="1",(35,8)*++{\Gfix}="2",(0,-8)*++{\Gkms}="3",(35,-8)*++{\Gtw.}="4",
{"1"\SelectTips{cm}{}\ar@{->}^{\Psis}"2"}, {"3"\SelectTips{cm}{}\ar@{->}_{\Theta}"1"},
{"2"\SelectTips{cm}{}\ar@{=}"4"},{"3"\SelectTips{cm}{}\ar@{->}_{\Phi}"4"}
\end{xy}
\]
\end{lemma}
\begin{proof}
Let $\hat{a}=a'+n\deltaa\in\Deltashr$ and $\nu\in\K$.
For an (R-1) type root $a'$,
it is easy to see that $(\Psi^\Gamma\circ\Theta)(x_{\hat{a}}(\nu))=x_{\alpha}(\nu z^{\frac{n}{r}})=\Phi(x_{\hat{a}}(\nu))$.
If $a$ is of type (R-2), then
\begin{eqnarray*}
(\Psi^\Gamma\circ\Theta) (x_{\hat{a}}(\nu)) = \Psi^\Gamma(\tilde{x}_{\hat{a}}(\xi_a^{-n}\nu))
&=& \Psi\big(x_{\hat{\alpha}}(\xi_a^{-n}\nu)x_{\sigma(\hat{\alpha})}(\xi_a^{-n}\xi^{-n}\nu)\big) \\
&=& x_{\alpha}(\xi_a^{-n}\nu z^{\frac{n}{r}})x_{\sigma(\alpha)}(\xi_a^{-n}\xi^{-n}\nu z^{\frac{n}{r}}) \\
&=& \tilde{x}_a(\xi_a^{-n}\nu z^{\frac{n}{r}}) = \Phi(x_{\hat{a}}(\nu)).
\end{eqnarray*}
Next, suppose that $a$ is of type (R-3).
Then $(\Psi^\Gamma\circ\Theta)(x_{\hat{a}}(\nu))$ is calculated as follows.
\begin{eqnarray*}
&& \Psi^\Gamma(\tilde{x}_{\hat{a}}(\epsilon_a\xi_a^{-n}\nu)) \\
&=& \Psi\big( 
x_{\hat{\alpha}}(\epsilon_a\xi_a^{-n}\nu) 
x_{\sigma(\hat{\alpha})}(\epsilon_a\xi_a^{-n}\xi^{-n}\nu)
x_{\hat{\alpha}+\sigma(\hat{\alpha})}(\tfrac12 N_{\sigma(\alpha),\alpha}\epsilon_a^2\xi_a^{-2n}\xi^{-n}\nu^2)
\big)
\\
&=& 
x_{\alpha}(\epsilon_a \xi_a^{-n}\nu z^{\frac{n}{r}}) 
x_{\sigma(\alpha)}(\epsilon_a\xi_a^{-n}\xi^{-n}\nu z^{\frac{n}{r}})
x_{\alpha+\sigma(\alpha)}(\tfrac12 N_{\sigma(\alpha),\alpha}\epsilon_a^2\xi^{-n}\nu^2 z^{\frac{2n}{r}}).
\end{eqnarray*}

On the other hand, by definition,
\[
\Phi(x_{\hat{a}}(\nu)) = \tilde{x}_a(\epsilon_a\xi_a^{-n}\nu z^{\frac{n}{2}} \hits (1, \tfrac{1}{2}))
= \tilde{x}_a(\epsilon_a\xi_a^{-n}\nu z^{\frac{n}{2}}, \tfrac{1}{2} \epsilon_a^2\xi^{-n}\nu^2 z^n).
\]
Since $r=2$, we conclude that $(\Psi^\Gamma\circ\Theta)(x_{\hat{a}}(\nu))=\Phi(x_{\hat{a}}(\nu))$.
Finally, if $a$ is of type (R-4),
then the similar calculus for (R-2) shows that the equation $(\Psi^\Gamma\circ\Theta)(x_{\hat{a}}(\nu))=\Phi(x_{\hat{a}}(\nu))$ also holds.
\end{proof}

Thus, we have the following commutative diagram:
\[
\begin{xy}
(0,8)*++{\Gkmtwzeta}="1", (35,8)*++{\Gfix}="2", (0,-8)*++{\Gkms}="3", (35,-8)*++{\Gtw}="4",
(-25,8)*++{\K^\times}="11", (-50,8)*++{1}="22", (-25,-8)*++{\K^\times}="33", (-50,-8)*++{1}="44", (55,-8)*++{1.}="9",
{"1"\SelectTips{cm}{}\ar@{->}^{\Psis} "2"}, {"22" \SelectTips{cm}{}\ar@{->}"11"}, {"11"\SelectTips{cm}{}\ar@{^(->}"1"},
{"44"\SelectTips{cm}{}\ar@{->}"33"}, {"33"\SelectTips{cm}{}\ar@{^(->}"3"}, {"3"\SelectTips{cm}{}\ar@{->}_{\Theta}"1"},
{"2"\SelectTips{cm}{}\ar@{=}"4"}, {"11"\SelectTips{cm}{}\ar@{=}"33"}, {"3"\SelectTips{cm}{}\ar@{->>}_{\Phi}"4"}, {"4"\SelectTips{cm}{}\ar@{->}"9"}
\end{xy}
\]
By Theorem~\ref{prp:main}, we see that the bottom sequence is exact.
In particular, $\Psis$ is surjective and the upper sequence is also exact.
Thus, $\Theta$ should be bijective, and hence we have the following result.
\begin{theorem} \label{prp:main2}
$\Gkmtwzeta \cong \Gkms$.
\end{theorem}

\appendix
\section{Kac-Moody Groups} \label{sec:KMgrp}
There are several ways to construct a ``Kac-Moody group'' associated to a given Kac-Moody algebra,
\cite{Kum,Mth,MP,Mor01,Peterson-Kac,Tit2} etc.
In our paper, we have used a representation theoretic approach (\'a la Chevalley).
In this appendix, we give the definition and review some basic properties.

\subsection{Kac-Moody Algebras} \label{sec:KM-alg}
In the following, we fix a natural number $\elln$ and put $I:=\{1,\dots,\elln\}$.
An integer matrix $A=(a_{ij})_{i,j\in I}$ is called a {\it generalized Cartan matrix} ({\it GCM} for short) if it satisfies the following conditions:
$a_{ii}=2$ for $i\in I$,
$a_{ij}\leq0$ for $i\neq j\in I$,
and $a_{ij}=0\Leftrightarrow a_{ji}=0$ for $i,j\in I$.
We let $(\Lambda,\Lambda^\vee,\Pi=\{\alpha_i\}_{i\in I},\Pi^\vee=\{h_i\}_{i\in I})$ be a realization of $A$ over $\mathbb{Z}$,
that is, $\Lambda$ is a free abelian group of finite rank,
$\Lambda^\vee$ is the $\mathbb{Z}$-dual of $\Lambda$, and
$\Pi$ (resp.~$\Pi^\vee$) is $\mathbb{Z}$-free in $\Lambda$ (resp.~$\Lambda^\vee$)
satisfying $\alpha_j(h_i)=a_{ij}$ for all $i,j\in I$.
We note that the rank $\elln'$ of $\Lambda$ satisfies $\elln'\geq \elln$.
We put $\h:=\Lambda^\vee\otimes_{\mathbb{Z}}\mathbb{C}$ and
extend $\Pi^\vee$ to a $\mathbb{C}$-basis of $\h$, say, $\{h_{i'}\}_{i'\in I'}=\{h_1,\dots,h_\elln,\dots,h_{\elln'}\}$,
where $I':=\{1,\dots,\elln'\}$.

A {\it Kac-Moody algebra} $\g$ defined over $\mathbb{C}$ (associated to the triple $(\h,\Pi,\Pi^\vee)$) is the Lie algebra defined over $\mathbb{C}$
generated by a set $\{h_{i'}, e_i, f_i\}_{i'\in I',i\in I}$, called the {\it Chevalley generators} of $\g$ (cf.~\cite{Kac}),
subject to the following relations:
\[
[h_{i'},h_{j'}]=0,\,\,
[h_{i'},e_j]=\alpha_j(h_{i'}) e_j,\,\,
[h_{i'},f_j]=-\alpha_j(h_{i'}) f_j,
\]
\[
[e_i,f_j]=\delta_{i,j} h_i,\,\,
\ad (e_{p})^{1-a_{pq}}(e_{q})=0,\,\,
\ad (f_{p})^{1-a_{pq}}(f_{q})=0
\]
for $i',j'\in I'$, $i,j,p,q\in I$ with $p\neq q$.
Here, $\delta_{i,j}$ is Kronecker's delta and $\ad:\g\to \End_{\mathbb{C}}(\g);X\mapsto(Y\mapsto[X,Y])$ is the adjoint representation on $\g$.

One easily sees that $\h$ is a Lie subalgebra of $\g$, called the {\it Cartan subalgebra} of $\g$.
Let $\g_+$ (resp.~$\g_-$) be the Lie subalgebra of $\g$ generated by $\{e_i\}_{i\in I}$ (resp.~$\{f_i\}_{i\in I}$).
Then we get the triangular decomposition $\g=\g_+\oplus\h\oplus\g_-$.
The $\mathbb{Z}$-submodule $Q:=\sum_{i\in I}\mathbb{Z}\alpha_i$ of $\Lambda$ is called the {\it root lattice}.
For latter use, we put $Q_\pm:=\{\alpha\in Q \mid \alpha\neq0\text{ and } \pm\alpha \in\sum_{i\in I}\mathbb{Z}_{\geq0}\alpha_i\}$,
where $\mathbb{Z}_{\geq0}:=\{0,1,2,3,\dots\}\subset \mathbb{Z}$.
For $\alpha\in Q$, we set the subspace $\g_\alpha:=\{X\in \g \mid [H,X]=\alpha(H)X \text{ for all } H\in\h\}$ of $\g$
is called the {\it root space} of $\g$ corresponding to $\alpha$.
The set $\Delta:=\{ \alpha \in Q \mid \alpha\neq0\text{ and }\g_\alpha\neq 0 \}$ is called the {\it root system} of $\g$ with respect to $\h$.
Set $\Delta_\pm:=\Delta\cap Q_\pm$.

Let $\h^*$ denote the $\mathbb{C}$-dual of $\h$.
For $\alpha_i\in\Pi$, we define $s_i\in\Aut_{\mathbb{C}}(\h^*)$ so that
$s_i(\lambda):=\lambda-\lambda(h_i)\alpha_i$ for each $\lambda\in\h^*$.
The subgroup $\WW$ of $\Aut_{\mathbb{C}}(\h^*)$ generated by the set $\{s_i\}_{i\in I}$ is called the {\it Weyl group} of $\g$.
The set of all {\it real roots} $\Deltar$ is given by $\WW(\Pi):=\{w(\alpha_i) \mid w\in \WW, \alpha_i\in\Pi\}$.
Set $\Deltar_\pm:=\Deltar \cap \Delta_\pm$.

\subsection{Admissible Lattices}\label{sec:adm-pair}
Let $\mathcal{U}(\g)$ be the universal enveloping algebra over $\mathbb{C}$ of $\g$.
For $x\in\mathcal{U}(\g)$ and a natural number $m$,
we set $x^{(m)}:=\frac{1}{m!}x^m$ and $\binom{x}{m}:=\frac{1}{m!}\prod_{i=0}^{m-1}(x-i)$,
where the product is taken in an arbitrary order.
Put $x^{(0)}:=\binom{x}{0}:=1$ for simplicity.
Let $\mathcal{U}_{\mathbb{Z}}$ be the $\mathbb{Z}$-subalgebra of $\mathcal{U}(\g)$
generated by the set
\[
\{\binom{h}{m}, e_i^{(m)}, f_i^{(m)} \mid h\in\Lambda^\vee, i\in I, m\in\mathbb{Z}_{\geq0} \}.
\]
By \cite{Tit}, this algebra is a {\it $\mathbb{Z}$-form} of $\mathcal{U}(\g)$, that is,
the canonical map $\mathcal{U}_{\mathbb{Z}}\otimes_{\mathbb{Z}}\mathbb{C}\to \mathcal{U}(\g)$
is an isomorphism of Lie algebras over $\mathbb{C}$.

Since each $e_i$ ($i\in I$) acts locally nilpotently on $\g$ under the adjoint representation, we get an automorphism on $\g$ as follows
\[
\exp(\ad(e_i)):\g\longrightarrow\g;\quad
X \longmapsto \sum_{m=0}^\infty \frac{1}{m!}\,\ad(e_i)^m(X).
\]
Similarly, we see that $\exp(\ad(f_i))\in\Aut_{\mathbb{C}}(\g)$.
For each $i\in I$,
the restriction of the automorphism $\exp(\ad(e_i))\exp(\ad(f_i))^{-1}\exp(\ad(e_i))$ to $\h$ coincides with the contragredient representation of $s_i$,
and hence we shall identify both.
The action of the Weyl group $\WW$ on $\g$ naturally extends to $\mathcal{U}(\g)$.
Then the $\mathbb{Z}$-algebra $\mathcal{U}_{\mathbb{Z}}$ is invariant under the action of $\WW$, see~\cite[Section~4.5]{Tit}.

\medskip

A representation $\g\to \mathsf{End}_{\mathbb{C}}(V)$ is said to be {\it $\h$-diagonalizable}
if $V$ is the direct sum of its $\h$-weight spaces $V=\bigoplus_{\lambda\in\h^*}V_\lambda$,
where $V_\lambda:=\{v\in V \mid H.v=\lambda(H)v \text{ for all } H\in\h \}$.
Set $\Lambda(V):=\{ \lambda \in \h^* \mid V_\lambda\neq0 \}$.
A $\h$-diagonalizable representation $V$ of $\g$ is said to be {\it integrable}
if $e_i$ and $f_i$ act locally nilpotently on $V$ for all $i\in I$.

\begin{definition}[Cf.~\mbox{\cite[Section~27]{Hum}}] \label{def:adm-pair}
Let $V$ be a representation of $\g$, and let $V_\mathbb{Z}$ be a $\mathbb{Z}$-submodule of $V$.
A pair $(V,V_\mathbb{Z})$ is said to be {\it admissible} if
\begin{itemize}
\item $V$ is integrable,
\item $V_{\mathbb{Z}}$ is a $\mathbb{Z}$-form $V_{\mathbb{Z}}$ of $V$, that is,
the canonical map $V_{\mathbb{Z}}\otimes_{\mathbb{Z}}\mathbb{C}\to V$ is an isomorphism of $\mathbb{C}$-vector spaces, and
\item
$V_{\mathbb{Z}}$ is invariant under the action of $\mathcal{U}_{\mathbb{Z}}$.
\end{itemize}
In this case, we call $V_{\mathbb{Z}}$ an {\it admissible lattice} in $V$.
\end{definition}

\begin{example}\label{ex:adm-pair}
We see some basic examples.
\begin{enumerate}
\item\label{ex:adm-pair-ad}
By the adjoint action, $\g$ is an integrable representation of $\g$.
Then $\g_{\mathbb{Z}}:=\g \cap \mathcal{U}_{\mathbb{Z}}$ is an admissible lattice in $\g$.
\item\label{ex:adm-pair-hw}
For $\lambda\in\h^*$,
we let $V^\lambda$ denote an integrable irreducible representation $V^{\lambda}$ of $\g$ with highest weight $\lambda$.
Then the $\mathcal{U}_{\mathbb{Z}}$-submodule $V_{\mathbb{Z}}^{\lambda}$ of $V^{\lambda}$ generated by a highest weight vector of $V^{\lambda}$
is an admissible lattice in $V^{\lambda}$ (cf.~\cite[Corollary~1 in Chapter~2]{Ste}).
\item\label{ex:adm-pair-sc}
For each $i\in I$, we let $\lambda_i$ be a fundamental weight, and let $V^{\lambda_i}$ be as in~\eqref{ex:adm-pair-hw} above.
Let $V_{\mathsf{sc}}$ be the direct sum of all $V^{\lambda_i}$ for $i\in I$.
Then the direct sum of all $V_{\mathbb{Z}}^{\lambda_i}$ for $i\in I$ is an admissible lattice in $V$.
\end{enumerate}
\end{example}

For an integrable representation $V$ of $\g$,
we let $\Xi(V)$ be the submodule of $\mathbb{Z}^\elln$ generated by
$\{(\lambda(h_i))_{i\in I} \in\mathbb{Z}^\elln \mid \lambda \in \Lambda(V) \}$, called the {\it valued weight lattice} of $V$.
As in \cite[Lemma~27]{Ste}, we have the relation $\Xi(\g) \subset \Xi(V) \subset \Xi(V_{\mathsf{sc}}) = \mathbb{Z}^\elln$.

\subsection{Chevalley Pairs}
A Lie algebra automorphism $\rho:\g\to\g$ is called the {\it Chevalley involution} of $\g$ if it satisfies
\[
\rho(H) = -H,\quad
\rho(e_i) = -f_i, \quad \text{and} \quad
\rho(f_i) = -e_i
\]
for all $H\in\h$ and $i\in I$ (cf.~\cite[Section~1.3]{Kac}).

For a real root $\alpha\in\Deltar$, recall that $\g_\alpha$ is the root space of $\g$ corresponding to $\alpha$.
By definition, there exists $w\in\WW$ and $\alpha_i\in\Pi$ such that $\alpha=w(\alpha_i)$.
We set $H_\alpha:=w(h_i)$.
Note that $H_{\alpha_i}=h_i$ for all $i\in I$.
\begin{definition}\label{def:chev-pair}
For a real root $\alpha\in\Deltar$,
a pair $(X,Y)\in\g_\alpha\times\g_{-\alpha}$ is called a {\it Chevalley pair} for $\alpha$ if it satisfies $\rho(X)=-Y$ and $[X,Y]=H_\alpha$.
\end{definition}
For a real root $\alpha\in\Deltar$ with $\alpha=w(\alpha_i)$ for some $w\in\WW$ and $\alpha_i\in\Pi$, we set
\[
X_\alpha := w(e_i) \quad\text{and}\quad X_{-\alpha} := w(f_i).
\]
Then it is easy to see that $(X_\alpha, X_{-\alpha})$ is a Chevalley pair for $\alpha$.
Using the Chevalley involution, a Chevalley pair uniquely exists up to sign (cf.~\cite{Mor87}).
Since $\mathcal{U}_{\mathbb{Z}}$ is stable under the action of $\WW$, we get the following lemma.
\begin{lemma}\label{prp:chev-U_Z}
For $\alpha\in\Deltar$ and $m\in\mathbb{Z}_{\geq0}$, we have $X_\alpha^{(m)} \in\mathcal{U}_{\mathbb{Z}}$.
\end{lemma}

For $\alpha,\beta\in\Deltar$, we define a scalar $N_{\alpha,\beta}$ by $[X_\alpha,X_\beta]=N_{\alpha,\beta} X_{\alpha+\beta}$.
We set $N_{\alpha,\beta}=0$ if $\alpha+\beta\not\in\Deltar$, for simplicity.
By \cite[Theorem~1]{Mor87}, the scalar $N_{\alpha,\beta}$ is an integer.
If the GCM $A$ is of finite type,
then the set $\{X_\alpha \mid \alpha\in\Deltar\}\sqcup\{ H_{\alpha_i} \mid \alpha_i\in\Pi\}$ forms a Chevalley basis of $\g$.

\medskip

As in Example~\ref{ex:adm-pair}~\eqref{ex:adm-pair-ad},
we consider the adjoint representation on $\g$ and $\g_\mathbb{Z}=\g\cap\mathcal{U}_{\mathbb{Z}}$.
For $\nu\in\mathbb{Z}$, we can define an isomorphism
$\exp(\nu \ad(X_\alpha)):\g_{\mathbb{Z}}\to\g_\mathbb{Z}$ of $\mathbb{Z}$-modules
by $X\mapsto\sum_{m=0}^\infty\frac{\nu^m}{m!}\,\ad(X_\alpha)^m (X)$ (cf.~Lemma~\ref{prp:chev-U_Z}).

For $\alpha,\beta\in\Deltar$ with $\alpha\neq\pm\beta$, we suppose
\[
Q_{\alpha,\beta} := \{ i\alpha + j\beta \in Q \mid i,j=1,2,3,\dots \} \cap \Delta \,\, \subset\,\, \Deltar.
\]
Then by \cite[Theorem~2]{Mor87}, we can find integers $c^{i,j}_{\alpha,\beta}$ such that
\begin{equation} \label{eq:cij}
[\exp(\nu \ad(X_\alpha)), \, \exp(\mu \ad(X_\beta))]
=\prod_{i\alpha+j\beta\in Q_{\alpha,\beta}}
\exp(c^{i,j}_{\alpha,\beta}\,\nu^i \mu^j \ad(X_{i\alpha+j\beta}))
\end{equation}
for all $\nu,\mu\in\mathbb{Z}$,
where the product is taken in an arbitrary order.
Here, $[\;,\;]$ denotes the commutator, that is, $[x,y]:=xyx^{-1}y^{-1}$ for two elements $x,y$ of a group.
Note that $c^{1,1}_{\alpha,\beta}$ coincides with $N_{\alpha,\beta}$.

For $\alpha,\beta\in\Deltar$, there exits a scalar $\eta_{\alpha,\beta}$ such that
\begin{equation} \label{eq:eta}
\exp(\ad(X_\alpha))\exp(-\ad(X_{-\alpha}))\exp(\ad(X_\alpha)) (X_\beta) = \eta_{\alpha,\beta} X_{s_\alpha(\beta)},
\end{equation}
where $s_\alpha(\beta):=\beta-\beta(H_\alpha)\alpha$ is the simple reflection.
By the property of Chevalley pairs, one sees that $\eta_{\alpha,\beta}=\pm1$.

\subsection{Kac-Moody Groups} \label{sec:sc_KM}
Let $(V,V_\mathbb{Z})$ be an admissible pair of $\g$.
For a fixed field $\K$, we set $V_\K := V_{\mathbb{Z}} \otimes_{\mathbb{Z}} \K$.
For $\alpha\in\Deltar$ and $\nu\in\K$, we can define an element $x_\alpha(\nu)$ of $\Aut_\K(V_{\K})$ as follows.
\begin{equation}\label{eq:unip-elem}
x_\alpha (\nu) \,:\, V_{\K} \longrightarrow V_{\K};
\quad v\otimes_{\mathbb{Z}}1 \longmapsto
\sum_{m=0}^\infty X_\alpha^{(m)} . v \otimes_{\mathbb{Z}} \nu^m,
\end{equation}
see Lemma~\ref{prp:chev-U_Z}.
Here, $X_\alpha^{(m)} . v$ denotes the action of $X_\alpha^{(m)}$ on $v$.

\begin{definition} \label{def:KM_grp}
The {\it Kac-Moody group $\GGk$} defined over $\K$ associated to the admissible pair $(V,V_\mathbb{Z})$
is the subgroup of $\Aut_{\K}(V_{\K})$ generated by the set $\{ x_\alpha(\nu) \mid \alpha\in \Deltar,\, \nu\in\K \}$.
We prepare some terminology as follows.
\begin{itemize}
\item
$\GGk$ is said to be of {\it adjoint type} if we take $V=\g$ the adjoint representation.
\item
$\GGk$ is said to be of {\it simply connected type} if we take $V=V_{\mathsf{sc}}$, or equivalently $\Xi(V)=\mathbb{Z}^\elln$.
\item
$\GGk$ is {\it affine} if $\g$ is affine (i.e.,~the GCM $A$ is of affine type).
\end{itemize}
For simplicity, we say that an affine Kac-Moody group $\GGk$ is {\it of type $\XX_N^{(r)}$}
if $\g$ is affine of type $\XX_N^{(r)}$,
where $\XX_N^{(r)}$ is one of Kac's list \cite[TABLE~Aff~$r$]{Kac}.
\end{definition}

In the following, we let $\GGk$ be a Kac-Moody group over $\K$ associated to an admissible pair $(V,V_\mathbb{Z})$.
For $\alpha\in\Deltar$ and $\tau\in\K^\times:=\K\setminus\{0\}$, we put
\begin{equation} \label{eq:w,h}
w_\alpha(\tau) := x_\alpha(\tau) x_{-\alpha}(-\tau^{-1}) x_\alpha(\tau)
\quad \text{and} \quad
h_\alpha(\tau) := w_\alpha(\tau) w_\alpha(-1).
\end{equation}
One easily sees that $w_\alpha(-\tau)=w_\alpha(\tau)^{-1}$ and $h_\alpha(\tau)^{-1}=h_\alpha(\tau^{-1})$.

Let $\TTk$ be the subgroup of $\GGk$ generated by $\{h_\alpha(\tau)\mid \alpha\in\Deltar, \tau\in\K^\times \}$.
The group $\TTk$ is an abelian group and is generated by the $h_{\alpha_i}(\tau)$'s.
We let $\UUk^\pm$ be the subgroup of $\GGk$ generated by $\{x_\alpha(\nu) \mid \alpha\in\Deltar_\pm, \nu\in\K\}$,
and let $\BBk^\pm$ be the subgroup of $\GGk$ generated by $\TTk$ and $\UUk^\pm$.
\begin{theorem} \label{prp:Bruhat}
The group $\GGk$ admits a Bruhat decomposition, that is,
\[
\GGk=\bigsqcup_{w\in\WW} \BBk^\pm \, w \, \BBk^\pm \qquad (\text{disjoint union}).
\]
\end{theorem}
\begin{proof}
We only show the claim for $\BBk^+$, since the proof is similar for $\BBk^-$.
Let $\NNk$ be the subgroup of $\GGk$ generated by $\{w_\alpha(\tau)\mid \alpha\in\Deltar, \tau\in\K^\times \}$.
It follows that $\TTk$ is a normal subgroup of $\NNk$,
since $w_\alpha(\theta)h_\beta(\tau)w_\alpha(-\theta)=h_{s_\alpha(\beta)}(\tau)$ for $\alpha,\beta\in\Delta$ and $\tau,\theta\in\K^\times$.
The map $\WW\to \NNk/\TTk$ defined by $s_i\mapsto w_{\alpha_i}(1)\TTk$ is an isomorphism (cf.~\cite[Lemma~22]{Ste}),
and hence the subset $\SSS:=\{w_{\alpha_i}(1) \TTk \mid \alpha_i\in\Pi \}$ of $\NNk/\TTk$ generates $\WW$.

First, we show that $\TTk$ coincides with $\BBk^+\cap \NNk$.
Since $\TTk\subset \BBk^+\cap \NNk$ is trivial, we take and fix $x\in \BBk^+\cap \NNk$ and put $w=x\TTk$.
For each $\lambda\in \Lambda(V_\K)$ and $v_\lambda\in (V_\K)_\lambda$, we write down the image of $v_\lambda$ under $x$ as
\[
x(v_\lambda) = \tau v_\lambda +\sum_{\beta\in Q_+}v_{\mu+\beta} = v_{w(\mu)}',
\]
where $\tau\in\K^\times$, $v_\mu+\beta\in (V_\K)_{\mu+\beta}$, and $v_{w(\mu)}'\in (V_\K)_{w(\mu)}$.
Then by the assumption, we get $x(v_\lambda)=\tau v_\lambda=v_{w(\mu)}'$, $\sum_{\beta\in\Lambda^+}v_{\mu+\beta}=0$, and $w(\mu)=\mu$.
For each $\alpha_i\in\Pi$, there exists $\mu_i\in\Lambda(V_\K)$ such that $0\neq e_i((V_\K)_{\mu_i}) \subset (V_\K)_{\alpha_i+\mu_i}$.
In particular, we see $w(\mu_i)=\mu_i$ and $w(\alpha_i+\mu_i)=\alpha_i+\mu_i$.
This means that $w(\alpha_i)=\alpha_i$ for all $\alpha_i\in\Pi$, and hence $w=1$.
We conclude that $x\in \TTk$.
Finally, the same argument as in \cite[Section~6.3]{MP} shows that $s\BBk w \subset \BBk w \BBk \cup \BBk sw\BBk$ and $s\BBk s \not\subset \BBk$
for all $s\in \SSS$ and $w\in \WW$.
These results show that the quadruple $(\GGk,\BBk^+,\NNk,\SSS)$ forms a {\it Tits system} for $\GGk$
(cf.~\cite{Moody-Teo,Tit3}).
\end{proof}

Let $Z(\GGk)$ denote the center of the Kac-Moody group $\GGk$, regarding as an abstract group.
Using the Bruhat decomposition above, we can prove the following result in the same way as \cite[Chapter~3]{Ste}.
\begin{proposition} \label{prp:cent_KM}
The center $Z(\GGk)$ of $\GGk$ lies in $\TTk$.
Moreover, $Z(\GGk)$ is explicitly given as
\[
Z(\GGk)=\{\prod_{i\in I}h_{\alpha_i}(\tau_i)\in\TTk \mid
\prod_{i\in I}\tau_i^{a_{ij}}=1 \,\,\text{ for all }\,j\in I\},
\]
and is isomorphic to $\Hom(\Xi(V)/\Xi(\g), \K^\times)$.
\end{proposition}

Let $\FDeltaa$ denote the free group with free generating set $\{\hat{x}_\alpha(\nu) \mid \alpha\in\Deltar,\,\nu\in\K\}$.
As \eqref{eq:w,h}, we put
\[
\hat{w}_\alpha(\tau):=\hat{x}_\alpha(\tau)\hat{x}_{-\alpha}(-\tau^{-1})\hat{x}_\alpha(\tau)
\quad \text{and} \quad
\hat{h}_\alpha(\tau):=\hat{w}_\alpha(\tau)\hat{w}_\alpha(-1)
\]
for $\alpha\in\Deltar$ and $\tau\in\K^\times$.
We define the following relations in $\FDeltaa$.
For $\alpha,\beta\in\Deltar$, $\nu,\mu\in\K$, $\tau,\theta\in\K^\times$,
\begin{description}
\item[(SC-A)] $\hat{x}_\alpha(\nu)\hat{x}_\alpha(\mu) = \hat{x}_\alpha(\nu+\mu)$,
\item[(SC-B)] $[\hat{x}_\alpha(\nu),\,\hat{x}_\beta(\mu)] =
\underset{i\alpha+j\beta\in Q_{\alpha,\beta}}{\prod} \hat{x}_{i\alpha+j\beta}(c^{i,j}_{\alpha,\beta}\, \nu^i\mu^j)$,
\item[(SC-B${}'$)] $\hat{w}_\alpha(\tau)\hat{x}_\beta(\theta)\hat{w}_\alpha(\tau)^{-1} =
\hat{x}_{s_\alpha(\beta)}(\eta_{\alpha,\beta}\,\nu\tau^{-\beta(H_\alpha)})$,
\item[(SC-C)] $\hat{h}_\alpha(\tau)\hat{h}_\alpha(\theta) = \hat{h}_\alpha(\tau\theta)$.
\end{description}
Here, $c^{i,j}_{\alpha,\beta}$ and $\eta_{\alpha,\beta}$ are defined in \eqref{eq:cij} and \eqref{eq:eta}, respectively.
By \cite{Tit} (see also \cite[Chapter~6]{Ste}), we have the following result.
\begin{proposition} \label{prp:str-const}
The Kac-Moody group of simply-connected type is isomorphic to the quotient group of $\FDeltaa$
by the normal subgroup generated by the relations (SC-A), (SC-B), (SC-B${}'$), and (SC-C).
\end{proposition}

\newcommand{\bibfont}[1]{{\it#1}}

\end{document}